\documentclass[11pt]{article}
\usepackage{amsmath}
\usepackage{graphicx} 
\usepackage{mathrsfs}
\usepackage{indentfirst}
\usepackage{enumerate}
\usepackage{cite}
\usepackage{comment}
\usepackage{color}
\usepackage{amsmath}
\usepackage{cases}
\usepackage{amsthm}
\usepackage{amssymb}
\numberwithin{equation}{section}
\usepackage{mathtools}
\usepackage{stmaryrd}
\theoremstyle{definition}
\newtheorem*{thm*}{Theorem}
\newtheorem{thm}{Theorem}[section]
\newtheorem{prop}[thm]{Proposition}
\newtheorem{lem}[thm]{Lemma}

\newtheorem{remark}[thm]{Remark}
\theoremstyle{plain}
\usepackage{tikz}
\usetikzlibrary{intersections, calc, angles, arrows.meta}
\usetikzlibrary{patterns}
\newcommand{\hquad}{\hspace{2mm}}
\newcommand{\al}{\alpha}
\newcommand{\gam}{\gamma}
\newcommand{\del}{\delta}
\newcommand{\ep}{\varepsilon}
\newcommand{\e}{\varepsilon}
\newcommand{\kap}{\kappa}
\newcommand{\lam}{\lambda}
\newcommand{\sig}{\sigma}
\newcommand{\om}{\omega}
\renewcommand{\phi}{\varphi}
\newcommand{\Del}{\Delta}
\newcommand{\Gam}{\Gamma}

\newcommand{\N}{\mathbb{N}}

\newcommand{\R}{\mathbb{R}}

\newcommand{\upsi}{\underline{\psi}}
\newcommand{\uPsi}{\underline{\Psi}}

\newcommand{\cL}{\mathcal{L}}

\renewcommand{\hat}{\widehat}
\renewcommand{\Tilde}{\widetilde}

\newcommand{\pl}{\partial}
\newcommand{\norm}[1]{\left\|#1\right\|}
\begin{document}
\title{Large-time behavior and grow-up rates of\\inhomogeneous semilinear heat equations}
\author{Kenta Kumagai and Yusuke Oka}
\date{\today}

\maketitle

\begin{abstract}
We consider the semilinear heat equation in the unit ball with the exponential nonlinearity and an inhomogeneous term $f$. When $f=0$, it is known that the bifurcation structure of the stationary problem undergoes a qualitative change at the critical dimension $N=10$. This change affects the large-time behavior of solutions to the heat equation, and in particular, the grow-up phenomenon occurs for $N\ge 10$.

In this paper, we show that once $f$ exceeds a threshold, the bifurcation structure changes to a type that does not appear in the case $f=0$. 
The change in the bifurcation structure leads to the disappearance of the grow-up phenomenon beyond the threshold. Moreover, we provide a quantitative characterization of this transition
by determining the sharp grow-up rates for $N\ge 11$. In particular, we identify a new dimension-specific phenomenon in the threshold case: a log-log type correction term emerges in the grow-up rate only for $N=11$.
\end{abstract}
\noindent Addresses:

\smallskip
\noindent
K.~K.:  Graduate School of Mathematical Sciences, The University of Tokyo,\\
3-8-1 Komaba, Meguro-ku, Tokyo 153-8914, Japan.

\smallskip
\noindent
E-mail: {\tt kumagai-kenta@g.ecc.u-tokyo.ac.jp}\\

\smallskip
\noindent
Y.~O.:  Mathematical Institute and Graduate School of Science, Tohoku University,\\ 6-3, Aramaki Aza-Aoba, Aoba-ku, Sendai 980-8578, Japan.

\smallskip
\noindent
E-mail: {\tt yusuke.oka.b6@tohoku.ac.jp}\\



\vspace{3pt}

\vspace{3pt}

\maketitle

\section{Introduction}
We consider the following semilinear heat equation in the unit ball $B_1\subset\R^N$:
\begin{align}\label{eq-intro-1}
  \begin{cases}
    \partial_t u-\Delta u=\lambda e^u-f(r)
    &\text{in } B_{1}\times (0,T),
    \quad u=0 \, \text{ on }\pl B_{1}\times (0,T),
    \\
    u(x,0)=u_0(x) &\text{in } B_1
  \end{cases}
\end{align}
and its stationary counterpart
\begin{align}\label{eq-intro-2}
-\Delta v = \lambda e^v - f(r) \quad \text{in $B_1$,}\quad v=0 \quad \text{on }\partial B_1,
\end{align}
where $N\geq 3$, $\lambda>0$, $r:=|x|$, $f\in \mathrm{Lip}[0,1]$ and
$T>0$ denotes the maximal existence time.
Throughout this section,
we assume that 
\begin{equation}
\label{asu0}
  u_{0}\in C^{0}(\overline{B_1})\quad \text{and}\quad 
\phi_0 \le u_{0}\quad \text{in $B_1$},
\end{equation}
where $\phi_0$ is the unique solution of
\begin{equation}
\label{Defphi0}
    -\Delta \phi_0 = -f(r) \quad \text{in $B_1$,}\quad \phi_0=0 \quad \text{on $\partial B_1$}.
\end{equation}
We remark that $\phi_0=0$ when $f=0$. The condition \eqref{asu0} is not restrictive, since every classical stationary solution satisfies \eqref{asu0}.

The aim of this paper is to clarify the bifurcation structure of \eqref{eq-intro-2} and
thereby deriving the large-time behavior of solutions to \eqref{eq-intro-1}.
More precisely, we investigate whether solutions blow up
in finite time, remain globally bounded, or exhibit grow-up behavior, which means that $\limsup_{t\to \infty}\norm{u}_{L^{\infty}(B_1)}=\infty$.
Moreover, in the grow-up case, we obtain the \textit{sharp}
grow-up rate. We first recall known results for the classical case.
\subsection{Classical case ($\boldsymbol{f=0}$)
}
For the problem \eqref{eq-intro-2}, every solution is radially symmetric by the symmetry result of \cite{Gidas}.
Moreover, 
the following properties are well-known (see \cite{BCMR, Kor, JL, Dup}).
\begin{enumerate}
    \item[(i)] \emph{Global solution curve.}  
    The set of classical solutions forms an unbounded curve emanating from $(0,0)$ and 
    described by $\{(\lambda(\alpha),v(r,\alpha)); \alpha>0\}$, where $v(r,\alpha)$ is a solution satisfying $\alpha=v(0)$. Moreover, no (even weak) solution exists if $\lambda>\lambda^{*}$, where  $\lambda^{*}:=\sup_{\alpha>0} \lambda(\alpha)\in (0,\infty)$.
    \item[(ii)] \emph{Stable branch.} There exists a unique solution \(v^*\in H^1_0(B_1)\) for \(\lambda=\lambda^*\). The set of stable solutions consists of the branch of classical solutions extending from \((0,0)\) to the limiting solution 
\((\lambda^*,v^*)\), together with the limiting solution \((\lambda^*,v^*)\). This branch is monotone in $\alpha$. In particular, the stable solution is unique for each $\lambda\in (0,\lambda^*]$. We denote it by $v_{\lambda}$.
\end{enumerate}
We call the set $\{(\lambda(\alpha),\alpha);\alpha>0\}$ the \textit{bifurcation curve}
(see the figure on page \pageref{fig:bifur})
and we say that a solution $v$ to \eqref{eq-intro-2} is \textit{stable} if
\begin{equation*}
Q_{v}(\xi):= \int_{B_1}|\nabla \xi|^2 - \lambda e^{v}\xi^2\,dx\ge 0 \quad \text{for all $\xi \in C^{1}_{0}(B_1)$.}
\end{equation*}
Joseph and Lundgren \cite{JL} showed that the bifurcation curve converges to the radial singular solution $(\lambda_{*}, V_{*})=(2N-4, -2\log r)$ as $\alpha\to\infty$. Here, a radial singular solution $(\lambda_*,V_*)$ means a solution
$V_*\in C^2(0,1]$ of \eqref{eq-intro-2} satisfying $\lim_{r\to 0}V_{*}(r)=\infty$.
In addition, the authors \cite{JL} showed that the bifurcation curve exhibits the following two types depending on $N$.

\medskip
\noindent\textbf{Type I:}\quad
The curve bends back at $\lambda=\lambda^{*}$,
then turns infinitely many times around $\lambda=\lambda_{*}$ and eventually approaches $\lambda=\lambda_{*}$.
In particular, $v^{*}\in C^2(B_1)$.

\noindent\textbf{Type II:}\quad
The curve coincides with the stable branch and thus $\lambda(\alpha) $ monotonically converges to $\lambda^{*}$ as $\alpha\to \infty$.
In particular, $(\lambda_{*}, V_{*})=(\lambda^{*}, v^{*})$.

\begin{figure}[ht]
    \centering
    \includegraphics[width=0.8\linewidth]{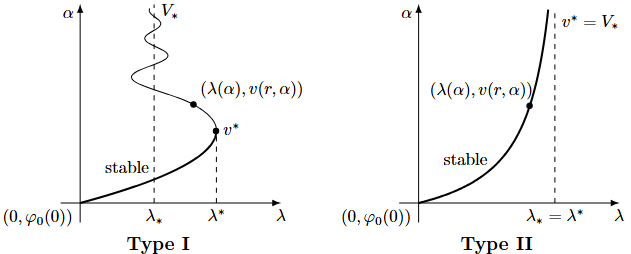}
    \label{fig:bifur}
\end{figure}

\noindent
More precisely, they showed that the bifurcation diagram is of Type I if
$3\leq N\le 9$ and of Type II if $N\ge 10$.

The stability of singular solutions plays a key role in the bifurcation structure.
Indeed,
Brezis and V\'{a}zquez \cite{BV} showed that the singular stable solution $V\in H^{1}_{0}(B_1)$ exists if and only if the bifurcation is of Type II,
for all non-negative non-decreasing and convex nonlinearities. For the exponential nonlinearity, they further showed that $V_*$ is stable if and only if $N\ge 10$, which gives an alternative explanation of the result of \cite{JL}.
Motivated by this connection, Miyamoto \cite{M14} conjectured that the number of turning points is equal to $m(V_{*})$ for general supercritical nonlinearities. Here, $m(V_{*})$ is defined as the maximal dimension of a subspace $X\subset H^{1}_{0,\mathrm{rad}}(B_1)$ such that $Q_{V_{*}}(\xi)<0$ for all
$\xi\in X\setminus \{0\}$. Note that $m(v)=0$ when $v$ is stable.
Following these results,
many studies have investigated the bifurcation structure, asymptotic behavior, and uniqueness of singular solutions for more general nonlinearities. We refer to \cite{BN87,DF07, GhGo, GW11, K25, KM26, KiWe, Lin, MP91, M14, M15, M18, MN20, MN23, MN24}.

\medskip

The bifurcation structure is deeply connected to the large-time behavior of solutions $u$ to \eqref{eq-intro-1}.
Indeed, it is known \cite{Fu1969, BCMR, PV1995} that the stable solution of \eqref{eq-intro-2}
is the large-time limit of the solution to \eqref{eq-intro-1} whenever $u_0$ is below the stable solution. In particular, in the case $\lambda=\lambda^*$, the large-time behavior of $u$ changes depending on $N$ as follows. 
\begin{enumerate}
    \item[{(i)}] If $\lambda<\lambda^{*}$ and $u_{0}\le v_{\lambda}$ in $B_1$, the solution $u$ of \eqref{eq-intro-1} exists globally and satisfies $u\to v_{\lambda}$ in $C^2(\overline{B_1})$ as $t\to\infty$.
    \item [{(ii)}] If $\lambda^{*}<\lambda$ and $u_{0}\in C^{0}(\overline{B_1})$, the solution blows up in finite time.
    \item[{(iii)}] If $\lambda=\lambda^{*}$ and $u_{0}\le v^{*}$ in $B_1$, then $u$ is global and
    \begin{enumerate}
        \item $u\to v^{*}$ in $C^{2}(\overline{B_1})$ as $t\to\infty$
        when $N\le 9$;
        \item $u\to v^{*}=V_{*}$ in $L^2(B_1)\cap C^{2}_{\mathrm{loc}}(\overline{B_1}\setminus \{0\})$ when $N\ge 10$. In particular, $\lVert u\rVert_{L^{\infty}(B_1)}\to\infty$ as $t\to\infty$.
    \end{enumerate}
\end{enumerate}
The case $u_0\not\le v_{\lambda}$ (or $u_0\not\le v^{*}$) is more delicate.
For $\lambda<\lambda^{*}$, the stable solution $v_\lambda$ does not separate globally bounded from finite-time blow-up. Instead, when they are smooth, unstable solutions and $v^{*}$ serve as thresholds separating globally bounded from finite-time blow-up. Moreover, $V_{*}$ is a threshold of the existence and non-existence of solutions to \eqref{eq-intro-1} in some sense. For a precise statement,
see Theorem \ref{p-thm-1}. For related works in this direction, we refer to \cite{HM25, LT1987, Wan, Tel, FK25, I11, GV97, Mizo05, V99, Martel98, QS25} and the references therein.

\medskip
In the grow-up case, Dold, Galaktionov, Lacey and V\'{a}zquez \cite{DGLV1998} determined the sharp 
grow-up rate $\lVert u\rVert_{L^{\infty}(B_1)}=H \gamma^{-1}t + O(1)$ for $N\ge 11$. Here, $H$ and $\gamma$ are defined later. For $N=10$, Galaktionov and King \cite{GK} formally derived the grow-up rate.
We refer to the related works on the grow-up rate \cite{CZ2022, Mizo06, PY03, FKWY06, FKWY07-ADE, FKWY07, FWY}.

The classical results motivate us to investigate the effect of the inhomogeneous term on the bifurcation structure of \eqref{eq-intro-2} and the large-time behavior of solutions to \eqref{eq-intro-1}. In the following subsection, we first clarify the bifurcation structure of \eqref{eq-intro-2}.
\subsection{Inhomogeneous stationary problem}
For the inhomogeneous case, radial symmetry of solutions is no longer guaranteed in general. Therefore, we restrict our attention to radial solutions. By using a standard change of variables (see \cite{KumagaiJDE, Korman14}), we obtain the following parameterization result
\begin{prop}
\label{intro-prop-1}
Assume that $N\ge 3$ and $f\in \mathrm{Lip}[0,1]$.
Then, the following hold.
\begin{enumerate}
    \item[(i)] The set of radial solutions of \eqref{eq-intro-2} forms an unbounded analytic curve emanating from $(0,\phi_0)$ and is described by 
    \begin{equation}
    \label{intro-parameter}
    \text{$\{(\lambda(\beta), v(r,\alpha(\beta))); \beta\in \mathbb{R}\}$ \quad with $\alpha(\beta):=v(0,\beta)=\beta-\log \lambda(\beta)$.}    
    \end{equation}
    Moreover, \eqref{eq-intro-2} has no (even non-radial or weak) solution for $\lambda>\lambda^{*}$, where 
\begin{equation*}
\lambda^{*}:=\sup\{\lambda>0; \text{ \eqref{eq-intro-2} admits a radial solution} \}\in (0,\infty).
\end{equation*}
\item[(ii)] There exists a unique solution $v^{*}\in H^{1}_{0}(B_1)$ with $\lambda=\lambda^{*}$. The set of stable solutions of \eqref{eq-intro-2} consists of the branch of radial classical solutions extending from \((0,\phi_0)\)  
to the limiting solution 
\((\lambda^*,v^*)\), together with the limiting solution \((\lambda^*,v^*)\). In particular, the stable solution is unique for each $\lambda\in (0,\lambda^*]$, which is denoted by $v_{\lambda}$. Moreover, $v_{\lambda}(r)<v_{\lambda'}(r)$ in $B_1$
for $\lambda<\lambda'<\lambda^{*}$.
\end{enumerate}
\end{prop}
We call the curve $\{(\lambda(\beta), \alpha(\beta)); \beta\in \mathbb{R}\}$ the bifurcation curve of the problem \eqref{eq-intro-2}. We say that $(\lambda(\beta), \alpha(\beta))$ is a turning point if $\lambda(\beta)$ is a local minimum or local maximum.

We now aim to clarify the bifurcation structure. 
Following the idea of \cite{KumagaiJDE}, we explicitly construct a one-parameter family of singular solutions $(\lambda_h, V_{h})$ with a parameterized forcing term $f_h$. The explicit form of the singular solutions enables us to study the Morse index of $V_h$, which suggests the bifurcation structure.  
\begin{prop}
\label{intro-prop-2}
Assume that $N\ge 3$ and
$f\in \mathrm{Lip}[0,1]$. Then, \eqref{eq-intro-2} admits a unique radial singular solution $(\lambda_{*}, V_{*})$ 
such that $\lambda(\beta)\to\lambda_{*}$ and $v(r,\alpha(\beta))\to V_{*}$ in $C^{2}_{\mathrm{loc}}(0,1]\cap H^{1}_{0}(B_1)$ as $\beta\to\infty$. In addition, 
$V_{*}=-2\log r+\log (2N-4)-\log \lambda_{*}+o(1)$ as $r\to 0$.
For the particular inhomogeneous term 
\begin{equation}
\label{Deffh}
    f_{h}(r)=h\left(8(N-2)d^{2}(r)+2(N-2)d(r)+1\right), \quad \text{where}  \quad d(r)=\frac{1}{2N-4+hr^2}
\end{equation}
with $h>-(2N-4)$,
the singular solution is explicitly given by
\begin{equation*}
(\lambda_h, V_h)=
\bigg(
    2N-4+h,
    -2\log r-\log d(r)-\log\left[2N-4+h\right]
\bigg).
\end{equation*}
Moreover, $m(V_h)=\infty$ if $3\le N\le 9$; $m(V_h)=0$ and $V_h$ is stable if $N\ge 10$, $h\le H$; and $1\le m(V_h)<\infty$ if $N\ge 10$, $h>H$, where 
\begin{equation}
\label{DefH}
H:= \text{the first eigenvalue of }
\begin{cases}
   \text{$-\Delta_D$ in $B_{1}^{2}\subset \mathbb{R}^2$}& \text{if $N=10$} \\
  \text{$-\Delta_D-\frac{2N-4}{|x|^2}$ in $B_{1}^{N}\subset \mathbb{R}^N$}  &  \text{if $N\ge 11$}\\
\end{cases}
\quad >0.
\end{equation}
\end{prop}
Proposition \ref{intro-prop-2} and the conjecture of Miyamoto \cite{M14} suggest that,
at the threshold $h=H$, the
bifurcation structure changes from Type II to Type III when $N\ge 10$. Here, we say that the bifurcation diagram is of Type III if the curve bends back at $\lambda=\lambda^{*}$ and then exhibits only finitely many further turning points (possibly none).
In fact, we prove the following classification result.
\begin{thm}
\label{intro-thm-1}
Assume that $N\ge 3$ and
$f\in \mathrm{Lip}[0,1]$. Then, the bifurcation diagram of \eqref{eq-intro-2}
is of
\begin{enumerate}
    \item[(a)] Type I if $N\le 9$;
    \item[(b)] Type II if $N\ge 10$ and $f\le f_H$ in  $B_1$;
    \item[(c)] Type I or Type III if $N\ge 10$, $f\ge f_H$ in $B_1$ and $f\not\equiv f_H$.
\end{enumerate}
In particular, $v^{*}\in C^2(\overline{B_1})$ 
in cases (a), (c), and $(\lambda^{*},v^{*})=(\lambda_{*}, V_{*})$ in case (b). Moreover, in case (c), the bifurcation diagram is of Type III if $N=10,11$ or $f\equiv f_h$ with some $h>H$.
\end{thm}
\begin{remark}\,
\rm{
\begin{enumerate}
    \item[(i)] Type III bifurcations were first established in \cite{M14} for specific nonlinearities
and later confirmed in \cite{KumagaiJDE}
for the weighted problem with $N=10$.
    \item[(ii)] In case (c), the proof of Type III bifurcations requires the
non-degeneracy condition \eqref{non-degenerate}. Lemma \ref{non-degenerate-lem} shows that this condition holds
when $N=10, 11$, or $f\equiv f_h$ with some $h>H$.
Moreover, \eqref{non-degenerate} holds for all but at most countably many $\nu$
along any analytic family $\{f^\nu\}_{\nu\in \R}$ with $f^{0}=f_H$.
\end{enumerate}
}
\end{remark}
\subsection{Inhomogeneous parabolic problem}
The change in the bifurcation structure induced by the inhomogeneous term $f$, established in Theorem \ref{intro-thm-1},
leads to qualitative changes in the asymptotic behavior of solutions to \eqref{eq-intro-1}.
In fact, we prove the following.
\begin{thm}
\label{intro-thm-2}
Assume that $N\ge 3$, $f\in \mathrm{Lip}[0,1]$ and \eqref{asu0}. Then,
\begin{enumerate}
\item[{(i)}] If $\lambda<\lambda^{*}$ and $u_{0}\le v_{\lambda}$ in $B_1$, there exists a unique global solution $u$ of \eqref{eq-intro-1} such that $u\to v_{\lambda}$ in $C^2(\overline{B_1})$ as $t\to\infty$.
    \item [{(ii)}] If $\lambda>\lambda^{*}$, there exists a unique local solution of \eqref{eq-intro-1} that blows up in finite time.
    \item[{(iii)}] If $\lambda=\lambda^{*}$ and $u_{0}\le v^{*}$ in $B_1$, there exists a unique global solution $u$ of \eqref{eq-intro-1}. Moreover,
    \begin{enumerate}
        \item $u\to v^{*}$ in $C^2(\overline{B_1})$ as $t\to\infty$ when $N\le 9$;
        \item $u\to v^{*}=V_{*}$ in $L^2(B_1)\cap C^{2}_{\mathrm{loc}}(\overline{B_1} \setminus \{0\})$ when $N\ge 10$ and $f\le f_H$ in $B_1$. In particular, $\lVert u\rVert_{L^{\infty}(B_1)}\to\infty$ as $t\to\infty$;
        \item $u\to v^{*}$ in $C^2(\overline{B_1})$ as $t\to\infty$ when $N\ge 10$, $f\ge f_H$ in $B_1$ and $f\not\equiv f_H$.
        \end{enumerate}
    \end{enumerate}
\end{thm}
\begin{remark}
\rm{For the case $u_0\not\le v_{\lambda}$ (or $u_0\not \le v^{*}$), we also show in Theorem \ref{p-thm-1} that when they are smooth, unstable solutions $v$ and $v^{*}$ are thresholds separating convergence to $v_{\lambda}$ or $v^{*}$ from finite-time blow-up. Moreover, $V_{*}$ is a threshold of the existence/non-existence of solutions to \eqref{eq-intro-1} in some sense.
}
\end{remark}
Theorem \ref{intro-thm-2} shows that the grow-up phenomenon disappears once $f$ exceeds the threshold $f_H$. This phenomenon can be understood qualitatively through the change in the stability of $V_{*}$. In the following theorem, we give a quantitative understanding by determining the \textit{sharp} grow-up rate in the case $N\ge 11$.
\begin{thm}
\label{intro-thm-3}
Assume that $N\geq 11$, $\lambda=\lambda^{*}$,
$f\in \mathrm{Lip}[0,1]$, \eqref{asu0} and $u_{0}\leq v^{*}$ in $B_1$. Let $\gamma:=\frac{1}{2}\left(N-2-\sqrt{(N-2)(N-10)}\right)$.
Then, the solution $u$ of \eqref{eq-intro-1} satisfies
\begin{enumerate}
    \item[(i)] $\lVert u\rVert_{L^{\infty}}=\frac{2\mu_1}{\gamma}t+O(1)$ if $0\le f\le f_H$ in $B_1$ and $f\not\equiv f_{H}$;
    \item[(ii)] $\lVert u\rVert_{L^{\infty}}=\frac{2}{\gamma}\log t+O(1)$ if $f\equiv f_H$ and $N\ge 12$;
\item[(iii)] $\lVert u\rVert_{L^{\infty}}=\frac{2}{\gamma}(\log t+\log \log t) +O(1)$ if $f\equiv f_H$ and $N=11$
\end{enumerate}
as $t\to\infty$, where $\mu_1$ is the first eigenvalue of $\cL=-\Delta - \lambda_{*}e^{V_{*}}$.
\end{thm}
\begin{remark}
\rm{For the case $N\ge 10$, $\cL$ has discrete spectrum, and its eigenfunctions form an orthonormal basis of $L^2(B_{1})$ (see Proposition \ref{sepa-pro-1} and Lemma \ref{b-v-lem}). 
The properties of $\cL$ are as follows. 
\begin{enumerate}
    \item[(i)]  $\mu_1(f_1)\ge \mu_1(f_2)>0$ if $f_1\le f_2\le f_H$ in $B_1$ and $f_2\not\equiv f_H$ (see Proposition \ref{sepa-pro-1}).
    \item[(ii)] The map $\mathrm{Lip}[0,1]\ni f\mapsto \mu_1$ is continuous and $\mu_1\downarrow 0$ as $f\uparrow f_H$ (see Lemma \ref{fnolem}). 
    \item[(iii)] $\mu_1=H-h$ for the case $f\equiv f_H$ (see Lemma \ref{b-v-lem}).
\end{enumerate}
}
\end{remark}
Theorem \ref{intro-thm-3} contains the result of \cite{DGLV1998} as a special case $f=0$. Moreover, this theorem provides a quantitative description of the disappearance of grow-up phenomenon. Indeed, the leading-order coefficient $2\mu_1 \gamma^{-1}$ tends to zero as $f\uparrow f_H$. In addition, at the threshold $f\equiv f_H$, the grow-up rate changes from linear to logarithmic. 
Finally, we emphasize that when $f\equiv f_H$,  an additional log-log type second-order term appears only in the case $N=11$. The exceptional grow-up behavior is driven by the singularity of the first eigenfunction $\phi_1$ of $\cL$. 
\begin{remark}
\rm{As mentioned above, the grow-up rate was formally computed in \cite{GK} in the case $N=10$ and $f=0$, while its rigorous verification remained open. In a forthcoming work by the first author, the grow-up rates will be characterized in the case $N=10$ and $f\le f_H$ in $B_1$. The result provides the rigorous verification of the computation in \cite{GK}.}
\end{remark}

\subsection{Mechanism of the change in grow-up rates}
The change in the grow-up rates is governed by the asymptotic behavior of $u$ in the outer region. Moreover, the following equation plays a key role in understanding the outer behavior of $u$:
\begin{equation*}
  \Phi_t +\cL \Phi=-\lambda_{*}e^{V_{*}} F(\Phi)\quad \text{in $B_1\times (0,\infty)$,} \quad \Phi=0 \quad \text{on $\partial B_1\times (0,\infty)$},
\end{equation*}
where $F(u)=e^{-u_{+}}-1+u_{+}$ and $\Phi:=V_{*}-u$.
When $f=0$, the authors \cite{DGLV1998} showed that the leading-order behavior of $\Phi$ in the outer region is governed by the linear equation $(\pl_{t}+\cL)w=0$. The same mechanism persists whenever $f\not\equiv f_H$ (i.e., $\mu_1>0$). As a result, 
we can prove that $\Phi\simeq_{\ep} e^{-\mu_1 t}\phi_1$ in $B_{1}\setminus B_{\ep}$ for any fixed $\ep>0$ by a modification of the argument in \cite{DGLV1998}.
In contrast, when $f\equiv f_H$ (i.e., $\mu_1=0$), the effect of the nonlinear term $\lambda_{*}e^{V_{*}} F(\Phi)$ becomes apparent in the leading-order dynamics of $\Phi$. To understand the nonlinear effect,
we formally assume that the leading-order term of $\Phi$ in the outer-region is governed by its projection $c(t)\phi_1$ onto the first eigenspace of $\cL$.
Then a formal computation yields
\begin{equation*}
c'(t)\simeq
-\left(\lambda_{*}e^{V_{*}}F(c(t)\phi_1), \phi_1\right)
_{L^2(B_1)}\simeq
\begin{cases}
-c(t)^2& \text{when }\,\,N\ge 12,\\
-c(t)^2\log c(t)& \text{when }\,\,N=11
\end{cases}
\end{equation*}
(see Section \ref{sec-matching} for details).
We emphasize that $\phi_1\simeq r^{-\gamma}$ near $r=0$ (see Lemma \ref{b-v-lem}) and this singularity generates the logarithmic term only for $N=11$. It follows from the above computation that $c(t)\simeq t^{-1}$ for $N\ge 12$ and $c(t)\simeq (t\log t)^{-1}$ for $N=11$. The different decay rates of $\Phi$ yield the distinct grow-up rates obtained in Theorem \ref{intro-thm-3}.

The main difficulty in the proof of Theorem \ref{intro-thm-3} is the rigorous justification of the above formal computation. Our idea is to construct specific super/sub-solutions in the form $c(t)\phi_1+\xi$, where $\xi\in\mathrm{span}\left\{\phi_{1}\right\}^{\perp}$ is a sufficiently small correction term in $L^2(B_1)$. The correction term is obtained by applying
$\cL^{-1}$ to the orthogonal component of suitable approximations of the nonlinear term with respect to $\phi_1$. 
This construction enables us to capture the nonlinear effects that are not captured by the linearized dynamics.

The novelty of this paper lies in revealing the disappearance of the grow-up phenomenon induced by the inhomogeneous term, providing its quantitative characterization, and identifying a new phenomenon specific to the dimension $N=11$.

This paper is organized as follows. In Section 2, we classify the bifurcation structure of \eqref{eq-intro-2}
and obtain the properties of $\cL$. We also prove Proposition \ref{inner-prop-1}, which plays a key role in determining the grow-up rate. In Section 3, we clarify the large-time behavior of solutions of \eqref{eq-intro-1}. In Section 4, we determine the grow-up rates for $N\ge 11$. 
\section{Bifurcation structure of the elliptic equation}
In this section,
we study the bifurcation structure of \eqref{eq-intro-2}. 
We start by introducing a specific change of variables used in \cite{KumagaiJDE, Korman14}.
Let $(\lambda, v)$ be a radial classical solution of \eqref{eq-intro-2} with $v(0)=\alpha$ and $(\lambda_{*}, V_{*})$ be a radial singular solution of \eqref{eq-intro-2}.
We define
\begin{equation}
\label{sp-transform}
w:=v+\log \lambda, \quad W:=V_{*}+\log \lambda_{*}, \quad \beta:=\alpha+\log \lambda.
\end{equation}
Then, $w$ and $W$ satisfy $w(1)=\log \lam$,
$W(1)=\log \lambda_{\ast}$ and
\begin{align}
\label{eqofw}
    w''+\frac{N-1}{r}w'+e^w-f(r)=0& \hspace{2mm}\text{in $(0,1]$}, \hspace{2mm}  w(0)=\beta\in\R, \hspace{2mm} w'(0)=0;
    \\
\label{eqofW}
    W''+\frac{N-1}{r}W'+e^W-f(r)=0& \hspace{2mm}\text{in $(0,1]$},
    \hspace{2mm} W\in C^2(0,1],\hspace{2mm}\lim_{r\to 0}W(r)=\infty.
\end{align}
Note that \eqref{eqofw} has the unique solution $w=w(r,\beta)\in C^2[0,1]$ for each $\beta\in \mathbb{R}$. Thus, the set of radial classical solutions of \eqref{eq-intro-2} is described by
\begin{equation}\label{parameter}
    \left\{(\lambda(\beta), v(r,\alpha(\beta)))=(e^{w(1,\beta)}, w(r,\beta)-w(1,\beta)): \beta\in\R\right\}.
\end{equation}
\subsection{Properties of the bifurcation curve}
In this subsection,
we study the properties of the bifurcation curve and its stable branch for \eqref{eq-intro-2}. As a result, we prove
Proposition \ref{intro-prop-1}.
We first introduce an apriori estimate for radial (possibly singular) solutions.
Let $w$ be any solution of either \eqref{eqofw} or \eqref{eqofW}.
We define $r_{w}$ as
\begin{align*}
r_{w}:=
\begin{cases}
  \sup\left\{r>0:w(s)>2\norm{f}_{L^{\infty}}\,\,\text{for all }s<r\right\}
  &\text{ if }\,\,w(0)>2\lVert f\lVert_{L^{\infty}},
  \\
  0 &\text{ if }\,\,w(0)\le 2\lVert f\lVert_{L^{\infty}}.
\end{cases}
\end{align*}
Then, we obtain the following.
\begin{lem}
\label{apriorilem1}
Assume that $N\ge 3$ and $f\in \mathrm{Lip}[0,1]$.
Let $w\in C^2(0,1]$ be a solution of either \eqref{eqofw} or \eqref{eqofW}.
Then, we have
\begin{equation}
\label{apriorieq-1}
w\le -2\log r +\log 4N \quad\text{and}\quad  0\le -w'\le \frac{6N}{N-2}r^{-1} \quad \text{for any $r\in (0,r_w)$;}
\end{equation}
\begin{equation}
\label{apriorieq-2}
 -C(\lVert f\rVert_{L^{\infty}})(1+|w'(r_{w})|)<w<C\lVert f\rVert_{L^{\infty}}, \quad |w'| \le C(\lVert f\rVert_{L^{\infty}})(1+|w'(r_{w})|)
\end{equation}
for any $r\in [r_w, 1]$, where $C>0$ is depending only on $N$. In addition, it follows that
\begin{equation}
\label{integralform-1}
    -r^{N-1}w'(r)=\int_{0}^{r}s^{N-1}(e^{w(s)}-f(s))\,ds \quad\text{for all $r\in [0,1]$.}
\end{equation}
\end{lem} 
\begin{proof}
Without loss of generality, we assume that $r_{w}>0$. we first show that $w'\leq 0$ in $(0,r_w)$.
We suppose to the contrary that 
$w'(r_{0})>0$ for some $r_{0}\in (0,r_w)$. Since
\begin{equation}
\label{hojo-apriori-1}
(-r^{N-1}w')'= r^{N-1}(e^{w} - f(r))\ge 0 \quad \text{for $r\in (0, r_w)$,} 
\end{equation}
we have $s^{N-1}w'(s)\ge r_{0}^{N-1}w'(r_{0})>0$ for $0<s<r_{0}$.
From this inequality, we get
\begin{equation*}
    w(s)\le w(r_{0})-r_{0}^{N-1}w'(r_{0})\frac{s^{-(N-2)}-r_{0}^{-(N-2)}}{N-2}\to -\infty\quad\text{as $s\to 0$,}
\end{equation*}
which is a contradiction.
Thus,
we deduce that $w'\leq 0$ in $(0,r_{w})$. Therefore,
for all $r\in\left[0,r_{w}\right]$
it follows from  \eqref{eqofw} that
\begin{equation*}
    -r^{N-1}w'(r)\ge \int_{0}^{r}s^{N-1} (e^{w(s)}-f(s))\,ds\ge \frac{r^{N}}{N}(e^{w(r)}-\lVert f\rVert_{L^{\infty}})\ge \frac{r^N}{2N}e^{w}.\quad
\end{equation*}
 It implies that $(e^{-w})'\ge r/2N$. By integrating this inequality on $(\rho,r)$ and then by letting $\rho\to 0$, we have $w\le -2\log r +\log 4N$.
 By \eqref{hojo-apriori-1},
 we have $-r^{N-1}w'(r)\to c$ for some $c\in[0,\infty)$.
 Then we deduce $c=0$ from the fact $w\le -2\log r+\log 4N$.
 Therefore, by integrating \eqref{hojo-apriori-1},
 we have \eqref{integralform-1} for $r\in (0,r_{w})$.
 Again, by using $w\le -2\log r+\log 4N$ and \eqref{integralform-1},
 we have
\begin{equation*}
    -w'=r^{1-N}\int_{0}^{r}s^{N-1}(e^w(s)-f(s))\,ds\le \frac{6N}{N-2}r^{-1} \quad\text{on $(0,r_w]$.}
\end{equation*}
Therefore, we obtain  \eqref{apriorieq-1}.

Now, since \eqref{integralform-1} holds for $(0,r_{w})$, a standard ODE argument yields \eqref{integralform-1} on $(0,1)$.
Moreover, by definitions of $r_{w}$ and
\begin{equation*}
   w'\le -r^{1-N}\int_{0}^{r}s^{N-1}(e^w(s)-f(s))\,ds
   \leq
   \frac{1}{N}
   \norm{f}_{L^{\infty}}r,
\end{equation*}
we have $w(r)=w(r)-w(r_{w})+w(r_{w})\le C\norm{f}_{L^{\infty}}$ for $r_{w}<r<1$.
As a result, we get
\begin{equation*}
|w'(r)|\le {\frac{r_{w}^{N-1}}{r^{N-1}}}|w'(r_{w})|
+\left|\int_{r_{w}}^{r}\frac{s^{N-1}}{r^{N-1}}(e^w-f(s))\,ds\right|
\le C\left(\norm{f}_{L^{\infty}}\right)(1+|w'(r_{w})|).
\end{equation*}
Therefore, we obtain
\eqref{apriorieq-2}.
\end{proof}
Next, we introduce the property of the stable-branch. 
\begin{prop}[see \cite{BV, BCMR, Dup}]\label{stableprop-1}
Assume that $N\ge 3$ and $f\in \mathrm{Lip}[0,1]$. Then,
there exists $\lambda^{*}\in (0,\infty)$ such that
\begin{enumerate}
    \item[(i)] For any $\lambda<\lambda^{*}$, \eqref{eq-intro-2} has a unique stable solution 
    $v_{\lambda}\in C^{2}(B_1)$. It follows that $v_{\lambda}$ is radially symmetric and $v_{\lambda_1}(r)<v_{\lambda_2}(r)$ in $B_1$ for $\lambda_1<\lambda_2<\lambda^{*}$. 
    \item[(ii)] When $\lambda=\lambda^{*}$, there exists a unique solution $v^{*}\in H^{1}_{0}(B_1)$ such that $v_{\lambda}\uparrow v^{*}$ in $B_1$ as $\lambda\uparrow\lambda^{*}$. Moreover, there exists no weak solution if $\lambda>\lambda^{*}$.
    \item[(iii)] For each $\lambda<\lambda^*$,
    the stable solution $v_{\lambda}$ is minimal. More precisely, if $v$ is any unstable solution of \eqref{eq-intro-2}, then $v>v_{\lambda}$ in $B_1$.
    \item[(iv)] The map $\lambda\mapsto v_{\lambda}(0)$ is continuous and  $v_{\lambda}\to \phi_0$ in $C^2(\overline{B_1})$ as $\lambda\to 0$.
\end{enumerate}
\end{prop}
\begin{proof}[Proof of Proposition \ref{intro-prop-1}]
By \eqref{sp-transform} and \eqref{parameter}, we obtain the parameterization result \eqref{intro-parameter}. The analyticity of the curve is a consequence of
the analyticity of the exponential function. The unboundedness of the curve follows from Lemma \ref{apriorilem1}.
The remaining assertions follow from Proposition \ref{stableprop-1}.  
\end{proof}
\subsection{Singular solution}
In this subsection,
following the idea of \cite{MN20},
we investigate the properties of a solution of \eqref{eqofW}
to prove Proposition \ref{intro-prop-2}.
We first show the uniqueness and asymptotic behavior of the singular solution.
\begin{lem}
\label{uniquelem}
Equation \eqref{eqofW} admits at most one solution $W$.
If it exist, $W$ satisfies
\begin{equation}
\label{as-eq-1}
W(r)=-2\log r+\log 2(N-2)+o(1), \quad \text{as $r\to 0$}.
\end{equation}
\end{lem}
\begin{proof}
We introduce the following Emden-Fowler type transformation
\begin{equation}
\label{transform-emden}
    \Tilde{W}=W-2s-\log 2(N-2), \quad s=-\log r.
\end{equation}
Then, $\Tilde{W}$ is a solution of
\begin{equation*}
    \frac{d^2 \Tilde{W}}{ds^2}-(N-2)\frac{d\Tilde{W}}{ds}+2(N-2)(e^{\Tilde{W}}-1)-e^{-2s}f(e^{-s})=0.
\end{equation*}
It follows from Lemma \ref{apriorilem1} that $\limsup_{s\to \infty}\Tilde{W}\le C$.
Moreover,
we have $\limsup_{s\to\infty} \Tilde{W}>-\infty$ by a similar argument to that in the proof in \cite[Lemma 2.4]{MN20}.
In addition, we have 
\begin{equation*}
   \frac{e^{-2s}f(e^{-s})}{e^{\Tilde{W}}}=(2N-4)e^{-W}f(e^{-s})\to 0 \quad \text{as $s\to\infty$.}
\end{equation*}
Thanks to the upper and lower bounds,
we can deduce from \cite[Lemma 3.2]{MN20} that $\Tilde{W}=o(1)$ as $s\to\infty$,
which means \eqref{as-eq-1}.
Finally, for any singular solutions $W_{1}$ and $W_{2}$,
$\eta=\Tilde{W}_1 - \Tilde{W}_2$ satisfies
\begin{equation*}
 \frac{d^2 \eta}{ds^2}-(N-2)\frac{d\eta}{ds}+2(N-2)\frac{e^{\Tilde{W}_1}-e^{\Tilde{W}_2}}{\Tilde{W}_1 -\Tilde{W}_2}\eta=0.
\end{equation*}
Since $\Tilde{W}_1, \Tilde{W}_2 =o(1)$ as $s\to\infty$, we obtain $W_1= W_2$ by using \cite[Lemma 4.2]{LLD}. 
\end{proof}
Then, we show the existence of the singular solution and the convergence of the bifurcation curve to it. We introduce the following Pohozaev-type identity.
\begin{lem}\label{lem:poho}
Assume that $N\ge 3$ and $f\in \mathrm{Lip}[0,1]$.
Let $w$ be a solution of \eqref{eqofw}. Then, for any $\nu>0$, we have
\begin{align}
\label{poho}
\begin{split}
    \frac{d}{dr}&\left\{r^N\left(\frac{1}{2}(w')^2+ e^{w}-f(r)w+\frac{\nu}{r}ww' \right) \right\}+r^{N}f'(r)w\\
    &=r^{N-1}\left\{N e^w-Nf(r)w-\nu w(e^w-f(r))+\left(\nu+1-\frac{N}{2}\right)\left(w'\right)^2\right\}
    \end{split}
\end{align}
for a.e. $r\in [0,1]$.
\end{lem}
Lemma \ref{lem:poho} is shown by a direct computation.
We remark that this identity is proved in \cite{serrin} for the case $f=0$. The identity \eqref{poho} allows us to obtain the following inequality,
which plays a key role in proving the convergence result.

\begin{lem}
\label{needle}
Assume that $N\ge 3$ and $f\in \mathrm{Lip}[0,1]$.
Let $w$ be a solution of \eqref{eqofw}. Then, there exists $q>1+2\norm{f}_{L^{\infty}}$ so that the following hold. 
\begin{enumerate}
    \item[(i)] If $w$ satisfies $w(r)>q$ on $[0,r_{0}]$ for some $r_{0}>0$, then 
    \begin{equation}
\label{grad}
0\le -rw'\le \frac{N-2}{2} w+4\lVert f'\rVert_{L^{\infty}} r \quad \text{for all $r\in\left(0,r_{0}\right)$}.
\end{equation}
\item[(ii)] Take any $\rho\ge q$ and define $r_{\rho}$ as
\begin{equation*}
  r_{\rho}:=\min\left\{
    \sqrt{
      \frac{N\rho}{2(e^{4\rho}+\norm{f}_{L^{\infty}})}
    },
    \frac{(N-2)\rho}{16
    \norm{f'}_{L^{\infty}}
    }
  \right\}
  \quad\text{for }\rho\geq q.
\end{equation*}
Then, we have $w(r,\beta)>\rho$ for any $r\in\left(0,r_{\rho}\right)$ provided that $\beta>4\rho$.
\end{enumerate}
\end{lem}
\begin{proof}
We fix $\nu=\frac{N-2}{4}$. Then, we can take $q>1+2\lVert f \rVert_{L^{\infty}}$ so that the right-hand side of \eqref{poho} is non-positive on $[0,r_{0}]$ for all solution $w$ satisfying $w>q$ on $[0,r_{0}]$. Then,
we can confirm that $e^w\ge f(r)w$ in $(0,r_0)$.

Let $r<r_0$. We integrate \eqref{poho} on $\left[0,r\right]$ to yield
\begin{equation*}
\frac{r^N}{2}\left(w'\right)^2+ r^N e^w-r^Nf(r)w+\nu r^{N-1}ww'\le \norm{f'}_{L^{\infty}} \int_{0}^{r}s^N w(s)\,ds
\end{equation*}  
for all $r\in\left[0,r_{0}\right]$.
From \eqref{integralform-1} and the definition of $r_w$, we deduce that
\begin{equation*}
    \int_{0}^{r}s^N w(s)\,ds
    \le r\int_{0}^{r} s^{N-1} e^w\,ds\le 
    2r\int_{0}^{r}s^{N-1}(e^{w(s)}-f(s))\,ds
    =
    -2r^{N}w'(r).
\end{equation*}
These estimates lead \eqref{grad}.
The remaining assertion can be proved by a similar argument to that in the proof of \cite[Lemma 5.1]{MN20}. 
\end{proof}
As a result, we obtain the following property.
\begin{prop}
\label{uniqueprop}
Assume that $N\ge 3$ and $f\in \mathrm{Lip}[0,1]$.
Then, \eqref{eqofW} admits the unique solution $W$.
Moreover, $W$ satisfies \eqref{as-eq-1} and $w(r,\beta)\to W$ in $C^{2}_{\mathrm{loc}}(0,1]$ as $\beta\to\infty$.
\end{prop}
\begin{proof}
By Lemma \ref{needle}, we obtain that $w(r,\beta)>q$ for any $r<r_{q}$ and $\beta>4q$, which implies that 
$0<r_{q}<r_{w}$ for all $\beta>4q$.
Thus, Lemma \ref{apriorilem1}, the elliptic regularity theory (see \cite{GTBook}) and a diagonal argument yield
the existence of a sequence $\beta_{n}$ and a radial solution $W\in C^{2}\left(0,1\right]$ of
$-\Delta W=e^{W}-f$ in $B_1$,
satisfying $w(r,\beta_n)\to W$ in $C^{2}_{\mathrm{loc}}(0,1]$ and $\beta_{n}\to\infty$ as $n\to \infty$.
Again by Lemma \ref{needle},
we deduce that $W\to \infty$ as $r\to 0$ and thus
$W$ is a solution of \eqref{eqofW}.
Therefore, the result follows from Lemma \ref{uniquelem}.
\end{proof}
\begin{remark}
\label{monotone-rem}
By virtue of Lemma \ref{apriorilem1} and the proof of Proposition \ref{uniqueprop},
we verify that there exists $\beta_0 >0$ such that
$w(r,\beta)<\beta$ in $(0,1]$ for $\beta>\beta_0$.
\end{remark}
Next, we compute the Morse index of the explicit singular solution.

\begin{prop}
\label{specificprop}
Let $N\ge 3$
Set $f=f_{h}$ as in Proposition \ref{intro-prop-2}.
Then, the solution $W=W_{h}$ of \eqref{eqofW} is represented by
\begin{equation*}
W_h(r)=-2\log r+\log\left(1+\frac{h}{2N-4}r^{2}\right)+\log (2N-4).
\end{equation*}
In addition, $W_h$ is stable if and only if $N\ge 10$ and $h\le H$, and
\begin{align}\label{specific-morse-ind}
m(W_h)=\infty\quad\text{if $N\le 9$;}
\quad
\begin{cases}
  \text{$m(W_h)=0$}& \text{if }\,\,N\geq 10,\,h\leq H,\\
    1\le m(W_h)<\infty & \text{if }\,\,N\ge 10,\,h>H.
\end{cases}
\end{align}
\end{prop}
In order to prove this proposition, we introduce two lemmata.
\begin{lem}
\label{b-v-lem}
Let $N\ge 10$ and $W$ be the solution of \eqref{eqofW}.
Assume that $f\in \mathrm{Lip}[0,1]$ and
\begin{equation}
\label{defk}
K(r):=e^{W}-\frac{2N-4}{r^2}\in C^{0}[0,1].
\end{equation}
We define $\mathcal{L}:=-\Delta- \lambda_{*}e^{V_{*}}
=-\Del-(2N-4)r^{-2}-K(r)$. Then, the following hold.
\begin{enumerate}[{\rm (i)}]
\item The operator $\cL$ has a countable sequence of discrete eigenvalues, and the corresponding eigenfunctions form an orthonormal basis of $L^2(B_1)$. Moreover, the first eigenvalue is simple.
    \item The restriction of $\cL$ to $L^{2}_{\mathrm{rad}}(B_1)$ has an orthonormal basis of eigenfunctions $\{\phi_k\}_{k\in \N}$ with corresponding eigenvalues $\{\mu_k\}_{k\in \N}$. 
    Each eigenspace is one-dimensional. Thus,
    after the normalization $\phi_k(r)>0$ near $r=0$, $\phi_k$ is uniquely determined. Moreover, $\mu_1$ and $\phi_1$
    coincide with the principal eigenvalue and a corresponding eigenfunction of $\cL$ in $L^2 (B_1)$, respectively. Furthermore, $\phi_1>0$ in $B_1$.
\item Let $m=N-2\gamma$, where $\gamma$ is that in Theorem \eqref{intro-thm-3}. We define
    \begin{equation*}
   \psi_k(r):=\left(\frac{N\omega_N}{m\omega_m}\right)^{1/2}r^{\gamma}\phi_k,
    \end{equation*}
    where $\omega_{m}$ is the volume of the $m$-dimensional unit ball.
    Then, $\psi=\psi_k$
    is an orthonormal basis in $L^{2}_{\mathrm{rad}}(B_{1}^m)$ 
    for the following eigenvalue problems
    \begin{equation}
    \label{b-v-eq-1} 
    -\Delta\psi - K(r) \psi=\mu\psi \quad \text{in $B_{1}^m$} \quad \psi=0 \quad \text{on $\partial B_{1}^m $}
    \end{equation}
    with $\mu=\mu_k$.
 Moreover, $\psi_k\in C^2(\overline{B_{1}^{m}})$ and $\psi_1>0$, $\psi_1'\le 0$ in $B_{1}^m$. For the definition of non-integer dimensions, see Appendix~\ref{sec-app-m}.

\item We have $\cL_{h}:=-\Delta-\lambda_{h}e^{V_{h}}=
-\Del -e^{W_{h}}=-\Delta-{(2N-4)}{r^{-2}}-h$. Thus, in the case $f\equiv f_H$, we have $\mu_1=H-h$.

\item It follows that
\begin{equation}
\label{stW}
    Q_{W}(\xi):=\int_{B_1} |\nabla \xi|^2-e^{W}\xi^{2}\,dx \ge \mu_1 \int_{B_1}\xi^2\,dx, \quad\text{for any $\xi\in C^{1}_{0}(B_1)$}.
\end{equation}
Moreover, $Q_{W_h}(\xi)\ge (Hr^{-2}-h)\lVert \xi\rVert_{L^2(B_1)}^2$ for all $\xi\in C^{1}_{0}(B_r)$ and $r>0$. 
\end{enumerate}
\end{lem}
\begin{proof}
The assertions (i) and (ii) can be proved by arguments similar to those in Subsections 3.1 and 4.1 of \cite{VZ}, together with a direct computation. Assertion (iii) follows from a direct computation and a standard ODE argument. Estimate \eqref{stW} is an immediate consequence of the Rayleigh quotient.
In addition, since $K=h$ when $f\equiv f_H$, the first eigenvalue of $\mathcal{L}$ in $H^{1}_{0}(B_r)$ coincides with $Hr^{-2}-h$. Thus, we get the assertion (iv).
\end{proof}
\begin{remark}
\rm{The condition \eqref{defk} is satisfied for the case $f\equiv f_H$. In Proposition~\ref{sepa-pro-1} below,
we show that \eqref{defk} is satisfied for any $f\in \mathrm{Lip}[0,1]$ when $N\ge 10$. 
}
\end{remark}

\begin{lem}
\label{intersec-lem}
Assume that $f\in \mathrm{Lip}[0,1]$ and $3\le N\le 9$. Then, there exists a sequence $r_i\downarrow 0$ such that the solution $W$ of \eqref{eqofW} is unstable in $B_{r_{i}}\setminus B_{r_{i+1}}$ for any $i\in \mathbb{N}$. In particular, $m(W)=\infty$.  
\end{lem}
\begin{proof}
The result follows from Lemma \ref{uniquelem} and Lemma \ref{optimality-1}.
\end{proof}

\begin{proof}[Proof of Proposition \ref{specificprop}]
Thanks to Proposition \ref{uniqueprop}, Lemma \ref{intersec-lem} and direct computations, it is sufficient to obtain \eqref{specific-morse-ind} and the stability of $W_h$
in the case $N\ge 10$ and $h\le H$. In this case, we can confirm the condition \eqref{defk}.  
Thanks to Lemma \ref{b-v-lem}, we can verify that $m(W)$ is equal to the number of negative radial eigenvalues of $\mathcal{L}$ by using an argument similar to that of the proof in \cite[Proposition 1.5.1]{Dup}. Therefore, the result follows from Lemma \ref{b-v-lem}. 
\end{proof}

\begin{proof}[Proof of Proposition \ref{intro-prop-2}]
The result follows from Propositions \ref{uniqueprop}, \ref{specificprop} and the change of variables \eqref{sp-transform}. 
\end{proof}
\subsection{Separation results}
In this subsection, we introduce separation results that play a key role in clarifying the bifurcation structure and determining the grow-up rate of solutions to \eqref{eq-intro-1}. 
\begin{prop}\label{sepa-pro-1}
Let $N\ge 10$ and $f\in \mathrm{Lip}[0,1]$. Then, 
\begin{enumerate}
    \item[(i)]\label{sepa-pro-1-1} there exist $\underline{r}$ and $\underline{h}>0$ depending only on $f$ such that $w(r,\beta_1)<w(r,\beta_2)<W_{\underline{h}}$ in $B_{\underline{r}}$ for any $\beta_1<\beta_2$. In particular, $W\le W_{\underline{h}}$ in $B_{\underline{r}}$.
    Moreover, when $f\le f_H$, we can take $\underline{r}=1$ and $\underline{h}=H$.
    \item[(ii)] \label{sepa-pro-2}
    $e^{W}-e^{W_{h_{0}}}\to 0$ as $r\to 0$, where $h_{0}=(N-2)f(0)/(2N-2)$. 
    \item[(iii)]\label{sepa-pro-3} When $f_1\le f_2\le f_H$ and $f_2\not\equiv f_H$ in $B_1$, we have 
    $W_{f_1}(r)\le W_{f_2}(r)\le W_H$ in $B_1$ and $\mu_1(f_1)\ge \mu_1(f_2)>0$. Moreover, $W_{f}\to W_{H}$ in $C^{0}(\overline{B_1})$ and $\mu_1(f)\to 0$ as $f\to f_H$ in $C^{0}(\overline{B_1})$.
\end{enumerate}
\end{prop}
\begin{remark}
    \rm{Since $e^{W_h}=h+(2N-4)r^{-2}$,
    Proposition \ref{sepa-pro-1} (ii) ensures that the assumption \eqref{defk} is always satisfied.
    Hence, Lemma \ref{b-v-lem} gives standard properties for the eigenvalues of $\mathcal{L}$.}
\end{remark}
\begin{proof}
We show (i).
Since $f_{h}'\leq 0$ in $(0,1]$ and $f_h(1)\geq h$,
we can choose $\underline{h}$ such that $f\le f_{\underline{h}}$ in $B_1$.
By Lemma \ref{b-v-lem}~(v),
there exists $\underline{r}$ such that $W_{\underline{h}}$ is stable in $B_{\underline{r}}$.
Suppose, to the contrary, that the first zero $r_{0}$ of 
$\eta=W_{\underline{h}}-w(r,\beta_2)$ satisfies $r_{0}<\underline{r}$.
Then, we have
\begin{equation*}
    -\Delta\eta
    =e^{W_{\underline{h}}}-e^{w(r,\beta_2)}
    -(f_{\underline{h}}-f)\le e^{W_{\underline{h}}}\eta
    \quad\text{in $B_{r_0}$}.
\end{equation*}
Hence, it follows from the strict-stability\footnote{We say that $W$ is strictly-stable in $B_{r}$ if the first eigenvalue of $-\Delta -e^{W}$ is positive.
It follows from Lemma \ref{b-v-lem} that $W_{\underline{h}}$ is strictly-stable in $B_s$ for any $s<r$ if $W_{\underline{h}}$ is stable in $B_r$.}
of $W_{\underline{h}}$ in $B_{r_0}$ that
\begin{equation*}
    \int_{B_{r_{0}}}|\nabla \eta|^2\,dx\le \int_{B_{r_0}}e^{W_{\underline{h}}}\eta^2\,dx< \int_{B_{r_{0}}}|\nabla \eta|^2\,dx,
\end{equation*}
which is a contradiction.
Thus we get $w(r,\beta_2)< W_{\underline{h}}$ in $B_{\underline{r}}$.
In particular, $w(r,\beta_2)$ is stable in $B_{\underline{r}}$.
By the same argument, we obtain $w(r,\beta_1)<w(r,\beta_2)$ in $B_{\underline{r}}$. Thus, we have
$W\leq W_{\underline{h}}$ in $B_{\underline{r}}$ by Proposition \ref{uniqueprop}.
The remaining assertion follows from the fact that $W_{H}$ is stable in $B_1$. 

Next we prove (ii). It suffices to prove the following: for any $\ep>0$, there exists $r_{0}$ such that $W_{h_{0}-\e}\le W \le W_{h_0 +\ep}$ in $B_{r_0}$. By the assertion (i) and the fact that $f(0)=f_{h_{0}}(0)$, we can choose $r_{0}$ such that $f_{h_{0}-\ep}\le f \le f_{h_{0}+\ep}$ in $B_{r_{0}}$ and
$W_{h_{0}+\ep}$ is stable in $B_{r_{0}}$. Therefore, the result follows as in the proof of (i).

Finally, we prove (iii). By the assertion (i), for any $f$ satisfying $f\le f_H$ in $B_1$, we deduce that
$W\le W_H$ in $B_1$ and thus $W$ is stable in $B_1$. Hence, the monotonicity of $W_{f}$ with respect to $f$ and the convergence to $W_H$ follows as in the proof of (i).
Moreover, thanks to (ii), we can check \eqref{defk}. Lemma \ref{b-v-lem} and the monotonicity of $W$ and the convergence to $W_H$ enable us to see the monotonicity of $\mu_1$ and the convergence to $0$.
The positivity of $\mu_1$ follows from the fact that $W\not\equiv W_H$.
\end{proof} 
From the separation result, we obtain the following.
\begin{prop}
\label{bifurcation-prop-1}
Let $N\ge 10$ and $f\in \mathrm{Lip}[0,1]$. Then,
the bifurcation diagram is of Type II if $N\ge 10$ and $f\le f_H$ in $B_1$. In addition, $v^{*}\in C^2(\overline{B_1})$ if $N\ge 10$, $f\ge f_H$ in $B_1$
and $f\not\equiv f_H$. Finally, $v^{*}\in C^2(\overline{B_1})$ if $N\le 9$.
\end{prop}
\begin{proof}
We first consider the case $N\ge 10$ and $f\le f_H$.
By Proposition \ref{sepa-pro-1} (iii), we have $W\le W_{H}$ and hence $W$ is stable. Hence, by Proposition \ref{intro-prop-1} and \eqref{sp-transform}, we obtain $(\lambda^{*}, v^{*})=(\lambda_{*},V_{*})$. In particular, the bifurcation diagram is of Type II.
Next let $N\ge 10$, $f\ge f_H$ in $B_1$
and $f\not \equiv f_H$. Suppose that $v^{*}$ is singular.
By \eqref{sp-transform},
the stability of $v^{*}$ and Proposition \ref{uniqueprop},
we verify that $W$ is stable in $B_1$.
Arguing as in the proof Proposition \ref{sepa-pro-1},
we have $W\ge W_H$ in $B_1$\footnote{
Here the strict stability of $W$ in $B_r$ for all $r<1$
is necessary.
For the proof, we note that $W$ is stable in $B_r$.
If $W$ is not strictly stable in $B_{r}$,
we can take the first eigenfunction $\phi$ of $\mathcal{L}$ in $H^{1}_{0}(B_{r})$.
we extend $\phi$ to $B_1$ so that $\phi=0$ on $B_1\setminus B_{r}$.
Then, $\phi$ is also a first eigenfunction of $\mathcal{L}$ in $B_1$,
which contradicts the fact that any first eigenfunction has no zero in $B_1$.
}. 
Here, we recall that $\mu_{1}=0$ for $f\equiv f_H$ (see Lemma \ref{b-v-lem}). Therefore, from the stability of $W$ and the fact that $W\ge W_H$ in $B_1$, we obtain $\mu_{1}=0$ and $W=W_H$,
which contradicts the fact that $f\not\equiv f_H$.
Finally, we consider the case $3\le N\le 9$.
By Lemma \ref{intersec-lem}, we deduce that $W$ is unstable, which implies $v^{*}\in C^{2}(\overline{B_1})$.
\end{proof}

The following is crucial for obtaining the grow-up rate of solutions to \eqref{eq-intro-1}.
\begin{prop}\label{inner-prop-1}
Suppose that $N\ge 10$ and $f\in \mathrm{Lip}[0,1]$ satisfies $0\le f\le f_{H}$ in $B_1$.
 Let $w_{0}(r,\beta)$ and $W_0(r)$ be solutions of \eqref{eqofw} and \eqref{eqofW} with $f=0$, respectively.
    Set $\hat{r}:=\min\{(2e)^{-1/2}, H^{-1/2}\}$.
  Then, 
    $W-w(r,\beta)\le W_{0}-w_{0}(r,\beta)$ in $B_{\hat{r}}$ for any $\beta>0$.
\end{prop}
In order to prove Proposition \ref{inner-prop-1},
we prepare the following.
\begin{lem}
\label{inner-lem-1}
Assume the hypothesis of Proposition \ref{inner-prop-1}.
 Let $k=k(r,\beta)$ be a solution of 
\begin{equation*}
-\Delta k=e^{k+w_{0}(r,\beta)}-e^{w_{0}(r,\beta)}-f(r)
\quad \text{in $B_{\hat{r}}$}\,,
\quad k(0)=k'(0)=0.
\end{equation*}
Then, we obtain
\begin{equation}
\label{kest1}
W-W_{0}\le k\le W-W_{0}+\frac{H}{2N}r^2\quad \text{in $B_{\hat{r}}$}.
\end{equation}
\end{lem}
\begin{proof}
We set $\xi:=k-W+W_0$.
Then, $\xi$ satisfies
\begin{equation}\label{xieq1}
-\Delta \xi
=e^{W-W_{0}
+w_{0}(r,\beta)}(e^{\xi}-1)
+(e^{W-W_{0}}-1)(e^{w_{0}(r,\beta)}-e^{W_{0}}) \quad \text{in $B_{\hat{r}}$}.
\end{equation}
Thanks to Proposition \ref{sepa-pro-1}, we have $\xi(0)=\xi'(0)=0$ 
and 
\begin{equation*}
    \xi''(0)=\lim_{r\to 0}\frac{1}{N}\Delta \xi(r)=\frac{h_{0}}{N}>0, \quad \text{where $h_{0}=\frac{(N-2)f(0)}{2N-2}$.}
\end{equation*}
Therefore, $\xi$ is positive if $r$ is sufficiently small.
Here, we define $r_0$ as the supremum of $r<\hat{r}$ so that $\xi(s)>0$ in $s\in (0,r)$.
Then, \eqref{xieq1} and Proposition \ref{sepa-pro-1} (iii) give
\begin{equation*}
    \frac{1}{r^{N-1}}\frac{d}{dr}\left(r^{N-1}\frac{d\xi}{dr}\right)=
    \Del\xi\leq (e^{W_{H}-W_{0}}-1)e^{W_{0}}= H \quad \text{in $B_{r_{0}}$}.
\end{equation*}
Integrating twice,
we obtain $\xi(r)\leq Hr^{2}/(2N)$ in $B_{r_{0}}$.
It remains to show that $r_{0}=\hat{r}$.
Suppose for contradiction that $r_0<\hat{r}$.
Since $\hat{r}\leq H^{-1/2}$, we have
\begin{equation*}
    e^{\xi}\le e^{\frac{H}{2N}r^2}\le 1+ \frac{Her^2}{2N} \quad \text{if $r<r_{0}$}.
\end{equation*} 
Thus, by combining \eqref{xieq1}, Proposition \ref{sepa-pro-1} and Lemma \ref{hardy}, we obtain 
\begin{align*}
&\norm{\nabla \xi}_{L^{2}(B_{r_0})}^2
\le \norm{e^{W-W_0+w_0(r,\beta)}(e^{\xi}-1)\xi}_{L^1(B_{r_0})} \le \norm{ e^{W}e^{\xi}\xi^2}_{L^1(B_{r_0})}\\
&\quad\le \norm{\frac{2N-4}{r^2}(1+Hr^2)\left(1+\frac{Her^2}{2N} \right)\xi^2}_{L^1}\le 
\norm{\frac{2N-4}{r^2}
\left(1+\frac{H(e+1)}{N}r^2\right)\xi^2}_{L^1(B_{r_0})}\\
&\quad<\norm{\nabla \xi}_{L^{2}(B_{r_0})}^2,
\end{align*}
which is a contradiction.
\end{proof}
\begin{proof}[Proof of Proposition \ref{inner-prop-1}]
The function $\xi:=w(r,\beta)-w_{0}(r,\beta)-k(r,\beta)$ satisfies
\begin{align*}
    -\Delta \xi=e^{w_0(r,\beta)+k(r,\beta)}(e^{\xi}-1) \quad \text{in $B_{\hat{r}}$},
    \qquad
    \xi(0)=\xi'(0)=0.
\end{align*}
Thus $\xi=0$ in $B_{\hat{r}}$ follows from the uniqueness of the solution to an ODE.
Finally, we use \eqref{kest1} to get
$W_{0}-w_{0}(r,\beta)-(W-w(r,\beta))=k(r,\beta)+W_{0}-W\ge 0.$
\end{proof}
\subsection{Bifurcation structure}
\begin{lem}
\label{extremal-lem}
Assume that $N\ge 3$, $f\in \mathrm{Lip}(B_1)$ and $v^{*}\in C^{2}(\overline{B_1})$. Then, the first eigenvalue of $-\Delta-\lambda^{*}e^{v^{*}}$ in $H^{1}_{0}(B_1)$ coincides with $0$. Moreover, the bifurcation curve turns at $\lambda=\lambda^{*}$.
\end{lem}
\begin{proof}
The first assertion follows from
Proposition \ref{stableprop-1} and the implicit function theorem. Hence, the result follows from \cite[Theorem 1.3]{Korbook}.
\end{proof}

\begin{prop}
\label{bifurcation-prop-2}
Assume that $3\le N\le 9$ and $f\in \mathrm{Lip}[0,1]$. Then, the bifurcation diagram is of Type I. 
\end{prop}
\begin{proof}
By Proposition \ref{uniqueprop}, we obtain $w(r,\beta)\to W$ in $C^{2}_{\mathrm{loc}}(0,1]$ as $\beta\to\infty$. Hence, thanks Lemmata \ref{intersec-lem} and \ref{intersec-lem-2}, we can deduce that the intersection number between $w(r,\beta)$ and $W$ in $(0,1]$ diverges as $\beta\to\infty$. Note that each intersection point is isolated by the uniqueness of the solution of ODE. Therefore, we deduce that $w(1,\beta)$ and $W(1)$ intersect infinitely many times, which implies the oscillation of the bifurcation curve around $\lambda=\lambda_{*}$ since $\lambda=e^{w(1,\beta)}$ and $\lambda_{*}=e^{W(1)}$ (see \eqref{parameter}). Therefore, the result follows from Lemma \ref{extremal-lem} and Proposition \ref{bifurcation-prop-1}.
\end{proof}
Next, we obtain the following proposition.
\begin{prop}
Assume that $N\ge 10$, $f\in \mathrm{Lip}[0,1]$, 
$f\ge f_H$ in $B_1$ and $f\not\equiv f_H$. Then, the bifurcation diagram is of Type I or Type III. If 
additionally the following condition
\begin{align}
\label{non-degenerate}
\text{$\mu_i=0$ for some $i\ge 1$ $\Longrightarrow$ $\int_{B_{1}}\lambda_* e^{V_*}\phi_{i}^3\,dx\neq 0$ }
\end{align}
is supposed, the bifurcation diagram is of Type III, where $\mu_i$, $\phi_i $ are those in Lemma \ref{b-v-lem}.
\end{prop}
\begin{proof}
By Proposition \ref{bifurcation-prop-1} and Lemma \ref{extremal-lem}, we can deduce that the bifurcation diagram is of Type I or Type III. Now, we show that the number of turning points is finite if \eqref{non-degenerate} is satisfied. Since $\lambda(\beta)$ is analytic in $(-\infty, \infty)$ 
and monotone in $(-\infty, v^{*}(0)+\log \lambda^{*}]$, it suffices to prove the non-degeneracy of $\lambda(\beta)$ for all $\beta$ sufficiently large. 
In order to prove it, we assume the contrary that there exists a sequence $\beta_k$ such that
$\dot{\lambda}(\beta_k)=0$ 
and $(-1)^{k}\ddot{\lambda}(\beta_k)\ge 0$, where $\,\dot{}\,$ is defined as the differentiation with respect to $\beta$. By differentiating the equation, we have
\begin{align*}
  \begin{cases}
    -\Delta \dot{w}= e^{w}\dot{w} \quad \text{in }B_1,
    \\
    \dot{w}(0)=\dot{w}'(0)=1, \hspace{2mm}
    \dot{w}(1)={\dot{\lambda}}/{\lambda},
  \end{cases}
  \hspace{2mm}
  \begin{cases}
    -\Delta \ddot{w}= e^{w}\ddot{w}+e^{w}\dot{w}^2 \quad \text{in }B_{1},
    \\
    \ddot{w}(0)=\ddot{w}'(0)=0, \hspace{2mm}
    \ddot{w}(1)={\left(\ddot{\lambda}\lambda-\dot{\lambda}^2\right)}/{\lambda^2}.
  \end{cases}
\end{align*}
Note that $\dot{w}(1)=0$ and $\ddot{w}(1)=\ddot{\lambda}\lambda^{-1}$ at $\beta=\beta_k$. Hence, it follows from the equations above and Green's identity that
\begin{equation}
\label{hojo-1}
N\omega_N\ddot{\lambda}\dot{w}'(1)\lambda^{-1}=\int_{B_1}e^w \dot{w}^3\,dx
\end{equation}
at $\beta=\beta_k$, where $\omega_N$ is the volume of $B_{1}$. Now, we define
\begin{equation*}
z_k:=\frac{r^{\gamma}\dot{w}(r,\beta_k)}{\lVert r^{\gamma}\dot{w}(r,\beta_k)\rVert_{L^2(B_{1}^{m})}}, \quad m=N-2\gamma.
\end{equation*}
Then, $z_k$ satisfies
\begin{equation*}
    -\Delta z_k=\left(e^{w(r, \beta_k)}-\frac{2N-4}{r^2}\right)z_k \quad \text{in $B_{1}^{m}$,} \quad z_{k}'(0)=z_k(1)=0, \quad \lVert z_{k}\rVert_{L^2}=1.
\end{equation*}
By Proposition \ref{sepa-pro-1}, we obtain $e^{w_k}\le (2N-4)r^{-2}+C$. Hence, by an energy estimate, we obtain $\lVert z_k\rVert_{H^{1}_{0}(B_{1}^{m})}<C$. Moreover, It follows from Proposition \ref{sepa-pro-1} (i)
that $z_k\ge 0$ for $r<\underline{r}$ with some $\underline{r}>0$ depending only on $f$. Hence, by using \cite[Proposition 47.5]{QSbook} and the elliptic regularity theory, we obtain $\lVert z_k \rVert_{L^{\infty}(B_1)}<C$ when $N\ge 11$. 
Moreover, by the $H^{1}$-bound and the elliptic regularity theorem, 
we obtain $z_k\to z$ in $C^2_{\mathrm{loc}}(0,1]$ and $z_{k}\rightharpoonup z$ in $H^{1}_{0,\mathrm{rad}}(B_{1}^{m})$ by taking a subsequence if necessary.
Hence, it follows from Proposition \ref{uniqueprop} that $z$ satisfies \eqref{b-v-eq-1} with $\mu=0$ and $\lVert z\rVert_{L^2}=1$. In particular, by the uniqueness of the solution, we deduce that $z_k\to z$ in $C^{2}_{\mathrm{loc}}(0,1]$
without taking a subsequence.
Moreover, thanks to Proposition \ref{sepa-pro-1} and Lemma \ref{b-v-lem}, we deduce that $\mu_i=0$ for some $i>0$ and $z=\psi_i$. Then, since $\lVert r^{\gamma}\dot{w}(r,\beta_k)\rVert_{L^2(B_{1}^{m})}\dot{w}'(1, \beta_k)^{-1}=z_{k}'(1)^{-1}$, it follows from \eqref{hojo-1} that
\begin{equation*}
\frac{N\omega_N \ddot{\lambda}(\beta_k)z_{k}'(1)}{\lambda(\beta_k) \lVert r^{\gamma}\dot{w}(r,\beta_k)\rVert_{L^2(B_{1}^{m})}^2}= \frac{N\omega_N}{m\omega_m} \int_{B_{1}^{m}}r^{-\gamma }e^{w}z_{k}^3\,dx.
\end{equation*}
By using the elliptic regularity theory, Fatou's Lemma, Proposition \ref{uniqueprop} and the fact that $z_k\ge 0$ for $r<\underline{r}$ and $\lVert z_k\rVert_{L^{\infty}}<C$ for $N\ge 11$, we have
\begin{align}
\label{sekibun-0}
\begin{split}
    &\lim_{k\to\infty}\int_{B_{1}^{m}}r^{-\gamma}e^{w}z_{k}^3\,dx=\lim_{k\to\infty}\int_{B_{\underline{r}}^m}r^{-\gamma}e^{w}z_{k}^3\,dx+\int_{B_{1}^{m}\setminus B_{\underline{r}}^{m}}r^{-\gamma}e^{W}z^3\,dx\\
    &=\left(\frac{N\omega_N}{m\omega_m}\right)^{\frac{1}{2}}\int_{B_{1}\setminus B_{\underline{r}}}e^{W}\phi^{3}_i\,dx+
    \begin{cases}
        \infty &\text{if $N=10, 11$}\\
        \left(\frac{N\omega_N}{m\omega_m}\right)^{\frac{1}{2}}\int_{B_{\underline{r}}}e^{W}\phi_{i}^3\,dx &\text{if $N\ge 12$}
    \end{cases}
    \\
    &=\left(\frac{N\omega_N}{m\omega_m}\right)^{\frac{1}{2}}\int_{B_1}e^{W}\phi_{i}^3\,dx.
   \end{split} 
    \end{align}
Note that $z_{k}'(1)\to \psi_k'(1)\neq 0$ as $k\to\infty$ and $(-1)^{k}\ddot{\lambda}(\beta_k)\ge 0$. It contradicts the non-degeneracy condition \eqref{non-degenerate}.
\end{proof}
In the next lemma, we obtain \eqref{non-degenerate} for some cases by using an analyticity argument. Note that this type of argument is also used in \cite{KM25}.
\begin{lem}
\label{non-degenerate-lem}
The non-degeneracy condition holds when $N=10,11$ or $f\equiv f_H$. Moreover, for any parametrized force term $f^{\nu}$ associated with $\nu\ge 0$, when $(-1,1)\ni \nu\mapsto f^{\nu}\in \mathrm{Lip}[0,1]$ is analytic and $f^{0}=f_H$, there exists a sequence $\{\nu_k\}_{k\in \mathbb{N}}$ so that the non-degeneracy condition holds for all $\nu\not\in \{\nu_k\}_{k\in \mathbb{N}}$.
\end{lem}
\begin{proof}
When $N=10$ or $N=11$, \eqref{non-degenerate} follows from \eqref{sekibun-0}. When $f\equiv f_H$ and $\mu_{i}=0$ for some $i\ge 1$, we obtain $\lambda_{*} e^{V_{*}}=(2N-4)r^{-2}+h$ by Proposition \ref{intro-prop-2}. From Lemma \ref{b-v-lem} and the Strum--Liouville theory, there exists $0=r_{0}<r_1<\cdots < r_{i}=1$ such that $z=r^{\gamma+(m-1)/2}\phi_{i}$ satisfies 
\begin{equation*}
z''= 
\left(\frac{(m-1)(m-3)}{4r^2}-h\right)z \quad \text{in $(0,1)$}, \hspace{2mm} (-1)^{i}z(r)>0 \hquad \text{on $r_{j}<r<r_{j+1}$} 
\end{equation*}
for all $0\le j\le i-1$. Since $m=N-2\gamma >3$ if $N\ge 11$, it follows that $z''/z$ is decreasing. Therefore, thanks to Lemma \ref{odelem}, we have
\begin{equation*}
    \int_{B_{1}} \lambda_{*}e^{V_{*}}\phi_{i}^3\,dx= N\omega_{N}\int_{0}^{1}((2N-4)r^{-\gamma-\frac{m+3}{2}}+hr^{-\gamma-\frac{m-1}{2}})z^3\,dr>0
\end{equation*}
Therefore, we obtain \eqref{non-degenerate}. Finally, we deal with the case
where $N\ge 12$ and the map $\nu\to f_{\nu}\in \mathrm{Lip} [0,1]$ is analytic.  
In this case, by using Emden-type transformation \eqref{transform-emden}, we obtain that the singular solution $W=W^{\nu}$ with $f=f^{\nu}$ transforms to the unique solution to the equation
\begin{equation*}
\frac{d^2 \Tilde{W}^{\nu}}{ds^2}-(N-2)\frac{d\Tilde{W}^{\nu}}{ds}+2(N-2)(e^{\Tilde{W}^{\nu}}-1)-e^{-2s}f^{\nu}(e^{-s})=0, \quad \Tilde{W}^{\nu}(\infty)=0.
\end{equation*}
As in the proof of Lemma \ref{uniquelem}, we can verify that the above equation admits the unique solution by using \cite[Lemma 4.2]{LLD}. Hence, we can confirm that $\Tilde{W}^{\nu}$ satisfies the following
$$
\Tilde{W}^{\nu}(t)=\frac{1}{N-2-2\gamma}\int_{t}^{\infty}(e^{(N-2-\gamma)(t-s)}-e^{\gamma(t-s)})Y(s)\,ds
$$
by using the fixed point theorem, where 
$$
Y(s)=(2N-4)(e^{\Tilde{W}^{\nu}(s)}-1-\Tilde{W}^{\nu}(s))-e^{-2s}f^{\nu}(e^{-s}).
$$
Since $\nu\mapsto f^{\nu}\in \mathrm{Lip}[0,1]$ is analytic, the map $\nu\mapsto \Tilde{W}^{\nu}\in C^{0}_{0}[0,\infty)$ is analytic by the implicit function theorem, where $C^{0}_{0}[0,\infty)$
is the set of $C^0[0,\infty)$ functions satisfying $u\to 0$ as $s\to\infty$. 

Moreover, it follows from Proposition \ref{sepa-pro-1} that $\hat{W}^{\nu}:=e^{2s}\Tilde{W}^{\nu}$ is a solution of 
\begin{equation*}
\frac{d^2 \hat{W}^{\nu}}{ds^2}-(N+2)\frac{d\hat{W}^{\nu}}{ds}+2N\hat{W}^{\nu}+(2N-4)\frac{e^{\Tilde{W}^{\nu}}-1}{\Tilde{W}^{\nu}}\hat{W}^{\nu}-f^{\nu}(e^{-s})=0, \hat{W}^{\nu}(\infty)=\frac{f^{\nu}(0)}{4(N-1)}.
\end{equation*}
Hence, we can deduce that the map $\nu\mapsto \Tilde{W}^{\nu}\in C^{0}_{0}[0,\infty)$ is analytic by a similar argument. Since $K(r)=(2N-4)(e^{W(s)}-1)W^{-1}(s)\hat{W}(s)$, we deduce that $K$ is analytic. Hence, each radial eigenvalue of $\mathcal{L}$ is analytic. As a result, $\mu_k$ coincides with $0$ at most countable times for any $k\ge 2$ because $f^{0}=f_H$. For the first eigenvalue, the above assertion or $\mu_1\equiv 0$ occur. Moreover, \eqref{non-degenerate} holds for the case $\mu_1\equiv 0$. Hence, the result follows. 
\end{proof}
Finally, we introduce the following.
\begin{lem}
\label{fnolem}
Assume that $N\ge 10$. Then, the maps
\[\mathrm{Lip}[0,1]\cap \{f\le f_H \} \ni f\mapsto \mu_1 \in [0,\infty), \hquad \mathrm{Lip}[0,1]\cap \{f\le f_H \} \ni f\mapsto  \psi_1(0) \in (0,\infty)
\]
are of class $C^1$. Moreover, for the case $0\le f\le f_H$ in $B_1$, we have
$\psi_{1}'\le 0$ in $B_1$.
\end{lem}
\begin{proof}
The $C^1$-regularity of the maps follows from an argument similar to that in the proof of Lemma \ref{non-degenerate-lem}. Next, we consider the case $0\le f\le f_H$ in $B_1$. In this case, by Proposition \ref{sepa-pro-1} and Lemma \ref{b-v-lem}, we have $-\Delta\psi_1=(K+\mu_1)\psi_1\ge 0$ in $B_1$. Hence, the assertion $\psi_{1}'\le 0$ follows by the same argument
as in the proof of Lemma \ref{apriorilem1}.
\end{proof}
\section{Parabolic problem}
The aim of this section is to prove Theorems \ref{intro-thm-2} and \ref{p-thm-1}. In this section, we do not assume the boundedness of the initial values unless stated otherwise. For any $u_0\in L^2(B_1)$,
we say that $u\in C^0 ((0,\infty); H^{1}_{0}(B_1))$ is a \textit{weak solution} if
$u_t$, $\Delta u$ and $e^u\in L^{1}(B_1\times [\tau, T))$ for every $0<\tau<T<\infty$,
$u$ satisfies \eqref{eq-intro-1} for a.e. $(x,t)\in Q$ and $u\to u_0$ in $L^2(B_1)$ as $t\to 0$, where $Q:=B_1\times (0,\infty)$.
\begin{thm}\label{p-thm-1} Assume that $N\ge 3$, 
$f\in \mathrm{Lip}[0,1]$ and $u_0\ge \phi_0$ in $B_1$.
\begin{enumerate}[{\rm (i)}]
\item Let $\lambda\in\left(0,\lambda^{*}\right)$. Assume that an unstable solution $v\in C^2(\overline{B_1})$ of \eqref{eq-intro-2}
exists. When $u_0\le v$ in $B_1$ and $u_0\not \equiv v$, there exists a unique global solution $u$ of \eqref{eq-intro-1}
such that $u\to v_{\lambda}$ in $C^{2}(\overline{B_1})$ as $t\to\infty$. On the contrary, when $v\le u_0$ in $B_1$, $u_0\in C^0(\overline{B_1})$ and $u_0\not\equiv v$, there exists a unique local solution $u$ which blows up in finite time.
    \item Let $\lambda=\lambda^{*}$. Assume that $v^*\in C^2(\overline{B_1})$, $u_0\in C^0 (\overline{B_1})$, $u_0\ge v^*$ in $B_1$
    and $u_0\not\equiv v^{*}$. Then, there exists a unique local solution $u$ which blows up in finite time.
\item Let $\lambda=\lambda_{*}$. When $u_0\leq V_{*}$ in $B_1$ 
and $u_0\not\equiv V_{*}$, there exists a global solution $u$ such that 
\begin{enumerate}
    \item $u\to v^{*}$ in $C^2(\overline{B_1})$
    as $t\to\infty$
    if $N\le 9$;
    \item $u\to V_{*}$ in $L^{2}(B_1)\cap C^{2}_{\mathrm{loc}}(\overline{B_1}\setminus \{0\})$ if $N\ge 10$, $f\ge f_H$ in $B_1$ and $f\not\equiv f_H$;
    \item $u\to v^{*}$ in $C^{2}(\overline{B_1})$ 
    if $N\ge 10$, $f\ge f_H$ in $B_1$ and $f\not\equiv f_H$.
\end{enumerate}
On the other hand, when  $u_0\ge V_{*}$ in $B_1$
and $u_0\not\equiv V_{*}$, there exists no weak solution which satisfies $u\ge V_{*}$ in $Q_{\tau}:=B_{1}\times\left(0,\tau\right)$, for any $\tau>0$.
\end{enumerate}
\end{thm}
We first assume that the initial data $u_{0}$ lies below a solution to \eqref{eq-intro-2}.
\begin{prop}
\label{p-prop-1}
Let $v$ be a (possibly singular) solution of \eqref{eq-intro-2}.
Assume that $\phi_0\le u_0\le v$ in $B_1$ and $u_0\not\equiv v$.
Then, there exists a global solution $u$ to the problem \eqref{eq-intro-1} satisfying $\phi_0\le u\le v$ in $Q$ and 
$u\in C^{2,1}(B_1\times\left(\tau,T\right))$
for all $T>\tau>0$.
In addition, if $u_0\in C^{0}(\overline{B_1})$,
\eqref{eq-intro-1} admits a unique classical global solution satisfying $u\le v$ in $Q$. 
\end{prop}
\begin{proof}
For the case $v\in C^{2}(\overline{B_1})$,
$\phi_0$ and $v$ are a sub-solution and a super-solution,
respectively. Hence, the assertions follow from a comparison principle and the parabolic regularity theory.
Thus, we focus on the case $v=V_{*}$.
By a standard iterative construction of solutions,
there exists a weak maximal solution $u\in C^{0}((0,\infty);H^{1}_{0}(B_1))$ satisfying $\phi_0\le u\le  V_{*}$ in $Q$  (see \cite{PV1995}).

Now, we show that $u\in C^{2,1}\left(B_{1}\times\left(\tau,\infty \right)\right)$ for any $\tau>0$. Take $0<\tau_2<\tau<T$.
Let $\psi$ be the solution of $\pl_{t}\psi=\Del\psi$ in $B_1$;
$\psi=0$ on $\pl B_{1}$; $\psi(x,0)=V_{\ast}-u_0$.
Then, $\psi$ is smooth and positive in $B_{1}\times\left[\tau_{2},\infty\right)$,
and $\psi\leq V_{\ast}-u$ in $B_{1}\times\left[\tau_{2},\infty\right)$ by a comparison principle.
Thus we have $u(x,\tau_2)<V_{*}-c$ in $B_{r_1}$ for some $c>0$ and $r_1\in\left(0,1\right)$. Moreover, let $S$ and $u^{(i)}$ be the solutions of
\begin{align*}
    &(\partial_t-\Delta)S=-f(r), \quad \text{in $B_1\times (\tau_2,\infty)$,}\quad S(x,\tau_2)=V_{*};
    \\
    &(\partial_t-\Delta) u^{(i)}
    =\lambda_{*}\exp\left({u^{(i-1)}}\right)
    -f(r)
    \quad \text{in }B_1\times (\tau_2,\infty),
    \quad u^{(i)}(x,\tau_2)=V_{*}-c\chi_{B_{r_1}}
\end{align*}
with 0-Dirichlet boundary condition and $u^{(0)}:=V_{*}$.
Since $\eta:=V_{*}-u^{(1)}$ satisfies
\begin{equation*}
    (\partial_t-\Delta)\eta=0
    \quad\text{in }B_1\times (\tau_2,\infty),
    \quad\eta(x,\tau_2)=c\chi_{B_{r_1}}(x),
\end{equation*}
we have $\eta>\ep$ in $B_{r_1/2}\times (\tau_2, T)$ for some $\ep>0$ depending only on $c$, $T$ and $r_1$.
Hence, there exist $\theta\in (0,1)$ and ${C}^{\dagger}>0$,
depending only on $c$, $r_1$ and $T$ such that
\begin{equation*}
\exp\left(u^{(1)}\right)=e^{V_{*}-\eta}
=\theta e^{V_{\ast}}+e^{V_{\ast}}
\left[e^{-\eta}-\theta\right]
\le \theta e^{V_{*}}+{C}^{\dagger} \quad \text{in $B_{1}\times (\tau_{2}, T)$.}
\end{equation*}
Therefore, we can deduce that $u^{(2)}\le \theta V_{*}+(1-\theta)S+{C}^{\dagger}\sigma$ in $B_{1}\times (\tau_2, T)$
by a comparison principle, where $\sigma$ is the solution of
$-\Delta \sigma =1$ in $B_1$;
$\sig=0$ on $\pl B_1$.
Fix $\tau_3\in (\tau_2,\tau)$.
Since $\norm{S(\cdot,t)}_{L^{\infty}(B_1)}<C(\tau_3)$ for any $t\ge \tau_3$,
Proposition \ref{intro-prop-2} yields $\exp\left({u^{(2)}}\right)\le C(\tau_2, \tau_3, u_0) r^{-2\theta}$ in $B_{1}\times [\tau_3,T)$, which implies that $\exp\left({u^{(2)}}\right)\in L^{\infty}((\tau_3,T);L^{\frac{N}{2}+\delta}(B_1))$ for some $\delta>0$.
As a result, we have $u^{(3)}\in L^{\infty}(B_{1}\times\left(\tau, \frac{T}{2}\right))$. Here, by the comparison principle, we have $u\le u^{(3)}$ in $B_{1}\times (\tau_2, T)$. Therefore,
the result follows from the standard parabolic theory \cite{Lieberman}.
In particular, if $u_{0}\in C^{0}(\overline{B_1})$, the uniqueness of the solution follows from a comparison principle. 
\end{proof}
Next, we identify the large time behavior of solutions.
\begin{prop}\label{p-prop-2}
    Fix $\lambda\le \lambda^{*}$ and let $v$ be a (possibly singular) solution of \eqref{eq-intro-2}.
    Let $u$ be a global-in-time solution of \eqref{eq-intro-1} with initial datum $u_{0}$,
    where $u_0\le v$ in $B_1$
    and $u_0\not\equiv v$.
    When $\lambda<\lambda^{*}$,
    we have $u\to v_{\lambda}$ in $C^2(\overline{B_1})$ as $t\to\infty$.
    For the case $\lambda=\lambda^{*}$, we have $u\to v^{*}$ in $L^2(B_1)\cap C^{2}_{\mathrm{loc}}(\overline{B_{1}}\setminus\{0\})$ as $t\to\infty$. In addition, 
    if $v^{*}\in C^{2}(\overline{B_1})$,
    we obtain $u\to v^{*}$ in $C^2(\overline{B_1})$ as $t\to\infty$.  
\end{prop}
\begin{proof}
Take $0<\tau<t$.
Multiplying $u_t$ to \eqref{eq-intro-1} and integrating it on $B_1\times (\tau,t)$, we have
\begin{equation*}
\int_{\tau}^{t}\int_{B_1}u_t^{2}
+\frac{1}{2}\int_{\tau}^{t}\int_{B_1}\frac{d}{dt}|\nabla u|^2
=
\int_{B_1}
\lambda(e^{u(r,t)}-e^{u(r,\tau)})
-f(r)(u(r,t)-u(r,\tau))
\,dx.
\end{equation*}
Then, we use Proposition \ref{p-prop-1} to yield
\begin{equation*}
\lVert u_{t}\rVert_{L^2(B_{1}\times (\tau,\infty))}\le C(\tau,f,v), \quad \sup_{t>\tau} \lVert \nabla u\rVert_{L^2(B_1)} \le C(\tau,f,v).
\end{equation*}
Arguing exactly as in the proof of Theorem~3.1 in \cite{PV1995}, 
we can deduce that the $L^2$-limit set of solutions is included in the set of solutions of \eqref{eq-intro-2}.
When $\lam=\lam^{\ast}$, Proposition~\ref{intro-prop-1}~(ii) tells us
$u(\cdot,t)\to v^{\ast}$ in $L^{2}(B_1)$ as $t\to\infty$.
Consider the case $\lam\in(0,\lam^{\ast})$.
For the case $v=v_{\lam}$, we can show that $u(r,t)\to v_{\lam}(r)$ in $L^2(B_1)$ as $t\to\infty$ since $v_{\lambda}$ is minimal. 
Thus, we focus on the case of $v\not\equiv v_{\lam}$ (namely, $v$ is unstable).
By Proposition~\ref{p-prop-1}, the comparison principle and Hopf's lemma shows that for $t_{0}>0$, there exists $\al\in(0,1)$ such that
$u(x,t)\leq \al v_{\lam}(x)+(1-\al)v(x)=:v_0$ for $t>t_0$. Hence,
we can assume without loss of generality that
$u_{0}\leq v_{0}$. We claim that
the function $v_0$ is a supersolution to \eqref{eq-intro-1} with initial data $v_{0}$. 
Indeed, we see that
$(v_{0})_{t}=0\geq\Del v_{0}+\lam_{\ast} e^{v_{0}}-f$ in $Q$
by the convexity of $s\mapsto e^s$. Therefore,
we have $u\leq v_{0}(x)<v(x)$ in $B_1$ by the comparison principle. 
Thus, by using Lemma~\ref{intersec-lem-2},
we deduce $u\to v_{\lam}$ in $L^2(B_{1})$ as $t\to\infty$. Note that the convergence of $u$ in $C^2_{\mathrm{loc}}(\overline{B_{1}}\setminus\{0\} )$ follows from the fact that $u\le v$ in $Q$ and the parabolic regularity theory.

Finally, we turn to the proof of the convergence of $u$ in $C^{2}(\overline{B_1})$,
except in case $\lam=\lam^{\ast}$ and $v=V_{\ast}$.
In light of parabolic regularity and the Arzel\`{a}--Ascoli theorem,
it suffices to show that $\norm{u}_{C^{2,1}(B_{1})}<C(T)$ for any $t>T$ with some $T>0$.
The case $v\in L^\infty(B_1)$ is trivial.
When $v=V_{\ast}$, again we use the supersolution $v_0$ to derive
$ \norm{e^{u}}_{L^{N/(2\beta)}(B_{1})}\le C(T)$ for all $t>T$ with some 
$\beta\in(0,1)$ and $T>0$. Hence, by the parabolic regularity theory, we obtain the result.
\end{proof}
We next turn to the case where the initial data $u_0$ lies above a solution to \eqref{eq-intro-2}.
\begin{prop}
\label{p-prop-3}
Assume that $N\ge 3$ and $f\in \mathrm{Lip}[0,1]$. Then,
there exists a unique local solution $u$ which blows up in finite time if 
\eqref{asu0} and one of the following three conditions holds:
\begin{enumerate}
    \item[{(i)}] $\lambda<\lambda^{*}$, an unstable solution $v\in C^2(\overline{B_1})$ exists, $u_0\ge v$ in $B_1$ and $u_0\not\equiv v$;
    \item[{(ii)}] $\lambda=\lambda^{*}$, $v^{*}\in C^2(\overline{B_1})$, $u_0\ge v^{*}$ in $B_1$ and $u_0\not\equiv v^{*}$;
    \item[{(iii)}] $\lambda>\lambda^{*}$.  
\end{enumerate}
Moreover, when $\lambda=\lambda_{*}$, $u_0\ge V_{*}$ in $B_1$ and $u_0\not\equiv V_{*}$, there exists no weak local solution satisfying $u\ge V_{*}$ in $Q_{\tau}$ for any $\tau>0$
(instantaneous blow up).  
\end{prop}
\begin{proof}
We first prove (i) and (ii) by an argument similar to that in \cite{PV1995}.
We define $\mu$ and $\xi>0$ as the first eigenvalue and the eigenfunction of
$-\Del-\lam e^{v}$ in $B_1$ with $\lVert \xi\rVert_{L^1(B_1)}=1$, respectively. Here, we set $v=v^{*}$ for the case (ii). Since $v$ is unstable, we have $\mu\le 0$. Moreover,
since $u_0\ge v$ in $B_1$
and $u_0\not\equiv v$, we have $u\ge v$ in $B_{1}\times (0,T)$, where $T$ is the maximal existence time. Hence, $y:=\int_{B_1}(u-v)\xi\,dx$ satisfies
\begin{align*}
y'(t)&=\int_{B_1}u_t\xi\,dx=\int_{B_1}(u-v)\Delta \xi+\lambda(e^u -e^v)\xi\,dx
\\&
=\int_{B_1}-\mu\xi(u-v)+\lambda\xi(e^{u}-e^{v}-e^{v}(u-v))\,dx
\\&
\ge \frac{\lam}{2}\int_{B_1}e^{v}(u-v)^2\xi\,dx
\ge \frac{\lam}{2e^{\norm{v}_{L^{\infty}}}}
y^2(t)
\end{align*}
in $(0,T)$ by Jensen's inequality. It implies that $T<\infty$.

When $\lambda>\lambda^{*}$, the result follows by using a method similar to that in the proof in \cite[Corollary~3]{BCMR}.
Finally, the nonexistence result follows by using a method similar to that in the proof in \cite[Theorem~7.1]{PV1995}.
\end{proof}

\begin{proof}[Proof of Theorem \ref{intro-thm-2} and Theorem \ref{p-thm-1}]
The results follow from Propositions \ref{p-prop-1}, \ref{p-prop-2}, \ref{p-prop-3} and Theorem \ref{intro-thm-1}. 
\end{proof}
\section{Grow-up rate}
\label{sec-rate}
The aim of this section is to determine the grow-up rate.
Throughout this section, we assume that $N\ge 10$, $f\in \mathrm{Lip}[0,1]$, \eqref{asu0} and $0\le f\le f_H$ in $B_1$.  
By a comparison principle and the uniqueness of the solution (see Theorem \ref{intro-thm-2}), we assume without loss of generality that $u$ is radially symmetric. 
\subsection{Inner-outer matching}
\label{sec-matching}
We first consider the behavior of $u$ in the inner region.
When $f=0$, we can assume 
without loss of generality that $\alpha(t):=\lVert u(\cdot,t)\rVert_{L^{\infty}(B_{1})}=u(0,t)$ by using 
symmetrization and a comparison argument. Therefore, by a standard intersection method, we can obtain $u\ge y_0$ in $B_1$ for large $t$, where $r_t:=e^{-\alpha(t)/2}$ and
$y=y_{0}$ is the solution of the problem
\begin{equation}
\label{eq-y}
-\Delta y =\lambda_{*}e^{y}-f \quad \text{in $B_1$},\quad y(0)=\alpha(t), \quad y'(0)=0
\end{equation}
with $\lambda_*=\lambda_0:=2N-4$ and $f=f_0=0$. 
On the other hand, when $f\not\equiv 0$, the maximum point of $u$ may not coincide with $r=0$. Nevertheless, we can obtain the following analogous estimate.
\begin{lem}\label{lem:4-1}
Let $y=y(r,t)$ be the solution to \eqref{eq-y} with $\al(t):=\norm{u(\cdot,t)}_{L^{\infty}(B_1)}$.
For any $\ep>0$, there exists $T>0$ such that $y\le u$ in $B_1\setminus B_{\ep}$ for all $t>T$.
\end{lem}
\begin{proof}
Fix sufficiently small $\ep>0$.
By Hopf's lemma, we have $-V_{*}'(1)>-\pl_r u(1,t)$ for any $t\ge 1$.
By Propositions \ref{uniqueprop} and \ref{sepa-pro-1},
the function
$\hat{y}(r,\beta):=w(r,\beta)-\log \lambda_{*}$ satisfies
\begin{align*}
    &-\Delta \hat{y}=\lambda_{*}e^{\hat{y}}-f\quad\text{in }B_1,
    \quad
    \hat{y}'(0)=0,\quad \hat{y}(1,\beta)<0
\end{align*}
and
$\hat{y}\to V_{*}$ in $C^{2}_{\mathrm{loc}}(0,1]$ as $\beta\to\infty$. 
In addition,
by Remark~\ref{monotone-rem},
we have  $\hat{y}(0,\beta)>\hat{y}(r,\beta)$ in $B_{1}$ for any sufficiently large $\beta>0$.
Therefore, there exists $\beta_0>\sup_{\ep<r<1}V_{*}(r)+\log \lambda_{*}$ such that $u(r,1)$ intersects with $\hat{y}(r,\beta)$ exactly once for any $\beta>\beta_0$.
By the intersection comparison method (see \cite[Proposition~52.28]{QSbook}),
$u(r,t)$ intersects with $\hat{y}(r,\beta)$ at most once for any $\beta>\beta_0$ and $t>1$.
Now, we take $T>1$ such that $\alpha(t)>\beta_0-\log \lambda_{*}$ for any $t>T$.
Since $u\le V_{*}$ in $B_1$ for all $t>0$,
we have $\alpha(t)=\sup_{r\le \ep}u(r,t)$ for any $t>T$.
Thus, we verify that $y=\hat{y}(r,\alpha(t)+\log \lambda_{*})$ intersects with $u(r,t)$ in $B_{\ep}$ for any $t>T$.
It implies that $u\ge y$ in $B_1\setminus B_{\ep}$ for any $t>T$.
\end{proof}
Next, for fixed $0<\ep\ll 1$,
we describe the asymptotic behavior of $y=y(r,t)$ in terms of $\alpha(t)$ near a matching point $r=\ep$.

When $f=0$,
by the self-similarity of the equation \eqref{eq-y},
we deduce that $y_0$ satisfies $y_0=\alpha(t)+\underline{y}(r/r_t)$,
where $\underline{y}$ is the solution of 
\begin{equation*}
    -\Delta \underline{y}=(2N-4)e^{\underline{y}} \quad \text{in $\mathbb{R}^{N}$,} \quad \underline{y}(0)=\underline{y}'(0)=0.
\end{equation*}
Since the solution $\underline{y}$ satisfies
\begin{equation*}
    \underline{y}(r)=
    \begin{cases}
          V_{0}(r)-b_{N}r^{-\gamma}(1+o(1)) &\text{if $N\ge 11$,}\\
          V_{0}(r)-b_{N}r^{-4}\log r(1+o(1))  &\text{if $N=10$,}
    \end{cases}
    \quad \text{as $r\to\infty$}
\end{equation*}
for some $b_N>0$ (see \cite{DGLV1998}),
Lemma \ref{lem:4-1} yields\label{error-inner}
\begin{align}\label{eq:v_0-y_0}
V_{0}-u\le V_{0}-y_{0}\simeq
 \begin{cases}
     e^{-\gamma \alpha(t)/2}r^{-\gamma} \quad &\text{if $N\ge 11$,}\\
    e^{-2\alpha(t)} r^{-4}\log (re^{\alpha(t)/2})  &\text{if $N=10$,}
 \end{cases}
 \quad \text{in $B_{1}\setminus B_{\ep}$}.
\end{align}
Here $f(t)\simeq g(t)$ means that there exists $C>0$ and $T$ independent of $t$ such that $C^{-1}f(t)\leq g(t)\leq Cf(t)$ for $t>T$.
As a result of \eqref{eq:v_0-y_0}, we can estimate $y_0$ in terms of $\alpha(t)$. 
When $f\not\equiv 0$, the above argument no longer applies due to the lack of scaling invariance.
Nevertheless, we can still obtain the following analogous result.
\begin{prop}
\label{prop-inner}
Define $\Phi:=V_{*}-u$ and let $\hat{r}$ be as in Proposition~\ref{inner-prop-1}.
For any small $\ep\in\left(0,\hat{r}\right)$,
there exist positive constants $T$ and $C$ such that 
\begin{align}\label{eq:prop-inner}
\Phi(r,t)\le 
 \begin{cases}
     Ce^{-\gamma\alpha(t)/2}r^{-\gamma} \quad &\text{if $N\ge 11$,}\\
    Ce^{-2\alpha(t)} r^{-4}\log (re^{\alpha(t)/2})  &\text{if $N=10$}
 \end{cases}
\end{align}
for all $r\in\left[\ep,\hat{r}\right)$ and for all $t>T$.
\end{prop}
\begin{proof}
  By \eqref{eq:v_0-y_0},
  it suffices to show that $\Phi\leq V_{0}-y_{0}$ in $B_{\ep}\setminus B_{\hat{r}}$.
  We recall that
  \begin{equation*}
    V_{*}-y=W-w(r,\alpha(t)+\log \lambda_{*}) \quad \text{and} \quad V_{0}-y_{0}=W_{0}-w_{0}(r,\alpha(t)+\log \lambda_0).
  \end{equation*}
  Since $\lam_{0}\leq\lam_{\ast}$ follows from Proposition~\ref{sepa-pro-1}~(iii) and the fact $\lambda_{0}=e^{W_{0}(1)}$ and $\lambda_{*}=e^{W(1)}$,
  Proposition~\ref{sepa-pro-1}~(i) implies
  $w_0(r,\alpha(t)+\log \lambda_0)\le w_0(r,\alpha(t)+\log \lambda_{*})$ in $B_1$.
  The desired estimate is a consequence of
  Proposition~\ref{inner-prop-1} and Lemma~\ref{lem:4-1}.
\end{proof}
We remark that Proposition~\ref{prop-inner} gives a suitable control at the matching point $r=\ep$.
In fact, as we shall see at the end of this subsection,
Theorem~\ref{intro-thm-3} follows from \eqref{eq:prop-inner} and the following estimates of $\Phi$ in the outer region.
\begin{prop}\label{prop-outer}
Let $\Phi$ be as in Proposition~\ref{prop-inner} and let $\hat{Q}$ be defined by
\begin{equation*}
\hat{Q}(t):= e^{-\mu_1 t} \text{ if $f\not\equiv f_H$,}\quad
\hat{Q}(t):=
\begin{cases}
t^{-1} &\text{if $f\equiv f_H$, $N\ge 12$,}\\
t^{-1}(\log t)^{-1} &\text{if $f\equiv f_H$, 
$N=11$.}
\end{cases}
\end{equation*}
There exists positive constants $C$ and $\hat{T}$
independent of $t$ such that
$\Phi\leq C\hat{Q}(t)\phi_1$ in $B_{1}$ for all $t>\hat{T}$.
Moreover, for any $\ep>0$, there exist $c>0$ and $T>0$ independent of $t$
such that $\Phi\ge c\hat{Q}(t)\phi_1$ in $B_1\setminus B_{\ep}$ for all $t>T$.
\end{prop}
Here we provide a formal derivation of the result.
We note that the function $\Phi$ satisfies
\begin{equation}\label{phieq-1}
  \Phi_t +\cL \Phi=-\lambda_{*}e^{V_{*}} F(\Phi)\quad \text{in $B_1\times (0,\infty)$,} \quad\text{where } F(u):=e^{-u_{+}}-1+u_{+}
\end{equation}
and $\Phi=0$ on $\pl B_{1}\times (0,\infty)$.
Assume formally that the large-time behavior of $\Phi$
in the outer region is governed by its projection $c(t)\phi_1$
onto the first eigenspace of $\mathcal{L}$. 
Then, we have
\begin{equation*}
    c'(t)+\mu_1 c(t)=-(\lambda_{*}e^{V_{*}}F(\Phi), \phi_1)_{L^2(B_1)}.
\end{equation*}
We remark that 
\begin{align}\label{bp1}
0\le F'(u)\le \min\{u_{+},1\}, \qquad
    \begin{cases}
     \frac{u^2}{2e^2} \le F(u)\le \frac{u^2}{2}
     & \text{if}\hquad  0\le u<2,
     \\
     \frac{u}{2}\le F(u)\le  u  & \text{if}\hquad  u\ge 2;
    \end{cases}
\\
\label{bp2}
N-2-3\gamma=0 \quad \text{if }N=11,
\qquad N-2-3\gamma> 0 \quad \text{if }N\ge 12.
\end{align}
Set $\delta(t):=c(t)^{1/\gamma}$.
By using the fact that $\phi_1\simeq r^{-\gamma}$ near $r=0$, \eqref{bp1}, \eqref{bp2} and Proposition~\ref{intro-prop-2}, we formally obtain
\begin{align*}
&(\lambda_{*}e^{V_{*}}F(\Phi), \phi_1)_{L^2(B_1)}\simeq \int_{0}^{\delta(t)}r^{N-3}\Phi \phi_1 \,dr+\int_{\delta(t)}^{1}r^{N-3}\Phi^2 \phi_1\,dr\\
&\simeq
\begin{cases}
c(t)\delta(t)^{N-2-2\gamma}+c(t)^2, & \text{if}\hquad N\ge 12,\\
c(t)\delta(t)^{N-2-2\gamma}-c(t)^2 \log \delta(t), &\text{if}\hquad  N=11
\end{cases}
\simeq
\begin{cases}
c(t)^{2} & \text{if}\hquad N\ge 12,
\\
c(t)^{2} |\log c(t)| & \text{if}\hquad  N=11.
\end{cases}
\end{align*}
As a result, we obtain $c\simeq \hat{Q}$ by the fact that $\mu_1>0$ if $f\not\equiv f_H$ in $B_1$, and $\mu_1=0$ if $f\equiv f_H$ (see Lemma~\ref{b-v-lem}).

The rigorous verification for the case $f\not\equiv f_H$ proceeds similarly to that in \cite[Section 5]{DGLV1998}.
Indeed, we can show that $\Phi\lesssim e^{-\mu_1 t}\phi_1$ in $B_1$ for $t\ge T$ with some $T>0$ because the function $\eta:=C(T) e^{-\mu_1 t}\phi_1$ is a super-solution of \eqref{phieq-1} with some $C(T)>0$.
In addition, we construct a sub-solution $z$ as the solution of
\begin{equation*}
z_t +\cL z=-\lambda_{*}e^{V_{*}} F(\eta)\quad \text{in $B_1\times (0,\infty)$,} \quad z=0 \quad \text{on }\pl B_{1}
\end{equation*}
with the same initial data as $\Phi$. Estimating $z$ as in \cite{DGLV1998}, 
we can show that, for any fixed $\ep>0$, 
we have $z\ge c(\ep)e^{-\mu_1 t}\phi_{1}$ in $B_1\setminus B_{\ep}$ for all $t>T(\ep)$ with some $T_{\ep}$. Therefore, the assertion of Proposition~\ref{prop-outer} follows for the case $f\not\equiv f_H$.

On the contrary, when $f\equiv f_H$,
the nonlinear term affects the leading-order asymptotics of $u$ in the outer region. Therefore, the construction of super/sub-solutions yielding the optimal leading-order term is more subtle.
In the subsequent two subsections, we rigorously prove Proposition \ref{prop-outer} for the case $f \equiv f_H$ by constructing specific super/sub-solutions.

We conclude this subsection with the proof of Theorem~\ref{intro-thm-3},
using Proposition~\ref{prop-inner} and assuming Proposition~\ref{prop-outer}.
\begin{proof}[Proof of Theorem~\ref{intro-thm-3}]
Let $\ep\in\left(0,\hat{r}\right)$.
Thanks to Propositions~\ref{prop-inner},
\ref{prop-outer} and Lemma \ref{b-v-lem},
there exists $T>0$ such that 
\begin{equation*}
 Q(t)\ep^{-\gamma}
 \lesssim Q(t)\phi_{1}(\ep)
 \lesssim\Phi
 \lesssim e^{-\gamma \alpha(t)/2}\ep^{-\gamma} 
\end{equation*}
for all $t>T$,
which implies that $\alpha\le -{2}{\gamma}^{-1}\log Q(t)+O(1)$ as $t\to\infty$.
On the other hand, by using Propositions~\ref{intro-prop-2} and \ref{prop-outer}, we have
\begin{equation*}
\alpha(t)\ge u(r,t)= V_{*}(r)-\Phi(r,t)\ge \alpha(t)- C_1 {Q}(t)r^{-\gamma}+\log (2N-4)+o(1) \quad \text{as $t\to\infty$,}
\end{equation*}
where $r=(\lambda^{*})^{1/2}e^{-\alpha(t)/2}$ and $C_1>0$ is independent of $t$. Therefore, we obtain 
$\alpha\ge -{2}{\gamma}^{-1}\log{Q(t)}-O(1)$ as $t\to\infty$.
Thus, the proof is complete.
\end{proof}
\subsection{Proof of Proposition \ref{prop-outer} ($\boldsymbol{N\geq 12}$, $\boldsymbol{f\equiv f_H}$)}
\label{sec-rate-geq:12}
First we deal with the case $N\geq 12$.
For the simplicity of the computation,
we set $\Psi(r,t):=r^{\gam}\Phi(r,t)$.
Then, by Lemma~\ref{b-v-lem}, we deduce that $\Psi$ is a solution of
\begin{align}\label{eq-diff-12dim}
\begin{dcases}
  \pl_{t}\Psi-\Del\Psi-H\Psi=
  -\left(
    \frac{2N-4}{r^2}+H
  \right)r^{\gam}F(r^{-\gam}\Psi)
  & \text{in $B_{1}^{m}\times (T,\infty)$},
  \\
  \Psi=0
  &\text{on }\pl B_{1}^{m}\times(T,\infty)
\end{dcases}
\end{align}
for any $T\ge0$,
where $F$ is that in \eqref{phieq-1} and $m:=N-2\gam$.
Note that $H>0$ coincides with the first eigenvalue of $-\Delta$ in $B_{1}^{m}$.

We first construct a specific supersolution as follows.
\begin{prop}
\label{prop-super-12}
There exist $T,\kappa_1, \kappa_2>0$ such that 
$\widetilde{\Psi}(r,t):=\kappa_1 t^{-1}\psi_1+\kappa_2 t^{-2}\Tilde{\psi}$ is a supersolution of \eqref{eq-diff-12dim}, 
where $\Tilde{\psi}\in C^2_{0}(\overline{B_1})$ is a solution of 
\begin{align}\label{se1}
  \begin{cases}
      (-\Delta-H)\Tilde{\psi}=
      -\psi_1^2+(\psi_1^2,\psi_1)_{L^2(B_{1}^m)}\psi_1
      \quad \text{in $B_{1}^{m}$,}
      \\
      \Tilde{\psi}=0 \quad \text{on }\partial B_1,
      \quad\widetilde{\psi}\perp\psi_{1} \text{ in } L^{2}(B_{1}^m).
  \end{cases}
\end{align}
In particular, there exists $C>0$ such that $\Psi\le Ct^{-1}\psi_1$ in $B_1$ for any $t>T$.
\end{prop}
\begin{proof}
Since $-\psi_1^2+(\psi_1^2,\psi_1)_{L^2(B_{1}^m)}\psi_1$ is orthogonal to $\psi_1$ in $L^{2}(B_1^m)$, we can show the existence of $\Tilde{\psi}\in C^{2}_{0}(\overline{B_1})$.
Set $\theta:=(\psi_1^2,\psi_1)_{L^2(B_{1}^m)}$.
By Hopf's lemma,
we can take $T>0$ and $\kap_{1}>16/\theta$
so that
\begin{align}
\label{kiso-1}
T=\frac{8}{\theta}\max\left\{
  \norm{\Tilde{\psi}/\psi_1}_{L^{\infty}},
  \norm{\psi_1}_{L^{\infty}}
\right\},
\quad
\frac{\kap_{1}\psi_{1}}{2T}
>r^{\gamma}\left(V_{H}(r)
  -\phi_{0}
\right)
\,\text{ in $[0,1]$.}
\end{align}
We define $\kap_{2}:=4\kap_{1}/\theta$.
Then, we have $\widetilde{\Psi}\ge \kappa_1\psi_1/(2t)$ in $B_1\times [T,\infty)$,
Thus we get $\widetilde{\Psi}(r,T)
\ge r^{\gamma}(V_{H}(r)-
\phi_{0}
)\ge \Psi(r,T)$ in $B_{1}^{m}$ by the fact that $u\ge \phi_0$ in $B_1\times (0,\infty)$.
Moreover, it follows from \eqref{bp1} and \eqref{kiso-1} that
\begin{equation*}
     \left(\frac{2N-4}{r^2}+H\right)
     r^{\gamma}F(r^{-\gamma}\widetilde{\Psi})\ge
    \begin{cases}
      r^{-2}\widetilde{\Psi}
      &\text{if }r^{-\gamma}\widetilde{\Psi}>2,
      \\
      r^{-2-\gamma}\widetilde{\Psi}^2
      &\text{if }r^{-\gamma}\widetilde{\Psi}\le 2
    \end{cases}\,
    \ge \kappa_2 t^{-2}\psi_{1}^{2}.
\end{equation*}
As a result, it follows from \eqref{se1} and \eqref{kiso-1} that
\begin{align*}
    (\partial_t -\Del-H)\Tilde{\Psi}
    &=-\kap_1 t^{-2}\psi_1
    -2\kappa_2t^{-3}\Tilde{\psi}
    -\kap_2 t^{-2}\psi_{1}^2
    +\kap_2 \theta t^{-2}\psi_1&
    \\
    &\ge\left(\frac{3\theta}{4}\kappa_2-\kappa_1\right)t^{-2}\psi_1
    -\kappa_2 t^{-2}\psi_1^2\ge -\left(\frac{2N-4}{r^2}+H\right)
     r^{\gamma}F(r^{-\gamma}\widetilde{\Psi})
\end{align*}
in $B_1\times [T,\infty)$.
Hence, $\Tilde{\Psi}$ is a supersolution to \eqref{eq-diff-12dim},
and the claim follows from a comparison principle.\footnote{Since $-F$ is non-increasing, we can use a standard comparison principle.}
\end{proof}
We next turn to the construction of a sub-solution.
we begin by proving the following 
\begin{lem}\label{12-sub-lem-1}
Fix $\sig\in\left(0,1/2\right)$. Then, the following hold.
\begin{itemize}
  \item[(i)] For any $t>1$,
  there exists the unique solution $\upsi\in W^{1,\infty}_{\mathrm{loc}}((1,\infty);H^{1}_{0}(B_{1}^m))$ of 
\begin{align}\label{problem-sub-1}
(-\Delta-H)\upsi=-2P+(2P,\psi_1)_{L^{2}(B_1^m)}\psi_1\,\,\text{in $B_1^{m}$,}
\quad \text{$\upsi \perp\psi_1$ in $L^{2}(B_1^m)$}
\end{align}
with 0-Dirichlet boundary condition, where
\begin{equation*}
    P=P(r,t)
    :=\left(\frac{2N-4}{r^2}+H\right)r^{\gamma}F(r^{-\gamma}\sigma t^{-1}\psi_{1}).
\end{equation*}
  \item[(ii)] Let $a:=\min\left\{({m-4})/{4\gam},1/2\right\}>0$ and $\delta=\delta(t)$ be a unique solution of
\begin{equation}\label{defdel-12}
    2\delta(t)^{\gamma}={\sigma t^{-1}}\psi_1(\del(t)).
\end{equation}
Then, there exist $C>0$ independent of $\sigma$ and $t$ so that
\begin{equation*}
0\le (P,\psi_1)_{L^{2}(B_{1}^m)}\le C \sigma^2 t^{-2}, \quad \upsi \le C\sigma^{1+a} t^{-1-a}\psi_1 \quad \text{in $B_{1}^m$}
\end{equation*}
for any $t>1$, and 
\begin{equation*}
    0\le -(P_t, \psi_1)_{L^2(B_{1}^m)} \le C\sigma^2 t^{-3}, \hquad -
\upsi_t \le C \sigma^{1+a} t^{-2-a}\psi_1 \quad \text{in $B_{1}^m$.}
\end{equation*}
for a.e. $t>1$.
\item[(iii)] For any $\ep>0$, there exist $C(\ep), T_{1}(\ep)>0$ independent of $\sigma$, $t$ such that
\begin{equation*}
-\upsi\le  C(\ep) \sigma^{1+a} t^{-1-a}\psi_1, \quad  
\upsi_t \le C(\ep) \sigma^{1+a} t^{-2-a} \psi_1\hquad
\text{in $B_{1}^m\setminus B_{\ep}^m$}
\end{equation*}
for a.e. $t>T_{1}(\ep)$.
\end{itemize}
\end{lem}
\begin{proof}
By a direct computation, we have
\begin{equation*}
    \pl_{t}P=-\sigma t^{-2}
    \left(\frac{2N-4}{r^2}+H\right)
    F'(r^{-\gamma}t^{-1}\sigma\psi_1)\psi_{1}
    \quad
    \text{and}\quad F'(u)=(1-e^{-u})_{+}.
\end{equation*}
We first prove the assertion
(i).
Note that $m\ge 5$ if $N\ge 11$,
whence $r^{-2}\in L^2(B_{1}^{m})$.
Thus we can check that $P\in W^{1,\infty}_{\mathrm{loc}}((1,\infty); L^{2}(B_{1}^{m}))$ by using \eqref{bp1}.
As a result, we can verify that the equation \eqref{problem-sub-1} admits the unique solution
$\upsi\in W^{1,\infty}_{\mathrm{loc}}((1,\infty); H^{1}_{0}(B_{1}^{m}))$. Next, we prove the assertions
(ii) and (iii). it follows from Lemma~\ref{fnolem} that the equation \eqref{defdel-12} has a unique solution $\delta$ such that
\begin{equation}
\label{kankei-12}
    r^{-\gamma}t^{-1}{\sigma}\psi_1(r)
    \leq 2
    \iff
    r\ge \delta(t).
\end{equation}
Hence, by \eqref{bp1} and \eqref{bp2}, we obtain
\begin{align}
\quad 0\le (P,\psi_1)_{L^2(B_1^{m})}&\lesssim
\int_{0}^{\delta(t)}r^{N-2\gam-3} \sigma t^{-1}
  \psi_{1}^2\,dr
  +
  \int_{\delta(t)}^{1}r^{N-3\gamma-3}\sigma^2 t^{-2}
  \psi_{1}^3\,dr
  \notag
  \\&
  \lesssim \sigma^2 t^{-2}\delta(t)^{N-3\gamma-2}
  +\sigma^{2}t^{-2}
  \lesssim\sigma^2 t^{-2}\quad \text{for any $t\ge 1$}.
\label{naiseki-1-a}
\end{align}
It also follows from \eqref{bp1} and \eqref{bp2} that
\begin{align*}
    \norm{P}_{L^2(B_{1}^m)}^2
    &\lesssim
    \int_{0}^{\del(t)}\sig^{2}t^{-2}r^{N-2\gam-5}\psi_{1}^{2}
    \,dr+
    \int_{\del(t)}^{1}
    \sig^{4}t^{-4}r^{N-4\gam-5}\psi_{1}^{4}\,dr
    \\
    &\lesssim
    \sig^{2}t^{-2}\del(t)^{m-4}+\sig^4 t^{-4}
    \lesssim \sigma^{2+2a}t^{-2-2a} \quad \text{for any $t\ge 1$}.
\end{align*}
Thus, by the condition $\underline{\psi}\perp\psi_{1}$,
the Poincar\'{e} inequality and energy estimate, we obtain
\begin{align*}
  \norm{\upsi}_{L^{2}(B_{1}^m)}^{2}
  \lesssim\norm{P}_{L^{2}}\norm{\upsi}_{L^{2}}
  +(P,\psi_{1})_{L^2}\norm{\upsi}_{L^2}
  \norm{\psi_1}_{L^2}
  \lesssim \sigma^{1+a} t^{-1-a }\norm{\upsi}_{L^2(B_{1}^m)} \quad \text{for $t\ge 1$.}
\end{align*}
Hence, we have
$
    \norm{\upsi}_{L^2(B_{1}^m)}\lesssim \left(\sig/t\right)^{1+a}
$
for $t\ge 1$.
Now we invoke \cite[Proposition~47.6]{QSbook} and \eqref{naiseki-1-a} to get
\begin{equation*}
\upsi\lesssim (P,\psi_{1})_{L^2(B_{1}^{m})} \lVert \psi_1\rVert_{L^{\infty}(B_{1}^m)}+\lVert \upsi_{+}\rVert_{L^2(B_{1}^m)}\lesssim \sigma^{1+a} t^{-1-a}  \hquad \text{in $B_{1}^m$}
\end{equation*}
for any $t>1$. Moreover, let $\ep>0$. Then, there exists $T_{1}(\ep)>0$ such that $\delta(t)<\ep$ for any $t>T_{1}$. Hence, by the elliptic regularity theory, we obtain
\begin{equation*}
  \norm{\upsi}_{C^{1}(\overline{B_{1}^m} \setminus B_{\ep}^m)}
  \lesssim_{\ep} (\norm{P}_{L^{\infty}(B_{1}^m \setminus B_{\ep/2}^m)}
  +\norm{\upsi}_{L^2(B_{1}^m)})
  \lesssim_{\ep} \sigma^{1+a} t^{-1-a}
\end{equation*}
for any $t>T_{\ep}$. Therefore, we obtain
\begin{equation}\label{12-sub-1}
\upsi\lesssim
\sigma^{1+a} t^{-1-a}\psi_1
\hquad\text{in $B_{1}^m\times (1,\infty)$},\quad
|\upsi|
\lesssim_{\ep}\sigma^{1+a} t^{-1-a}\psi_1 \hquad \text{in $B_{1}^m \setminus B_{\ep}^m\times (T_{1},\infty)$.}
\end{equation}
The remaining assertions follow by using the following and 
a similar method:
\begin{equation*}
(-\Delta -H)\upsi_{t}=-2P_t+(2P_t,\psi_1)\psi_1 \hquad
\text{in $B_{1}^{m}$},
\hquad \upsi_t=0 \hquad \text{on $\partial B_1^{m}$,}
\hquad \text{$\upsi_{t}\perp\psi_1$ in $L^{2}(B_{1}^m)$.}
\end{equation*}
\end{proof}

Next, we obtain upper bounds for $-\upsi$ and $\upsi_t$.
\begin{lem}\label{12-sub-lem-2}
There exist $C>0$ and $T_{2}>0$, independent of $\sigma$, $t$
such that
\begin{equation}\label{daiji-1}
-\upsi\leq  -C\frac{\sigma}{t}\log r
\hquad \text{in $(0,\delta(t)]$,}
\hquad -\upsi
\leq C\frac{\sigma^{1+a}}{t^{1+a}}r^{-a\gamma}
\log\left(\frac{t}{\sigma}\right)\psi_1 \hquad \text{in $(\delta(t), 1]$}
\end{equation}
for any $t>T_2$ and
\begin{equation}\label{daiji-2}
\upsi_t(r)\leq 
-\frac{C\sigma}{t^2}\log r \hquad \text{in $(0,\delta(t)]$},
\,\,\, \upsi_t(r)\leq
\frac{C\sigma^{1+a}}{t^{2+a}}r^{-a\gamma}
\log\left(\frac{t}{\sigma}\right)
\psi_1 \hquad \text{in $(\delta(t),1]$}
\end{equation}
for a.e. $t>T_2$.
\end{lem}

\begin{proof}
We only prove \eqref{daiji-1}.
It follows from  
\eqref{bp1} and \eqref{kankei-12} that
\begin{align}\label{mikata-1}
&\frac{1}{2}\left(\frac{2N-4}{r^2}+H\right)\sigma t^{-1}\psi_1\le 
P\le \frac{2N-4+H}{r^{2}}\sigma t^{-1}\psi_1 \quad
&\text{in $(0,\delta(t)]$}
,
\\
\label{mikata-2}
    &\frac{1}{2e^2}\left(\frac{2N-4}{r^2}+H\right)\sigma^2 t^{-2}r^{-\gamma}\psi_{1}^2\le 
P\le \frac{2N-4+H}{2r^{2+\gamma}t^2}\sigma^2 \psi_1^{2} \quad
&\text{in $(\delta(t),1]$}
\end{align}
for any $t>1$.
By using above inequalities and the definition of $\del(t)$,
we can choose sufficiently large constants $C_1>0$ and $1/\ep>0$,
independent of $\sig$ and $t$, so that the function
$\eta_1=\eta_{1}(r,t):=C_1\sigma t^{-1}\log r$ satisfies
\begin{align*}
 \frac{t}{C_1\sigma}(-\Delta-H)\eta_1
 =-(m-2)r^{-2}+H(-\log r)
 \le  -\frac{2t}{C_1\sigma} P(r,t) \quad \text{in $B^{m}_{\ep}$}
\end{align*}
for all $t>1$.
Then, we deduce that $(-\Delta-H)(\eta_1-\upsi)\le 0$ in $B_{\ep}^m$ for any $t>1$
in light of Lemma~\ref{12-sub-lem-1}~(i). In addition,
by Lemma~\ref{12-sub-lem-1}~(iii),
we can choose $T_{2,1}>T_1$ independent of $\sig$ 
such that $\eta_1(\ep,t)<\upsi(\ep,t)$ for any $t>T_{2,1}$. By applying the maximum principle, we have
$\eta_1<\upsi$ in $B_{\ep}^m$ for any $t>T_{2,1}$,
and thus we obtain the first inequality of \eqref{daiji-1}.

Now, we prove the second inequality. By the definition of $\delta(t)$ and \eqref{12-sub-1}, there exists $C_2>C_1$ independent of $\sigma, t>0$ so that $\eta_2<\eta_1$ at $r=\delta(t)$ for any $t>T_{2,1}$, where
\begin{equation*}
    \eta_2=\eta_{2}(r,t)
    :=C_2\sigma^{1+a}t^{-1-a}
    \log\left(\sig t^{-1}\right)r^{-a\gamma}.
\end{equation*}
It implies that $\eta_2(\delta(t))\le \upsi(\delta(t))$ for any $t>T_{2,1}$.
Moreover, from \eqref{mikata-1} and \eqref{mikata-2},
there exist $T_{2,2}>T_{2,1}$ and $\ep>0$ independent of $\sigma$ and $t$ such that
\begin{align*}
    &\frac{t^{1+a}}
    {C_{2}\sigma^{1+a}\log(\sig t^{-1})}
    (-\Delta-H)\eta_2
    = a\gamma (m-a\gamma-2)r^{-2-a\gamma}
    -H r^{-a\gamma}
    \\&\hspace{40pt}
    \ge
    \frac{\del(t)^{(1-a)\gam}}{4}a\gamma (m-a\gamma-2)r^{-2-\gamma}
    \ge -\frac{2t^{1+a}}{C_2\sigma^{1+a}\log(\sig t^{-1})}P
\end{align*}
in $B_{\ep}^{m}\setminus B_{\delta(t)}^{m}$ for any $t>T_{2,2}$. 
Furthermore, by \eqref{12-sub-1} and the definition of $a$,
we obtain $\eta_2(\ep)\leq\upsi(\ep)$ for any $t>T_{2,3}$,
where $T_{2,3}>T_{2,2}$ is independent of $\sigma$, $t$.
Therefore, we verify that
$(-\Delta-H)\eta_2-\upsi\le 0$ in $B_{\ep}^m \setminus B_{\delta(t)}^{m}$ and $\eta_2-\upsi \le 0$ for $r\in\left\{\delta(t),\ep\right\}$.
By using the maximum principle and \eqref{12-sub-1}, we obtain \eqref{daiji-1}.
We note that \eqref{daiji-2} follows from a similar method.
\end{proof}
Then, we estimate the zero of $\uPsi$, which is necessary to confirm the initial condition.
\begin{lem}\label{12-sub-lem-3}
Define $\underline{\Psi}:=\sigma t^{-1} \psi_1+M_{1} \upsi$,
where $\upsi$ is that in Lemma~\ref{12-sub-lem-1}~(i).
Then, there exist $M_{1}, T_3>2$
independent of $\sigma$ such that $\uPsi(r,t)\le 0$ in $B_{\delta(t)}^{m}$
for all $t>T_3$.
\end{lem}
\begin{proof}
It follows from Lemmata~\ref{12-sub-lem-1}, \ref{12-sub-lem-2} and \eqref{problem-sub-1}
that $\Delta \upsi\ge 0$ near $r=0$ for any fixed $t$.
Combining the same argument as in Lemma~\ref{apriorilem1} with  Lemmata~\ref{12-sub-lem-1} and \ref{12-sub-lem-2},
we can deduce that $r^{m-1}\upsi'\downarrow 0$ as $r\downarrow 0$. 
Thus we have
\begin{equation*}
\upsi(r)
=\upsi(1/2)
+\int_{r}^{1/2}\int_{0}^{\rho}
    \frac{s^{m-1}}{\rho^{m-1}}
    [H\upsi-2P(s,t)+(2P,\psi_1)_{L^2}\psi_1]
\,ds\,d\rho
\end{equation*}
for any $r\in [0,1/2)$.
By using Lemma~\ref{12-sub-lem-1}~(ii), we see that
\begin{equation*}
\upsi(r)\le
-2\int_{\delta}^{1/2}\int_{\delta}^{\rho}
    \frac{s^{m-1}}{\rho^{m-1}}P(s,t)
\,ds\,d\rho
+ O(\sigma^{1+a} t^{-1-a})+ O(\sigma^{2} t^{-2})\quad\text{in $B_{\delta(t)}$.}
\end{equation*}
Now, by using \eqref{mikata-2} and the fact $N-3\gam-2>0$ (if $N\geq12$), we have
\begin{align*}
2\int_{\delta}^{1/2}&\int_{\delta}^{\rho}
\frac{s^{m-1}}{\rho^{m-1}}P(s,t)
\,ds\,d\rho
\gtrsim
\frac{\sig^2}{t^2}
\int_{\delta}^{1/2}\int_{\delta}^{\rho}
\frac{s^{N-3\gam-3}}{\rho^{N-2\gam-1}}
\,ds\,d\rho\\
&=\frac{(\sigma t^{-1})^2}{N-3\gamma-2}\int_{\delta}^{1/2}\frac{1}{\rho^{N-2\gamma-1}}(\rho^{N-3\gamma-2}-\delta^{N-3\gamma-2})
\\&=\frac{(\sigma t^{-1})^2}{N-3\gamma-2}\left[\frac{\rho^{-\gam}}{\gam}-\frac{\del^{N-3\gamma-2}\rho^{-(N-2\gam-2)}}{N-2\gam-2}\right]_{\delta}^{1/2}
=\frac{\sigma t^{-1}}{\gam(N-2\gam-2)}+O(\sigma^2 t^{-2}).
\end{align*}
Thus,
by choosing sufficiently large $M_1$ and $T_3$ independent of $\sigma$, we obtain the result.
\end{proof}

As a result of the three lemmata above, we obtain the following
\begin{prop}
\label{prop-sub-12}
There exists $T>0$ and $\sigma>0$ such that the function
\[
\underline{\Psi}\in W^{1,\infty}_{\mathrm{loc}}((1,\infty); H^{1}_{0}(B_{1}^{m}))
\]
defined in Lemma~\ref{12-sub-lem-3} is a sub-solution of \eqref{eq-diff-12dim}.
In addition, for any $\ep>0$,
there exist $M>1$ and $c>0$ such that $\Psi>ct^{-1}\psi_1$ in
$\left(B_{1}^{m}\setminus B_{\ep}^m\right)\times(MT,\infty)$.
\end{prop}
\begin{proof}
Define $b=\frac{\gam+1}{\gam+2}a$. Then, there exists $\sigma_{0}>0$ depending only on $\gam$ and $N$
such that $\sig^{a}\log(\sig^{-1})\leq \sig^{b}$ for all $\sigma\in\left(0,\sig_0\right)$.
Then, the Young inequality and \eqref{daiji-2} imply
\begin{equation*}
\upsi_t \le C
t^{-2-\frac{a}{2}} (\sigma^{1-b}) (r^{-a\gamma}\sigma^{2b}) \psi_1 \le Ct^{-2-\frac{a}{2}}
\left[
  (1-b)\sig
  +br^{-a\gam/b}\sig^{2}
\right] \psi_1 \hquad \text{in $(\delta(t),1]$}
\end{equation*}
for any $t>T_{4,1}$ and $\sigma<\sigma_0$, where $C$ and
$T_{4,1}>0$ are independent of $\sigma$ and $t$.
Hence, by using \eqref{mikata-1}, \eqref{mikata-2} and \eqref{daiji-2},
we obtain
\begin{equation}
\label{upsit-1-a}    
\text{$\upsi_t \le C\sigma t^{-2-\frac{a}{2}}\psi_1+ t^{-a/2} P$ \quad in $B_{1}^m$ \quad for any $t>T_{4,2}$ and $\sigma<\sigma_0$}
\end{equation}
with some $T_{4,2}>T_{4,1}$ independent of $\sigma$ and $t$.
Now we combine the mean value theorem with Lemma \ref{12-sub-lem-1},
\eqref{bp1}, \eqref{12-sub-1} and \eqref{mikata-2} to get
\begin{align}\label{sub-12-nonlin}
\left(\frac{2N-4}{r^2}+ H\right)
r^{\gamma} F(r^{-\gamma}\uPsi)
&\le P+C\left(\frac{2N-4}{r^2}+ H\right)
F'(
    r^{-\gamma}
    (t^{-1}\sigma\psi_1+M_{1}\upsi_{+})
)\upsi_{+}
\notag
\\
&\le  P+ C\left(\sigma t^{-1}\right)^a P
\leq 3P/2 \quad \text{in $B_{1}^m$} 
\end{align}
for any $t>T_{4,3}$ and $\sigma<\sigma_{0}$
with some $T_{4,3}>T_{4,2}$ independent of $\sigma$.
Therefore, by combining \eqref{naiseki-1-a}, \eqref{upsit-1-a} and \eqref{sub-12-nonlin}, we obtain
\begin{align*}
\begin{split}
&\hspace{-30pt}
(\partial_t-\Delta-H)\uPsi
+\left(\frac{2N-4}{r^2}+ H\right)r^{\gamma}F(r^{-\gamma}\uPsi)
\\&\le
\left(
    M_{1}\left(2P,\psi_{1}\right)_{L^2}
    -{\sigma}t^{-2}
\right)\psi_1
+M_{1}\upsi_{t}
+\left({3}/{2}-2M_{1}\right)P
\\&
\leq
\sig t^{-2}
\left(
    CM_{1}\sig+{M_1}{T^{-a/2}}-1
\right)\psi_{1}
+
\left[
    3/2-M_{1}\left(2-{T^{-a/2}}\right)
\right]P
\le 0
\end{split}
\end{align*}
for any $t>T$ and $\sig<\sig_1$, where $T>T_{4,3}$ and $\sig_1<\sig_0$
are independent of $\sigma$ and $t$.

Finally, we fix $T>0$ above and show $\uPsi(r,T)\le \Psi(r,T)$ in $B_{1}^m$.
Since $r\geq\del(T)$ means
$
    r^{\gamma}(-\log r+\gamma^{-1})\ge
    \delta^{\gam}(-\log \delta+\gamma^{-1})
$,
it follows from
Lemma~\ref{12-sub-lem-1}~(ii) and Lemma \ref{12-sub-lem-3} that
\begin{equation*}
   r^{-\gamma} \uPsi(r,T)\le C(-\log r + \gamma^{-1})(-\log \delta+\gamma^{-1})^{-1}(1+CM_{1}\sig^{a}T^{-a})\psi_1(r).
\end{equation*}
On the other hand,
by the Hopf's lemma and Proposition~\ref{p-prop-1},
we have
\begin{equation*}
r^{-\gamma}\Psi(r,T)=\Phi(r,T)=V_{*}-u(r,T)\ge c(T)(-\log r + \gamma^{-1}) \psi_{1}(r).
\end{equation*}
By taking $\sigma$ sufficiently small (i.e., $-\log{\del}\gg 1$)
and using Lemma~\ref{12-sub-lem-3}, we obtain the result.
Therefore, we can deduce that $\uPsi$ is a sub-solution of \eqref{eq-diff-12dim}. The remaining assertion follows from a comparison principle and \eqref{daiji-1}.
\end{proof}
Note that for the case $N\ge 12$, Proposition \ref{prop-outer} follows from Propositions \ref{prop-super-12} and \ref{prop-sub-12}.
\subsection{Proof of Proposition \ref{prop-outer}
($\boldsymbol{N=11}$, $\boldsymbol{f\equiv f_H}$)}
\label{sec-rate-11}
In this subsection, we deal with the case $N=11$.
Note that in this case, $\gamma=3$ and $m=5$ are satisfied.
As before, we take any $T>0$ and set $\Psi:=r^{3}\Phi$ to yield
\begin{align}\label{eq-diff-11dim}
  \begin{dcases}
    \pl_{t}\Psi-\Del\Psi-H\Psi
    =-\left(\frac{2N-4}{r^{2}}+H\right)r^{3}F(r^{-3}\Psi)
    &\text{in }B_{1}^{5}\times(T,\infty),
    \\
    \Psi=0&\text{on }\pl B_{1}^{5}\times(T,\infty).
  \end{dcases}
\end{align}
We first construct a super-solution.
\begin{prop}
\label{prop-super-11}
Define $R(r,t):=r^{-2}(\log r)^{-2}\min\left\{t^{1/2}, r^{-3}(\log r)^{2}\right\}\psi_{1}^{2}$
and set $\widetilde{\psi}:=(t\log t)^{-2}\widetilde{\psi}_0$,
where $\widetilde{\psi}_0$ be a solution to
\begin{align*}
  \begin{cases}
    (-\Del-H)\widetilde{\psi}_{0}
    =\left(R,\psi_{1}\right)_{L^{2}\left(B_{1}^{5}\right)}\psi_{1}-R  \hquad \text{in } B_{1}^{5},
    \\
    \widetilde{\psi}_{0}=0 \hquad \text{on }\pl B_{1}^{5},
    \quad \Tilde{\psi}_0\perp \psi_1 \text{ in } L^{2}(B_{1}^5).
  \end{cases}
\end{align*}
Then, there exist $T, \kappa_1, \kappa_2>0$ such that
$\widetilde{\Psi}(r,t):=\kappa_1 (t\log t)^{-1}\psi_{1}(r)+\kappa_2 \widetilde{\psi}(r,t)$ is a super-solution to \eqref{eq-diff-11dim}.
In particular, there exists $C>0$
independent of $r$ and $t$
such that $\Psi\le C(t\log t)^{-1}\psi_{1}(r)$ in $B_1^5$ for any $t>T$.
\end{prop}
\begin{proof}
We define $\delta_1(t)$ as the solution to $\delta_{1}^{-3}(\log \delta_{1})^2=\sqrt{t}$.
Since 
\begin{align*}
  \pl_{t}R(r,t)
  =\frac{1}{2} t^{-1/2}r^{-2}(\log r)^{-2}\psi_{1}^{2}(r)\chi_{\left\{r<\delta_1(t)\right\}}(r)
\end{align*}
and $(r\log r)^{-2}\in L^{5/2,1}(B_{1}^5)$\footnote{See Appendix \ref{sec-app-m} for the definition of the Lorentz space},
we obtain $R(r,t)\in W^{1,\infty}_{\mathrm{loc}}((e,\infty);L^2(B_1^{5}))$.
Thus, we have $\widetilde{\psi}_{0}\in W^{1,\infty}_{\mathrm{loc}}((e,\infty);H^{1}_{0}(B_1^{5}))$.
Moreover, we have
\begin{align*}
-(\Del+H)\pl_{t}\widetilde{\psi}_{0}
=\left(\pl_{t}R,\psi_{1}\right)_{L^{2}\left(B_{1}^{5}\right)}\psi_{1}-\pl_{t}R\hquad \text{in } B_{1}^{5},
\quad \pl_{t} \widetilde{\psi}_{0}=0 \hquad \text{on }  \pl  B_{1}^{5}
\end{align*}
and $\partial_t\widetilde{\psi}_{0}\perp\psi_1$ in $L^{2}(B_{1}^5)$. Hence, 
by applying Lemma \ref{apenprop}, the elliptic regularity theory and the Hopf's lemma, we have $|\partial_t\widetilde{\psi}_{0}|\le C t^{-1/2} \psi_1$ in $B_{1}^{5}$
and $|\widetilde{\psi}_{0}|\le C t^{1/2} \psi_1$ in $B_{1}^5$
for any $t>e$, where $C>0$ is independent of $t$.
Moreover, we have
\begin{equation}
\label{naiseki-2-a}
(R(r,t), \psi_{1})_{L^2(B_{1}^5)}\ge
\int_{\delta_{1}(t)}^{1}r^{-1}\psi_1^{3}(r)\,dr\ge \theta \log t  
\end{equation}
for some $\theta>0$ depending only on $\psi$.
Let $\kap_{1}>0$ be a constant to be fixed later
and define $\kappa_2:=3\theta^{-1}\kappa_1$.
Then, by the above estimates, we have
\begin{equation}
\label{katei-1}
(t\log t)^{2}  |\partial_t \Tilde{\psi}|\le 
\frac{C}{t^{1/2}}\psi_1  \le  \frac{\theta}{2}\psi_1 \quad\text{and}\quad  \kappa_2|\widetilde{\psi}|\le \frac{\kappa_1}{2t\log t}\psi_1
\end{equation}
in $B_{1}^5$ for any $t>T_1$,
for some $T_1>e$ independent of $\kappa_1$.
In particular,
\eqref{katei-1} implies $F(r^{-3}\widetilde{\Psi})\ge F(r^{-3}\kappa_1(2t\log t)^{-1}\psi_1)$ in $B_{1}^5$ for $t>T_1$.
Moreover, we have 
\begin{align}\label{ababa}
    &\hspace{-10pt}
    \left(\frac{2N-4}{r^2}+H\right)(t\log t)^2 r^3
    F\left(
        \frac{\kappa_1 \psi_1}{2t\log t r^3}
    \right)
    \notag
    \\
    &\ge (t\log t)^2
    \min\left\{
        \frac{\kappa_1\psi_1}{2t\log t r^2},
        \frac{\kappa_{1}^2 \psi_1^2}{4(t\log t)^2 r^5}
    \right\}
    \geq
    \min\left\{
        \frac{\kappa_1 t^{1/2}}{2\norm{\psi_1}_{L^{\infty}}},
        \frac{\kappa_1^{2}}{4}
    \right\}R(r,t)
\end{align}
in $B_{1}^5\times (T_1,\infty)$.
We choose $T>T_1$ and $\kappa_1$ so that the right-hand side of \eqref{ababa} is greater than $\kappa_2 R$ 
for any $r\in [0,1]$, $t>T$ and 
$\kappa_1(2T\log T)^{-1}\psi_1\ge r^{3}(V_H-\phi_0)$ in $B_{1}^5$.
Thus \eqref{katei-1} yields $\Psi(r, T) \le \widetilde{\Psi}(r,T)$ in $B_1^{5}$.
Thanks to \eqref{naiseki-2-a}, \eqref{katei-1} and \eqref{ababa}, we obtain
\begin{align*}
&(t\log t)^{2}(\pl_{t}
  -\Del-H)\widetilde{\Psi}= 
  -\kappa_{1}(1+\log t)\psi_{1}
  +\kappa_{2}\left[
      \left(R,\psi_{1}\right)_{L^{2}}\psi_{1}
      -R+(t\log t)^2 \partial_t \Tilde{\psi}
  \right]
\\&\ge
-\kappa_{1}(1+\log t)\psi_1
-\kappa_2 R
+\kappa_2\left(
    \theta\log t
    -\frac{\theta}{2}
\right)\psi_1
\ge -\kappa_2 R
\\&\ge
-\left(\frac{2N-4}{r^2}+H\right)
(t\log t)^2 r^3
F\left(
    \frac{\kap_1 \psi_1}{2t\log t r^3}
\right)
\ge
-\left(\frac{2N-4}{r^2}+H\right)
(t\log t)^2 r^3 F(r^{-3}\widetilde{\Phi})
\end{align*}
in $B_{1}^{5}$ for any $t>T$.
Hence, $\widetilde{\Psi}$ is a super-solution to \eqref{eq-diff-11dim}.
the remaining assertion is obtained by a comparison principle and the fact that $|\widetilde{\psi}_{0}|\le Ct^{1/2}\psi_1$ in $B_{1}^5$.
\end{proof}
Next, we turn to the construction of a sub-solution. As in the case $N\ge 12$, we introduce three lemmata below.
\begin{lem}\label{11-sub-lem-1}
Fix $\sig\in(0,1/2)$. Then, the following hold.
  \begin{enumerate}
    \item[(i)] For any $t>e$,
        there exists the unique solution $\upsi\in W^{1,\infty}_{\mathrm{loc}}((e,\infty); H^{1}_{0}(B_{1}^{5}))$ of 
        \begin{equation}
        \label{problem-sub-2}
        (-\Delta-H)\upsi=-2P+(2P,\psi_1)\psi_1 \hquad \text{in $B_1^{5}$,} 
        \quad \upsi\perp \psi_1\hspace{2mm} \text{in $L^2(B_{1}^{5})$} 
        \end{equation}
        with 0-Dirichlet boundary condition,
        where
        \begin{equation*}
        P=
        P(r,t):=\left(\frac{2N-4}{r^2}+H\right)r^{\gamma}F(r^{-\gamma}\sigma (t\log t)^{-1}\psi_{1}).
        \end{equation*}
    \item[(ii)]
    We set $a:=1/12$ and define $\delta=\delta(t)$ as a unique solution of
    \begin{equation}\label{defdel-11}
        \delta(t)^{3}
        =\frac{1}{2}\sigma\psi_1(\del(t))(t\log t)^{-1}. 
    \end{equation}
    Then, there exists  $C>0$ is independent of $\sigma$ and $t$ such that 
    \begin{equation*}
        0\le (P,\psi_1)_{L^{2}(B_{1}^5)}
        \le C \sigma^2 |\log \sigma|(\log t)^{-1} t^{-2},
        \quad \upsi \le C\sigma^{1+a} t^{-1-a}\psi_1
        \quad \text{in $B_{1}^5$}
    \end{equation*}
    for any $t>e$ and
    \begin{equation*}
        0\le -(P_t, \psi_1)_{L^2}
        \le C\sigma^2 |\log \sigma|(\log t)^{-1} t^{-3},
        \hquad -\upsi_t \le C \sigma^{1+a} t^{-2-a}\psi_1
        \quad \text{in $B_{1}^5$}
    \end{equation*}
    for a.e. $t>e$.
    \item[(iii)]
    For any $\ep>0$,
    there exist $C(\ep)>0$ and $T_1(\ep)>0$ independent of $\sigma$ such that 
    \begin{equation*}
        -\upsi\le  C(\ep) \sigma^{1+a} t^{-1-a}\psi_1, \quad  
        \upsi_t \le C(\ep) \sigma^{1+a} t^{-2-a} \psi_1\hquad
        \text{in $B_{1}^5\setminus B_{\ep}^5$}
    \end{equation*}
 for a.e. $t>T_{1}(\ep)$.
  \end{enumerate}
\end{lem}
\begin{proof}
We first prove the assertion (i). By a direct computation, we obtain 
\begin{equation*}
    P_{t}=-\sigma 
    (t\log t)^{-2}(1+\log t)
    \left(\frac{2N-4}{r^2}+H\right)
    F'(\frac{\sigma}{r^3 t\log t}\psi_1)\psi_1,
    \hspace{2mm} F'(u)=1-e^{-u_+}.
\end{equation*}
Note that $r^{-2}\in L^2(B_{1}^5)$.
Then, we use \eqref{bp1} to see $P\in W^{1,\infty}_{\mathrm{loc}}((1,\infty); L^2(B_{1}^5))$. Hence, we obtain
$\upsi\in W^{1,\infty}_{\mathrm{loc}} ((1,\infty); H^{1}_{0}(B_{1}^5))$. Next, we prove the assertions 
(ii) and (iii).
It follows from Lemma~\ref{fnolem} that \eqref{defdel-11} has a unique
solution $\delta$ such that
\begin{equation}
\label{kankei-11}
    r^3\le \frac{\sigma}{2}\psi_1(r)(t\log t)^{-1}
    \quad\iff \quad
    r<\delta(t).
\end{equation}
By \eqref{bp1} and \eqref{bp2}, we obtain
\begin{align}
\label{naiseki-2}
\begin{split}
\quad 0\le&(P,\psi_1)_{L^2(B_1^{5})}\lesssim \sigma
\int_{0}^{\delta(t)}\frac{r^2}{t\log t}
  \psi_{1}\,dr
  +\sigma^2 \int_{\delta(t)}^{1}\frac{\psi_1^{2}}{r(t\log t)^2}\,dr\\
&\lesssim C(\sigma^2 (t\log t)^{-2}+t^{-2}(\log t)^{-1}\sigma^2|\log \sigma|)\lesssim \sigma^2 |\log\sigma| t^{-2}(\log t)^{-1}
\end{split}
\end{align}
and
\begin{align*}
    \norm{P}_{L^2(B_{1}^5)}^2&\lesssim
    \int_{0}^{\del(t)}\sig^{2}(t\log t)^{-2}\psi_{1}^{2}
    \,dr+
    \int_{\del(t)}^{1}
    \sig^{4}t^{-4}r^{-6} \psi^4\,dr\\
    &\lesssim
    \sig^{2}(t\log t)^{-2}\del(t)\lesssim (\sigma t^{-1})^{1+a}
\end{align*}
for all $t>e$.
Therefore, by the Poincar\'e inequality and an energy estimate, we obtain
\begin{align*}
\norm{\upsi}_{L^{2}(B_{1}^5)}^{2}\lesssim\norm{P}_{L^{2}}\norm{\psi}_{L^{2}}
  +(P,\psi)_{L^2}\norm{\psi}_{L^2}\norm{\psi_1}_{L^2}\lesssim  C(\sigma t^{-1})^{1+a} \norm{\upsi}_{L^2(B_{1}^5)}
\end{align*}
for all $t>e$. Thus, we get
$
 \norm{\psi}_{L^{2}(B_{1}^5)}\lesssim C(\sigma t^{-1})^{1+a}
$
for all $t>e$.
Then, it follows from \cite[Proposition 47.6]{QSbook} and  \eqref{naiseki-2} that
\begin{equation*}
\upsi\lesssim (P,\psi_{1})_{L^2(B_{1}^{5})} \lVert \psi_1\rVert_{L^{\infty}(B_{1}^5)}+\lVert \upsi_{+}\rVert_{L^2(B_{1}^5)}\lesssim \sigma^{1+a} t^{-1-a}  \hquad \text{in $B_{1}^5$}
\end{equation*}
for any $t>e$.
Moreover, let $\ep>0$. Then, there exists $T_{1}(\ep)$ independent of $\tau$ so that
\begin{equation*}
  \norm{\upsi}_{C^{1}(\overline{B_{1}^5}\setminus B_{\ep}^5)}
  \lesssim_{\ep} (\norm{P}_{L^{\infty}(B_{1}^5\setminus B_{\ep}^5)}
  +\norm{\upsi}_{L^2(B_{1}^5)})
  \lesssim_{\ep} \sigma^{1+a} t^{-1-a}
\end{equation*}
for any $t>T_{1}$.
Therefore, we obtain
\begin{equation}
\label{11-sub-1}
\upsi\lesssim \sigma^{1+a} t^{-1-a} \hquad\text{in $B_{1}^5\times (1,\infty) $} \hquad \text{and} \hquad |\upsi|\lesssim_{\ep} \sigma^{1+a} t^{-1-a}\psi_1 \hquad \text{in $B_{1}^5 \setminus B_{\ep}^5 \times (T_1,\infty)$}.
\end{equation}
The remaining assertions follow by using the following and a similar method:
\begin{equation*}
(-\Delta -H)\upsi_{t}=-2P_t+(2P_t,\psi_1)\psi_1 \hquad
\text{in $B_{1}^{5}$},
\hquad \upsi_t=0 \hquad \text{on $\partial B_1^{5}$,}
\hquad \text{$\partial_t\upsi \perp\psi_1$ in $L^{2}(B_{1}^5)$}.
\end{equation*}
\end{proof}

\begin{lem}\label{11-sub-lem-2}
There exist $C_{2}>0$ and $T_2>0$ independent of $\sigma$ and $t$
such that
\begin{equation}
\label{daiji-3}
-\upsi(r)\le C_{2}\sigma\frac{-\log r}{t\log t}
\hspace{2mm} \text{in $(0,\delta(t)]$,}
\hspace{2mm} -\upsi(r)\le
C_2 \frac{\sigma^{1+a}}{t^{1+a}r^{3a}}
\log(t\sigma^{-1})\psi_1
\hspace{2mm} \text{in $(\delta(t),1]$}
\end{equation}
for any $t>T_2$ and
\begin{equation}
\label{daiji-4}
\upsi_t(r)\le -\frac{\sigma}{t^2\log t}\log r \hspace{2mm} \text{in $(0,\delta(t)]$}, \hspace{2mm} \upsi_t(r)\le \frac{\sigma^{1+a}}{t^{2+a}r^{3a}} \log (t\sigma^{-1})\psi_1 \hspace{2mm} \text{in $(\delta(t),1]$}
\end{equation}
for a.e. $t>T_{2}$.
\end{lem}
\begin{proof}
We only show \eqref{daiji-3}. By \eqref{bp1} and \eqref{kankei-11}, we have
\begin{equation}
\label{mikata-3}
\left(\frac{2N-4}{r^2}+H\right)\frac{\sigma}{2t\log t} \psi_1\le P\le \frac{2N-4+H}{r^2 t\log t}
\sigma\psi_1 \quad \text{in $(0,\delta(t))$}
\end{equation}
and
\begin{align}
\label{mikata-4}
\left(\frac{2N-4}{r^2}+H\right)\frac{\sigma^2}{2e^2 r^3(t\log t)^{2}}\psi_{1}^2\le 
P&\le \frac{2N-4+H}{2r^5(t\log t)^2}\sigma^2 \psi_1^{2}\quad \text{in $(\delta(t),1)$.}
\end{align}
As a result, by a similar method to that in the proof of \eqref{daiji-1}, we can deduce that there exist $\ep>0$, $T_{2,1}>T_1$ and $C_2$ independent of $\sigma$ so that $(-\Delta-H)(\eta_1-\upsi)\le 0$ in $B_{\ep}^{5}\setminus B_{\del}^{5}$ and $\eta_1<\upsi$ at $r=\ep$ for any $t>T_{2,1}$, where
\begin{equation*}
    \eta_1=\eta_{1}(r,t)
    :=\sigma C_2 (t\log t)^{-1}\log r.
\end{equation*}
By the maximum principle,
we have $\eta_1\le \upsi$ in $B_{\ep}^5$ for any $t>T_{2,1}$. It implies the first inequality of \eqref{daiji-3}.
Next, by \eqref{defdel-11}, there exist $C_3>C_2$ independent of $\sigma$ such that $\eta_2<\eta_1\le \upsi$ at $r=\delta(t)$ for any $t>T_{2,1}$, where $\eta_2:=-C_3\sigma^{1+a} \log (\sigma^{-1}t) t^{-1-a}r^{-3a}$.
In addition, it follows from \eqref{mikata-3} and \eqref{mikata-4} that there exist $\ep>0$ and $T_{2,2}>T_{2,1}$ independent of $t$ and $\sigma$ that 
\begin{align*}
       \frac{-t^{1+a}}{C_{3}\sigma^{1+a}\log (\sigma^{-1}t)}&(-\Delta-H)\eta_2
    = 9a(1-a)r^{-2-3a}-H r^{-3a}\\
    &\ge 9a(1-a)\frac{\delta(t)^{3(1-a)}}{2}r^{-5}
        \ge \frac{-2t^{1+a}}{C_3 \sigma^{1+a} \log(\sigma^{-1}t)}P
\end{align*}
in $B_{\ep}^{5}\setminus B_{\delta(t)}^{5}$ for $t>T_{2,2}$.
Moreover, by Lemma \ref{11-sub-lem-1}, we obtain
$\eta_2<\upsi$ at $r=\ep$ for $t>T_{2,3}$ with some $T_{2,3}>T_{2,2}$ independent of $\sigma$. Therefore, by the maximum principle and \eqref{11-sub-1}, we obtain the latter inequality of \eqref{daiji-3}. Finally, we note that \eqref{daiji-4} is obtained by a similar argument.
\end{proof}

\begin{lem}\label{11-sub-lem-3}
Define $\uPsi:= \sigma (t\log t)^{-1}\psi_1+ M_{1}\upsi$,
where $\upsi$ is as in Lemma~\ref{11-sub-lem-1}.
There exist $M_{1}, T_3>2$ independent of $\sig$
such that $\uPsi\leq 0$ in $B_{\delta(t)}^5$ for all $t>T_3$.
\end{lem}
\begin{proof}
It follows from \eqref{problem-sub-2} and
Lemmata \ref{11-sub-lem-1}, \ref{11-sub-lem-2} that $\Delta \upsi\ge 0$ around $r=0$ for any fixed $t$. Hence, by using Lemmata \ref{11-sub-lem-1}, \ref{11-sub-lem-2} again and
the same argument with Lemma \ref{apriorilem1}, we arrive at $r^{4}\upsi' \downarrow 0$ as $r\downarrow 0$. 
Hence,
\begin{equation*}
\upsi(r)
=\upsi(1/2)
+\int_{r}^{1/2}\int_{0}^{\rho}
\frac{s^4}{\rho^4}
(H\upsi-2P(s,t)+(2P,\psi_1)\psi_1)\,ds\,d\rho.
\end{equation*}
It follows from Lemma \ref{11-sub-lem-1} that
\begin{equation*}
\upsi(r)\le
-2\int_{\delta}^{1/2}\int_{\delta}^{\rho}
\frac{s^4}{\rho^4}P(s,t)
\,ds\,d\rho + O(\sigma^{1+a} t^{-1-a})+ O(\sigma^{2} t^{-2})
\end{equation*}
in $B_{\delta}^m$. By using \eqref{mikata-1} and \eqref{mikata-2}, we have
\begin{align*}
&2\int_{\delta}^{1/2}\int_{\delta}^{\rho}
\frac{s^4}{\rho^4}
P(s,t)\,ds\,d\rho
\gtrsim\frac{\sigma^2}{(t\log t)^{2}}
\int_{\delta}^{1/2}\rho^{-4}
\int_{\delta}^{\rho} s^{-1}
\,ds\,d\rho\\
&=
\frac{\sigma^2}{(t\log t)^{2}}
\left(
  \frac83\log{2}-\frac89+\frac83\log{\del}+\frac19\del^{-3}
\right)
\gtrsim
\frac{\sigma^2}{(t\log t)^{2}\delta^{3}}
\gtrsim\frac{\sigma}{t\log t}.
\end{align*}
As a result, the conclusion follows.
\end{proof}
\begin{prop}
\label{prop-sub-11}
There exists $T>0$ and $\sigma>0$ such that the function
\[
\underline{\Psi}\in W^{1,\infty}_{\mathrm{loc}}((e,\infty)\,;H^{1}_{0}(B_1^5))
\]
defined in Lemma \ref{11-sub-lem-3} is a sub-solution of \eqref{eq-diff-11dim}. As a result, for any $\ep>0$, there exist $M>1$ and $c>0$ such that $\Psi>c(t\log t)^{-1}\psi_1$ in $B_{1}^5 \setminus B_{\ep}^5$ for any $t>MT$.
\end{prop}

\begin{proof}
We take $b=4a/5$. Then, there exists $\sigma_{0}$ independent of $\sigma$ such that
$\sigma^{a}\log (\sigma^{-1})\le \sigma^b$ for any $\sigma<\sigma_0$. Then,
it follows from \eqref{daiji-4} and \eqref{mikata-3} that 
\begin{align*}
\upsi_t \le C \sig^{1+b}t^{-2-\frac{a}{2}}r^{-3a}\psi_1
&\leq
Ct^{-2-\frac{a}{2}}
\left[
  (1-b)(\sig^{1-b})^{\frac{1}{1-b}}
  +b (r^{-3a}\sig^{2b})^{1/b}
\right]\psi_1\notag
\\&
\le C t^{-2-\frac{a}{2}}(\sigma+r^{-3a/b}\sigma^2)\psi_1 \quad \text{in $B_{1}^5\setminus B_{\delta}^5$}
\end{align*}
for any $t>T_{4,1}$ and $\sigma<\sigma_0$, where $C>0$ and $T_{4,1}>0$ are independent of $\sigma$ and $t$.
Combining it with \eqref{daiji-4}, \eqref{mikata-3} and \eqref{mikata-4},
there exists $T_{4,2}>0$ such that
\begin{equation}
\label{upsit-11}
\upsi_t \le \sigma t^{-2-\frac{a}{2}}\psi_1+ t^{-\frac{a}{4}} P(r,t)\quad \text{in $B_{1}^5$}
\end{equation}
for all $t>T_{4,2}$ and $\sigma<\sigma_0$.
Moreover, by combining the mean-value theorem with Lemma \ref{11-sub-lem-1},
\eqref{mikata-3}, \eqref{mikata-4} and \eqref{bp1}, we obtain
\begin{align*}
&\left(\frac{2N-4}{r^2}+H\right)
r^{3} F(r^{-3}\uPsi)
\le P+
M_1 \left(\frac{2N-4}{r^2}+ H\right)
F'\left(
    r^{-3}\left(
        \frac{\sigma \psi_1}{t\log t}+M_{1}\upsi_{+}
\right)\right)
\upsi_{+}
\\
&\le P+ C(\sigma t^{-1})^{1+a}r^{-2}\psi_1
\min\left\{1, r^{-3}\left(\frac{\sigma }{t\log t}+ M(\sigma t^{-1})^{1+a}\right)\psi_1 \right\}\le \frac{3P}{2}
\end{align*}
for any $t>T_{4,3}$, where $T_{4,3}>T_{4,2}$ is independent of $\sigma$.
Therefore, by using \eqref{upsit-11} and Lemma \ref{11-sub-lem-1}, we deduce that there exists $T>T_{4,3}$ and $\sigma_1<\sigma_0$ such that  
\begin{align*}
(&\partial_t-\Delta-H)\uPsi +\left(\frac{2N-4}{r^2}+ H\right)r^{3}F(r^{-3}\uPsi)\\
&\le -\sigma (t\log t)^{-2}(1+\log t)\psi_1+M_{1}\upsi_{t}+ M_{1}(2P,\psi_1)_{L^2}\psi_1+(\frac{3}{2}-2M_1) P\\
&\le -\frac{\sigma}{t^2\log t}\psi_1(1+(\log t)^{-1}-2M_1 C\sigma|\log \sigma|-M_1 t^{-a/2}(\log t))\\
&\qquad +\left(\frac{3}{2}-2M_1+M_1 t^{-a/4}\right) P\\
&
\le 0\qquad \text{in $B_{1}^5$}
\end{align*}
for any $t>T$ and $\sigma<\sigma_1$. Finally, we fix $T>0$ above and show $\uPsi(r,T)\le \Psi(r,T)$ in $B_{1}^5$. Since 
\begin{equation*}
    r^{3}\left(-\log r+{1}/{3}\right)\geq
    \delta^{3}\left(-\log \delta+{1}/{3}\right)
    \quad\text{for any $r>\delta(t)$,}
\end{equation*}
by using Lemmata \ref{11-sub-lem-1} and \ref{11-sub-lem-3}, we deduce that
\begin{equation*}
   r^{-3} \uPsi(r,T)
   \le C\left(-\log r + \frac{1}{3}\right)
   \left(-\log \delta+\frac{1}{3}\right)^{-1}
   \left(1+\frac{CM_{1}\sigma^a\log T}{T^{a}}\right)\psi_1 \quad \text{in $B_{1}^5$.}
\end{equation*}
Moreover, by the Hopf's lemma and Proposition \ref{p-prop-1}, we have
\begin{equation*}
r^{-3}\Psi(r,T)=\Phi(r,T)=V_{*}-u(r,T)\ge c(T)
\left(-\log r + {1}/{3}\right) \psi_{1}.
\end{equation*}
Now, taking $\sigma>0$ sufficiently small, we obtain the result. Therefore, we can deduce that $\uPsi$ is aw sub-solution of \eqref{eq-diff-11dim}. 
The remaining assertion follows from a comparison principle and \eqref{daiji-3}.
\end{proof}
Note that in the case $N=11$, Proposition \ref{prop-outer} follows from Propositions \ref{prop-super-11} and \ref{prop-sub-11}.

\bigskip
 \noindent
 {\bf Acknowledgements}\\
 KK was supported by JSPS KAKENHI Grant Number 26KJ0083. YO was supported by JSPS KAKENHI Grant Number 26KJ0018. We thank Sho Katayama for bringing the result of \cite{KM25} to our attention and for fruitful discussions on the analyticity argument.

\appendix

\section{Some basic lemmata}
We first quote the improved Hardy inequality obtained in \cite{BV}.
\begin{lem}
Let $\lambda_1>0$ be the first eigenvalue of $-\Delta_{D}$ in $B_{1}^{2}$. Then,
\label{hardy}
\begin{equation*}
\int_{B_r}|\nabla \xi|^2 -\frac{(N-2)^2}{4|x|^2}\xi^2\,dx\ge \frac{\lambda_1}{r^2}\int_{B_r}\xi^2\,dx\quad \text{for all $\xi\in C^{1}_{c}(B_r)$.}
\end{equation*}
\end{lem}
Next, we introduce a lemma which gives an optimality of the Hardy inequality.
\begin{lem}
\label{optimality-1}
For any $\ep>0$, there exist sequences $r_i$ and $\xi_i\in C^{0,1}_{c}(B_1)$ satisfying $r_{i+1}<r_i$ for any $i\in \N$ and $\mathrm{supp}(\xi_i)=[r_{i+1},r_i]$ such that $r_{i}\downarrow 0$ as $i\to \infty$ and 
\begin{equation*}
    \int_{B_1}|\nabla \xi_i|^2 -\frac{(N-2)^2+\ep}{4|x|^2}\xi^{2}_{i}<0 \quad\text{for all $i\in \N$.}
\end{equation*}
\end{lem}

\begin{proof}
We use the method in \cite{M14}. Define $\xi_i=r^{(2-N)/2}\sin (\frac{1}{2}\ep_{0}^{1/2} \log r)\chi_{[r_{i+1}, r_{i}]}$ with $r_{i}=e^{-2\pi i/\ep^{1/2}_{0}}$ and $\ep_{0}=\ep/2$. Then, the result follows from the
equation
\begin{equation*}
-\Delta \xi_{i}=\frac{(N-2)^2+\ep_0}{4r^2}\xi_{i} \quad \text{in $B_{r_i}\setminus \overline{B_{r_{i+1}}}$}.
\end{equation*}
\end{proof}

We next introduce the following intersection result. Since the proof is the same as that of Proposition \cite[Lemma 5.3]{KK25}, we omit the proof.
\begin{lem}
\label{intersec-lem-2}
Let $w_1$ and $w_2$ be radial (possibly singular) solutions of \eqref{eqofw} or \eqref{eqofW}. Assume that $w_1$ and $w_2$ are unstable in $B_{r_2}\setminus B_{r_1}$ for $0\le r_1<r_2<1$, where $B_{0}=\emptyset$. 
Then, $w_1(r)=w_2(r)$ for some $r\in (r_1,r_2)$.
\end{lem}
Finally, we obtain the following
\begin{lem}
\label{odelem}
Let $0=r_0<r_1<\cdots r_i=1$ with $i\ge 2$.
Assume that $b(r)\in C^{1}(0,1]$ and $b'<0$ for $(0,1]$. Let $z\in C^2(0,1]$ satisfies 
\begin{equation*}
    z''=b(r)z \quad\text{in $[0,1]$,}\quad  (-1)^j z>0 \quad\text{on $r_j<r<r_{j+1}$}
\end{equation*}
for any $0\le j\le i$. 
Then, for any fixed $j\ge 1$, 
\begin{equation*}
  \hat{z}(r):=(-1)^{j-1} (z(r_j+r)+z(r_j-r))>0 \quad \text{for $0<r<r_{j}-r_{j-1}$}.
\end{equation*}
In particular, we obtain $r_{j}-r_{j-1} \ge r_{j-1}-r_{j-2}$ for any $2\le j\le i$.
\end{lem}
\begin{proof}
We note that $\hat{z}$ satisfies $\hat{z}(0)=\hat{z}'(0)=0$ and 
\begin{equation*}
\hat{z}''=b(r_{j}+r)\hat{z}-(b(r_{j}+r)-b(r_{j}-r)) (-1)^{j-1}z(r_j-r)> b(r_{j}+r)\hat{z} \hspace{2mm}\text{for $0<r<r_{j}-r_{j-1}$.}
\end{equation*}
As a result, we obtain $\hat{z}>0$ for $0<r<r_{j}-r_{j-1}$. Moreover, we 
claim that $r_{j+1}-r_{j}\le r_{j}-r_{j-1}$. Indeed, if not, we have
$0<\hat{z}(r_{j}-r_{j-1})=(-1)^{j-1}z(r_{j}+(r_{j}-r_{j-1}))<0$, which is a contradiction.
\end{proof}

\section{Definitions of non-integer dimensions and Lorentz spaces}
\label{sec-app-m}
In this appendix, we first introduce the notion of non-integer dimensions. For a radial function $u\in C^2[0,1]$ and 
$f\in C^{0}[0,1]$, we say that $u$ satisfies
$-\Delta u=f$ in $B_{1}^m$ if $(-r^{m-1}u')'=r^{m-1}f$ for all $r\in [0,1]$. Moreover, for $1<q<\infty$,
the spaces $L^q(B_{1}^m)$ and $W^{1,q}(B_{1}^m)$ are defined by
$$
L^q(B_{1}^m):=\left\{u: (0,1)\mapsto \mathbb{R}; \norm{u}_{L^q(B_{1}^m)}= \left(\int_{0}^{1}m\omega_m r^{m-1} |u|^{q}\,dr \right)^{1/q}<\infty \right\};
$$
$$
W^{1,q}(B_{1}^m):=\left\{u\in L^{q}(B_{1}^m); u_{r}\in L^{q}(B_{1}^m)\right\}.
$$
We recall that $\omega_m$ is defined by $\om_m:=\frac{\sqrt{\pi}^{m}}{\Gam\left(\frac{m}{2}+1\right)}$, where $\Gam$ is the Gamma function. Moreover, we note that $L^2(B_{1}^m)$ is a Hilbert space equipped with $$
(f,g)_{L^2(B_{1}^m)}=\int_{0}^{1} m\omega_m r^{m-1}f(r)g(r)\,dr.
$$
We remark that the standard elliptic existence and regularity theories remain valid in non-integer dimensions. 
The standard comparison principle for the parabolic equations also holds in non-integer dimensions. 

Next, we introduce Lorentz spaces for the integer dimension $N$. For a measurable function $f:B_1\to\R$, we define $\al_{f}:(0,\infty)\to\left[0,\infty\right]$ and $f^{\ast}:(0,\infty)\to\left[0,\infty\right]$ as 
\begin{align*}
  \al_{f}(\sig):=
  \left|
    \left\{ x\in B_{1}; |f(x)|>\sig
    \right\}
  \right|, \quad  f^{\ast}(\lam)
  :=
  \sup
  \left\{
    \sig\in (0,\infty);
    \al_{f}(\sig)
    >\lam
  \right\}.
\end{align*}
Here we define $\sup\varnothing=0$.
Let $p\in(0,\infty)$ and $q\in(0,\infty]$.
The Lorentz space $L^{p,q}(B_1)$ is defined as the set of measurable functions $f$ such that $\norm{f}_{L^{p,q}(B_1)}<\infty$, where
\begin{align*}
    \norm{f}_{L^{p,q}(B_1)}
    :=
    \left(
        \int_{0}^{\infty}
        \left(
          \lam^{\frac{1}{p}}
          f^{\ast}(\lam)
        \right)^{q}
        \,\frac{d\lam}{\lam}
    \right)^{1/q}.
\end{align*}
Finally, we quote the following result \cite[Lemma 2.1]{AR96}.
\begin{lem}
\label{apenprop}
Let $N\ge 3$ and
$f\in L^{\frac{N}{2},1}(B_1)$.
Assume that $u\in H^{1}_{0}(B_{1})$ is a solution of 
$$
-\Delta u=f(x) \quad \text{in $B_1$}.
$$
Then, $u\in C^{0}(B_1)$ and there exists $C>0$ depending only on $N$ such that 
$$
\norm{u}_{L^{\infty}(B_1)}\le C\lVert f\rVert_{L^{\frac{N}{2},1}(B_1)}.
$$
\end{lem}

\section{List of symbols}
We collect some frequently used notation for convenience.
\subsection*{Constants}
\begin{tabular}{|c|c|}
\hline
Symbols 
& Descriptions 
\\[6pt] 
\hline
  $\lam^{\ast}$ & supremum of $\lambda$ for which \eqref{eq-intro-2} has a radial solution.
  \\
 $\alpha(\beta)$ & $\alpha(\beta)=v(0,\beta)=\beta-\log \lambda(\beta)$.\\
   $\lam_{\ast}$ & parameter for which \eqref{eq-intro-2} has the singular solution.
  \\
  $\lam_{h}$ & parameter $\lam_{\ast}$ for the case $f=f_h$. $\lam_{h}=2N-4+h$.\\
$H$ & constant defined in \eqref{DefH}.
\\
$\mu_1$ & first eigenvalue of $\cL$.\\
$\gam>0$& $\gam=\frac{1}{2}\left(N-2-\sqrt{(N-2)(N-10)}\right)$.\\
  $m$ & $m=N-2\gam$.   
  \\[6pt]
  \hline
\end{tabular}

\subsection*{Specific functions}
\begin{tabular}{|c|c|}
\hline
Symbols 
& Descriptions 
\\[6pt] 
\hline
$\phi_0(r)$ & unique solution to \eqref{Defphi0}. \\
 $(\lambda(\beta), v(r,\alpha(\beta)))$& solution to \eqref{eq-intro-2} described by $\beta$.\\
 $v_{\lambda}(r)$ & unique stable solution to \eqref{eq-intro-2} (defined for the case $\lambda<\lambda^*$).\\
$v^{*}(r)$ & unique solution to \eqref{eq-intro-2} with $\lambda=\lambda^{*}$.\\
$V_{*}(r)$ &  unique radial singular solution of \eqref{eq-intro-2}.\\ 
$f_h(r)$ &  function defined in \eqref{Deffh}. $f_0=0$. \\
$V_{h}(r)$ &  radial singular solution of \eqref{eq-intro-2} with $f=f_h$.\\
$\Phi(r,t)$ & $\Phi=V^{*}-u$.\\
$F(u)$ & $F(u)=e^{-u_{+}}-1+u_{+}$.\\
  $\phi_{1}(r)$ & first eigenfuntion of $\cL$ with $\norm{\phi}_{L^2(B_1)}=1$ and $\phi_{1}>0$ in $B_1$.\\
$w(r,\beta)$ & unique solution of \eqref{eqofw}. $w=v(r,\alpha(\beta))+\log \lambda$.\\
$w_0(r,\beta)$&  solution of \eqref{eqofw} for the case $f=0$.\\
$W(r)$ & unique solution of \eqref{eqofW}. $W=V_{*}+\log \lambda_{*}$.\\
$W_{h}(r)$ & solution of \eqref{eqofW} for $f=f_h$.\\
$K(r)$ & function defined in \eqref{defk}.\\  
  $\psi_{1}(r)$& function defined in Lemma~\ref{b-v-lem}.
  \\
  $w_{0}(r,\beta)$ & solution to \eqref{eqofw} with $f=0$.
  \\
  $W_{0}(r)$ & solution to \eqref{eqofW} with $f=0$.
  \\
  $\Psi(r,t)$ & $\Psi(r,t)=r^{\gamma}\Phi(r,t)$.\\
$\Tilde{\Psi}(r,t)$ & super-solutions constructed in Propositions \ref{prop-super-12} and \ref{prop-super-11}.\\
$\Tilde{\psi}(r,t)$ & correction terms for the super-solutions in Propositions \ref{prop-super-12} and \ref{prop-super-11}.\\
$\uPsi(r,t)$ & sub-solutions constructed in Propositions \ref{prop-sub-12} and \ref{prop-sub-11}.\\
$\upsi(r,t)$ & correction terms for the sub-solutions in Propositions \ref{prop-sub-12} and \ref{prop-sub-11}.\\
$\delta(t)$ & function defined separately in Lemmata \ref{12-sub-lem-1} and \ref{11-sub-lem-1}. \\
$P(r,t)$ & function defined separately in Lemmata \ref{12-sub-lem-1} and \ref{11-sub-lem-1}.
  \\[6pt]
  \hline
\end{tabular}
\subsection*{Operators}
\begin{tabular}{|c|c|}
\hline
Symbols 
& Descriptions 
\\[6pt] 
\hline
  $\cL$& $\cL=-\Del-\lam_{\ast}e^{V_{\ast}}$.
  \\
  $\cL_{h}$ & $\cL_{h}=-\Del-\lam_{h}e^{V_{h}}=-\Del-(2N-4)r^{-2}-h$.
  \\[6pt]
  \hline
\end{tabular}
{\small
\bibliographystyle{abbrv}
\bibliography{Grow_up_rate}

@article {BCMR,
    AUTHOR = {Brezis, Ha\"im and Cazenave, Thierry and Martel, Yvan and
              Ramiandrisoa, Arthur},
     TITLE = {Blow up for {$u_t-\Delta u=g(u)$} revisited},
   JOURNAL = {Adv. Differential Equations},
  FJOURNAL = {Advances in Differential Equations},
    VOLUME = {1},
      YEAR = {1996},
    NUMBER = {1},
     PAGES = {73--90},
      ISSN = {1079-9389},
   MRCLASS = {35K60 (35J65)},
  MRNUMBER = {1357955},
MRREVIEWER = {Song\ Mu\ Zheng},
}

@article {AR96,
    AUTHOR = {Alberico, Angela and Ricciardi, Tonia},
     TITLE = {Continuity properties for linear elliptic equations with
              lower-order terms},
   JOURNAL = {Rend. Accad. Sci. Fis. Mat. Napoli (4)},
  FJOURNAL = {Societ\`a{} Nazionale di Scienze, Lettere e Arti in Napoli.
              Rendiconto dell'Accademia delle Scienze Fisiche e Matematiche.
              Serie IV},
    VOLUME = {63},
      YEAR = {1996},
     PAGES = {7--16},
      ISSN = {0370-3568},
   MRCLASS = {35J25 (35B45 35J70)},
  MRNUMBER = {1451839},
MRREVIEWER = {Ya.\ A.\ Ro\u itberg},
}

@article {BV,
    AUTHOR = {Brezis, Haim and V\'azquez, Juan Luis},
     TITLE = {Blow-up solutions of some nonlinear elliptic problems},
   JOURNAL = {Rev. Mat. Univ. Complut. Madrid},
  FJOURNAL = {Revista Matem\'atica de la Universidad Complutense de Madrid},
    VOLUME = {10},
      YEAR = {1997},
    NUMBER = {2},
     PAGES = {443--469},
      ISSN = {0214-3577},
   MRCLASS = {35J65 (35B05 35P30 35Q55)},
  MRNUMBER = {1605678},
MRREVIEWER = {C.\ A.\ Swanson},
}

@article {BN87,
    AUTHOR = {Budd, C. and Norbury, J.},
     TITLE = {Semilinear elliptic equations and supercritical growth},
   JOURNAL = {J. Differential Equations},
  FJOURNAL = {Journal of Differential Equations},
    VOLUME = {68},
      YEAR = {1987},
    NUMBER = {2},
     PAGES = {169--197},
      ISSN = {0022-0396,1090-2732},
   MRCLASS = {35J65},
  MRNUMBER = {892022},
MRREVIEWER = {H.\ J.\ Kuiper},
       DOI = {10.1016/0022-0396(87)90190-2},
       URL = {https://doi-org.utokyo.idm.oclc.org/10.1016/0022-0396(87)90190-2},
}

@article {DF07,
    AUTHOR = {Dolbeault, Jean and Flores, Isabel},
     TITLE = {Geometry of phase space and solutions of semilinear elliptic
              equations in a ball},
   JOURNAL = {Trans. Amer. Math. Soc.},
  FJOURNAL = {Transactions of the American Mathematical Society},
    VOLUME = {359},
      YEAR = {2007},
    NUMBER = {9},
     PAGES = {4073--4087},
      ISSN = {0002-9947,1088-6850},
   MRCLASS = {35J60 (34C20 34C37 35B33 35J25 37C99)},
  MRNUMBER = {2309176},
MRREVIEWER = {Luisa\ Moschini},
       DOI = {10.1090/S0002-9947-07-04397-8},
       URL = {https://doi-org.utokyo.idm.oclc.org/10.1090/S0002-9947-07-04397-8},
}

@book {Dup,
    AUTHOR = {Dupaigne, Louis},
     TITLE = {Stable solutions of elliptic partial differential equations},
    SERIES = {Chapman \& Hall/CRC Monographs and Surveys in Pure and Applied
              Mathematics},
    VOLUME = {143},
 PUBLISHER = {Chapman \& Hall/CRC, Boca Raton, FL},
      YEAR = {2011},
     PAGES = {xiv+321},
      ISBN = {978-1-4200-6654-8},
   MRCLASS = {35-02 (35B08 35B32 35B35 35J91 35J92)},
  MRNUMBER = {2779463},
MRREVIEWER = {Jos\'e\ C.\ Sabina de Lis},
       DOI = {10.1201/b10802},
       URL = {https://doi-org.utokyo.idm.oclc.org/10.1201/b10802},
}

@book {GTBook,
    AUTHOR = {Gilbarg, David and Trudinger, Neil S.},
     TITLE = {Elliptic partial differential equations of second order},
    SERIES = {Classics in Mathematics},
      NOTE = {Reprint of the 1998 edition},
 PUBLISHER = {Springer-Verlag, Berlin},
      YEAR = {2001},
     PAGES = {xiv+517},
      ISBN = {3-540-41160-7},
   MRCLASS = {35-02 (35Jxx)},
  MRNUMBER = {1814364},
}

@article {GhGo,
    AUTHOR = {Ghergu, Marius and Goubet, Olivier},
     TITLE = {Singular solutions of elliptic equations with iterated
              exponentials},
   JOURNAL = {J. Geom. Anal.},
  FJOURNAL = {Journal of Geometric Analysis},
    VOLUME = {30},
      YEAR = {2020},
    NUMBER = {2},
     PAGES = {1755--1773},
      ISSN = {1050-6926,1559-002X},
   MRCLASS = {35J91 (35B40 35J61 35J75)},
  MRNUMBER = {4081330},
       DOI = {10.1007/s12220-019-00277-1},
       URL = {https://doi-org.utokyo.idm.oclc.org/10.1007/s12220-019-00277-1},
}

@article {Gidas,
    AUTHOR = {Gidas, B. and Ni, Wei Ming and Nirenberg, L.},
     TITLE = {Symmetry and related properties via the maximum principle},
   JOURNAL = {Comm. Math. Phys.},
  FJOURNAL = {Communications in Mathematical Physics},
    VOLUME = {68},
      YEAR = {1979},
    NUMBER = {3},
     PAGES = {209--243},
      ISSN = {0010-3616,1432-0916},
   MRCLASS = {35J25 (35B50)},
  MRNUMBER = {544879},
MRREVIEWER = {\`E.\ M.\ Saak},
       URL = {http://projecteuclid.org.utokyo.idm.oclc.org/euclid.cmp/1103905359},
}

@article {GW11,
    AUTHOR = {Guo, Zongming and Wei, Juncheng},
     TITLE = {Global solution branch and {M}orse index estimates of a
              semilinear elliptic equation with super-critical exponent},
   JOURNAL = {Trans. Amer. Math. Soc.},
  FJOURNAL = {Transactions of the American Mathematical Society},
    VOLUME = {363},
      YEAR = {2011},
    NUMBER = {9},
     PAGES = {4777--4799},
      ISSN = {0002-9947,1088-6850},
   MRCLASS = {35J91 (35B32 35B33 35J25 58E05)},
  MRNUMBER = {2806691},
MRREVIEWER = {Vicen\c tiu\ D.\ R\u adulescu},
       DOI = {10.1090/S0002-9947-2011-05292-X},
       URL = {https://doi-org.utokyo.idm.oclc.org/10.1090/S0002-9947-2011-05292-X},
}

@article {JL,
    AUTHOR = {Joseph, D. D. and Lundgren, T. S.},
     TITLE = {Quasilinear {D}irichlet problems driven by positive sources},
   JOURNAL = {Arch. Rational Mech. Anal.},
  FJOURNAL = {Archive for Rational Mechanics and Analysis},
    VOLUME = {49},
      YEAR = {1972/73},
     PAGES = {241--269},
      ISSN = {0003-9527},
   MRCLASS = {34B15},
  MRNUMBER = {340701},
MRREVIEWER = {Jean\ Mawhin},
       DOI = {10.1007/BF00250508},
       URL = {https://doi-org.utokyo.idm.oclc.org/10.1007/BF00250508},
}

@article{KK25,
      title={Singular solutions and bifurcation diagram of semilinear elliptic equations with general nonlinearity in two dimensions}, 
      author={Hiroaki Kikuchi and Kenta Kumagai},
      journal={preprint, arXiv:2511.08961},
      year={},
      eprint={2511.08961},
      archivePrefix={arXiv},
      primaryClass={math.AP},
      url={https://arxiv.org/abs/2511.08961}, 
}

@article {KiWe,
    AUTHOR = {Kikuchi, Hiroaki and Wei, Juncheng},
     TITLE = {A bifurcation diagram of solutions to an elliptic equation
              with exponential nonlinearity in higher dimensions},
   JOURNAL = {Proc. Roy. Soc. Edinburgh Sect. A},
  FJOURNAL = {Proceedings of the Royal Society of Edinburgh. Section A.
              Mathematics},
    VOLUME = {148},
      YEAR = {2018},
    NUMBER = {1},
     PAGES = {101--122},
      ISSN = {0308-2105,1473-7124},
   MRCLASS = {35J61 (35B32 35J25)},
  MRNUMBER = {3749337},
MRREVIEWER = {Rosa\ Pardo},
       DOI = {10.1017/S0308210517000154},
       URL = {https://doi-org.utokyo.idm.oclc.org/10.1017/S0308210517000154},
}

@article {Kor,
    AUTHOR = {Korman, Philip},
     TITLE = {Solution curves for semilinear equations on a ball},
   JOURNAL = {Proc. Amer. Math. Soc.},
  FJOURNAL = {Proceedings of the American Mathematical Society},
    VOLUME = {125},
      YEAR = {1997},
    NUMBER = {7},
     PAGES = {1997--2005},
      ISSN = {0002-9939,1088-6826},
   MRCLASS = {35J60},
  MRNUMBER = {1423311},
       DOI = {10.1090/S0002-9939-97-04119-1},
       URL = {https://doi-org.utokyo.idm.oclc.org/10.1090/S0002-9939-97-04119-1},
}

@article {Korman14,
    AUTHOR = {Korman, Philip},
     TITLE = {Global solution curves for self-similar equations},
   JOURNAL = {J. Differential Equations},
  FJOURNAL = {Journal of Differential Equations},
    VOLUME = {257},
      YEAR = {2014},
    NUMBER = {7},
     PAGES = {2543--2564},
      ISSN = {0022-0396,1090-2732},
   MRCLASS = {35J91 (35B07 35B09 35B40 35J25)},
  MRNUMBER = {3228976},
MRREVIEWER = {Alan\ V.\ Lair},
       DOI = {10.1016/j.jde.2014.05.045},
       URL = {https://doi-org.utokyo.idm.oclc.org/10.1016/j.jde.2014.05.045},
}

@article {Lin,
    AUTHOR = {Lin, Song Sun},
     TITLE = {Positive singular solutions for semilinear elliptic equations
              with supercritical growth},
   JOURNAL = {J. Differential Equations},
  FJOURNAL = {Journal of Differential Equations},
    VOLUME = {114},
      YEAR = {1994},
    NUMBER = {1},
     PAGES = {57--76},
      ISSN = {0022-0396,1090-2732},
   MRCLASS = {35J60 (35B05 35B40)},
  MRNUMBER = {1302134},
MRREVIEWER = {Jann-Long\ Chern},
       DOI = {10.1006/jdeq.1994.1140},
       URL = {https://doi-org.utokyo.idm.oclc.org/10.1006/jdeq.1994.1140},
}

@article {MP91,
    AUTHOR = {Merle, F. and Peletier, L. A.},
     TITLE = {Positive solutions of elliptic equations involving
              supercritical growth},
   JOURNAL = {Proc. Roy. Soc. Edinburgh Sect. A},
  FJOURNAL = {Proceedings of the Royal Society of Edinburgh. Section A.
              Mathematics},
    VOLUME = {118},
      YEAR = {1991},
    NUMBER = {1-2},
     PAGES = {49--62},
      ISSN = {0308-2105,1473-7124},
   MRCLASS = {35J65 (35B05 35P30)},
  MRNUMBER = {1113842},
MRREVIEWER = {Philip\ W.\ Schaefer},
       DOI = {10.1017/S0308210500028882},
       URL = {https://doi-org.utokyo.idm.oclc.org/10.1017/S0308210500028882},
}

@article {M14,
    AUTHOR = {Miyamoto, Yasuhito},
     TITLE = {Structure of the positive solutions for supercritical elliptic
              equations in a ball},
   JOURNAL = {J. Math. Pures Appl. (9)},
  FJOURNAL = {Journal de Math\'ematiques Pures et Appliqu\'ees. Neuvi\`eme
              S\'erie},
    VOLUME = {102},
      YEAR = {2014},
    NUMBER = {4},
     PAGES = {672--701},
      ISSN = {0021-7824,1776-3371},
   MRCLASS = {35J91 (35A24 35B09 35B32 35B33 35J20 35J25)},
  MRNUMBER = {3258127},
MRREVIEWER = {Vicen\c tiu\ D.\ R\u adulescu},
       DOI = {10.1016/j.matpur.2014.02.002},
       URL = {https://doi-org.utokyo.idm.oclc.org/10.1016/j.matpur.2014.02.002},
}

@article {M15,
    AUTHOR = {Miyamoto, Yasuhito},
     TITLE = {Classification of bifurcation diagrams for elliptic equations
              with exponential growth in a ball},
   JOURNAL = {Ann. Mat. Pura Appl. (4)},
  FJOURNAL = {Annali di Matematica Pura ed Applicata. Series IV},
    VOLUME = {194},
      YEAR = {2015},
    NUMBER = {4},
     PAGES = {931--952},
      ISSN = {0373-3114,1618-1891},
   MRCLASS = {35J91 (35B32 35B33 35J25)},
  MRNUMBER = {3357688},
MRREVIEWER = {Alessandro\ Maria\ Selvitella},
       DOI = {10.1007/s10231-014-0404-8},
       URL = {https://doi-org.utokyo.idm.oclc.org/10.1007/s10231-014-0404-8},
}

@article {M18,
    AUTHOR = {Miyamoto, Yasuhito},
     TITLE = {A limit equation and bifurcation diagrams of semilinear
              elliptic equations with general supercritical growth},
   JOURNAL = {J. Differential Equations},
  FJOURNAL = {Journal of Differential Equations},
    VOLUME = {264},
      YEAR = {2018},
    NUMBER = {4},
     PAGES = {2684--2707},
      ISSN = {0022-0396,1090-2732},
   MRCLASS = {34C10 (35B06 35B32 35J25 35J61 35J91)},
  MRNUMBER = {3737851},
MRREVIEWER = {Vicen\c tiu\ D.\ R\u adulescu},
       DOI = {10.1016/j.jde.2017.10.034},
       URL = {https://doi-org.utokyo.idm.oclc.org/10.1016/j.jde.2017.10.034},
}

@article {MN20,
    AUTHOR = {Miyamoto, Yasuhito and Naito, Y\=uki},
     TITLE = {Fundamental properties and asymptotic shapes of the singular
              and classical radial solutions for supercritical semilinear
              elliptic equations},
   JOURNAL = {NoDEA Nonlinear Differential Equations Appl.},
  FJOURNAL = {NoDEA. Nonlinear Differential Equations and Applications},
    VOLUME = {27},
      YEAR = {2020},
    NUMBER = {6},
     PAGES = {Paper No. 52, 25},
      ISSN = {1021-9722,1420-9004},
   MRCLASS = {35J61 (35A24)},
  MRNUMBER = {4160935},
       DOI = {10.1007/s00030-020-00658-4},
       URL = {https://doi-org.utokyo.idm.oclc.org/10.1007/s00030-020-00658-4},
}

@article {MN23,
    AUTHOR = {Miyamoto, Yasuhito and Naito, Y\=uki},
     TITLE = {Singular solutions for semilinear elliptic equations with
              general supercritical growth},
   JOURNAL = {Ann. Mat. Pura Appl. (4)},
  FJOURNAL = {Annali di Matematica Pura ed Applicata. Series IV},
    VOLUME = {202},
      YEAR = {2023},
    NUMBER = {1},
     PAGES = {341--366},
      ISSN = {0373-3114,1618-1891},
   MRCLASS = {35J61 (35A24 35B32 35B40)},
  MRNUMBER = {4531724},
       DOI = {10.1007/s10231-022-01244-4},
       URL = {https://doi-org.utokyo.idm.oclc.org/10.1007/s10231-022-01244-4},
}

@article {serrin,
    AUTHOR = {Ni, Wei-Ming and Serrin, James},
     TITLE = {Nonexistence theorems for singular solutions of quasilinear
              partial differential equations},
   JOURNAL = {Comm. Pure Appl. Math.},
  FJOURNAL = {Communications on Pure and Applied Mathematics},
    VOLUME = {39},
      YEAR = {1986},
    NUMBER = {3},
     PAGES = {379--399},
      ISSN = {0010-3640,1097-0312},
   MRCLASS = {35J65},
  MRNUMBER = {829846},
MRREVIEWER = {Maria\ J.\ Esteban},
       DOI = {10.1002/cpa.3160390306},
       URL = {https://doi-org.utokyo.idm.oclc.org/10.1002/cpa.3160390306},
}

@article {KumagaiJDE,
    AUTHOR = {Kumagai, Kenta},
     TITLE = {Classification of bifurcation diagrams for semilinear elliptic
              equations in the critical dimension},
   JOURNAL = {J. Differential Equations},
  FJOURNAL = {Journal of Differential Equations},
    VOLUME = {398},
      YEAR = {2024},
     PAGES = {290--318},
      ISSN = {0022-0396,1090-2732},
   MRCLASS = {35B32 (35B35 35J25 35J61)},
  MRNUMBER = {4730299},
       DOI = {10.1016/j.jde.2024.03.026},
       URL = {https://doi-org.utokyo.idm.oclc.org/10.1016/j.jde.2024.03.026},
}

@article {LLD,
    AUTHOR = {Liu, Yi and Li, Yi and Deng, Yinbin},
     TITLE = {Separation property of solutions for a semilinear elliptic
              equation},
   JOURNAL = {J. Differential Equations},
  FJOURNAL = {Journal of Differential Equations},
    VOLUME = {163},
      YEAR = {2000},
    NUMBER = {2},
     PAGES = {381--406},
      ISSN = {0022-0396,1090-2732},
   MRCLASS = {35J60 (35B05 35B40)},
  MRNUMBER = {1758703},
MRREVIEWER = {Alan\ V.\ Lair},
       DOI = {10.1006/jdeq.1999.3735},
       URL = {https://doi-org.utokyo.idm.oclc.org/10.1006/jdeq.1999.3735},
}

@article {Wan,
    AUTHOR = {Wang, Xuefeng},
     TITLE = {On the {C}auchy problem for reaction-diffusion equations},
   JOURNAL = {Trans. Amer. Math. Soc.},
  FJOURNAL = {Transactions of the American Mathematical Society},
    VOLUME = {337},
      YEAR = {1993},
    NUMBER = {2},
     PAGES = {549--590},
      ISSN = {0002-9947,1088-6850},
   MRCLASS = {35K57 (35B40)},
  MRNUMBER = {1153016},
MRREVIEWER = {Reinhard\ Redlinger},
       DOI = {10.2307/2154232},
       URL = {https://doi-org.utokyo.idm.oclc.org/10.2307/2154232},
}

@article {MN24,
    AUTHOR = {Miyamoto, Yasuhito and Naito, Y\=uki},
     TITLE = {A bifurcation diagram of solutions to semilinear elliptic
              equations with general supercritical growth},
   JOURNAL = {J. Differential Equations},
  FJOURNAL = {Journal of Differential Equations},
    VOLUME = {406},
      YEAR = {2024},
     PAGES = {318--337},
      ISSN = {0022-0396,1090-2732},
   MRCLASS = {35J61 (35A24 35B32 35B33)},
  MRNUMBER = {4765652},
MRREVIEWER = {Yuanze\ Wu},
       DOI = {10.1016/j.jde.2024.06.026},
       URL = {https://doi-org.utokyo.idm.oclc.org/10.1016/j.jde.2024.06.026},
}

@article{K25,
  author       = {Kumagai, Kenta},
  title        = {Classification of bifurcation structure for semilinear elliptic equations in a ball},
  journal      = {Calc. Var. Partial Differential Equations},
  volume       = {65},
  year         = {2026},
  articlenumber = {240},
   DOI = {10.1007/s00526-026-03405-2},
       URL = {https://doi.org/10.1007/s00526-026-03405-2},
}

@article {GV97,
    AUTHOR = {Galaktionov, Victor A. and Vazquez, Juan L.},
     TITLE = {Continuation of blowup solutions of nonlinear heat equations
              in several space dimensions},
   JOURNAL = {Comm. Pure Appl. Math.},
  FJOURNAL = {Communications on Pure and Applied Mathematics},
    VOLUME = {50},
      YEAR = {1997},
    NUMBER = {1},
     PAGES = {1--67},
      ISSN = {0010-3640,1097-0312},
   MRCLASS = {35K55 (35B05 35B60)},
  MRNUMBER = {1423231},
MRREVIEWER = {Dian\ K.\ Palagachev},
       DOI = {10.1002/(SICI)1097-0312(199701)50:1<1::AID-CPA1>3.3.CO;2-R},
       URL =
              {https://doi-org.utokyo.idm.oclc.org/10.1002/(SICI)1097-0312(199701)50:1<1::AID-CPA1>3.3.CO;2-R},
}

@article {PY03,
    AUTHOR = {Pol\'{a}\v{c}ik, Peter and Yanagida, Eiji},
     TITLE = {On bounded and unbounded global solutions of a supercritical
              semilinear heat equation},
   JOURNAL = {Math. Ann.},
  FJOURNAL = {Mathematische Annalen},
    VOLUME = {327},
      YEAR = {2003},
    NUMBER = {4},
     PAGES = {745--771},
      ISSN = {0025-5831,1432-1807},
   MRCLASS = {35K55 (35K15)},
  MRNUMBER = {2023315},
MRREVIEWER = {Chang\ Hao\ Lin},
       DOI = {10.1007/s00208-003-0469-y},
       URL = {https://doi-org.utokyo.idm.oclc.org/10.1007/s00208-003-0469-y},
}

@article {Mizo05,
    AUTHOR = {Mizoguchi, Noriko},
     TITLE = {Boundedness of global solutions for a supercritical semilinear
              heat equation and its application},
   JOURNAL = {Indiana Univ. Math. J.},
  FJOURNAL = {Indiana University Mathematics Journal},
    VOLUME = {54},
      YEAR = {2005},
    NUMBER = {4},
     PAGES = {1047--1059},
      ISSN = {0022-2518,1943-5258},
   MRCLASS = {35K55 (35B45)},
  MRNUMBER = {2164417},
MRREVIEWER = {Ya\ Zhe\ Chen},
       DOI = {10.1512/iumj.2005.54.2694},
       URL = {https://doi-org.utokyo.idm.oclc.org/10.1512/iumj.2005.54.2694},
}

@article {FWY,
    AUTHOR = {Fila, Marek and Winkler, Michael and Yanagida, Eiji},
     TITLE = {Grow-up rate of solutions for a supercritical semilinear
              diffusion equation},
   JOURNAL = {J. Differential Equations},
  FJOURNAL = {Journal of Differential Equations},
    VOLUME = {205},
      YEAR = {2004},
    NUMBER = {2},
     PAGES = {365--389},
      ISSN = {0022-0396,1090-2732},
   MRCLASS = {35K57 (35B45)},
  MRNUMBER = {2092863},
MRREVIEWER = {Song\ Mu\ Zheng},
       DOI = {10.1016/j.jde.2004.03.009},
       URL = {https://doi-org.utokyo.idm.oclc.org/10.1016/j.jde.2004.03.009},
}

@article {QS25,
    AUTHOR = {Quittner, Pavol and Souplet, Philippe},
     TITLE = {Threshold, subthreshold, and global unbounded solutions of
              superlinear heat equations},
   JOURNAL = {Proc. Lond. Math. Soc. (3)},
  FJOURNAL = {Proceedings of the London Mathematical Society. Third Series},
    VOLUME = {131},
      YEAR = {2025},
    NUMBER = {6},
     PAGES = {Paper No. e70110, 27},
      ISSN = {0024-6115,1460-244X},
   MRCLASS = {35K58 (35B40 35B44 35K57)},
  MRNUMBER = {5002468},
MRREVIEWER = {Gao-Feng\ Zheng},
       DOI = {10.1112/plms.70110},
       URL = {https://doi-org.utokyo.idm.oclc.org/10.1112/plms.70110},
}

@article {Martel98,
    AUTHOR = {Martel, Yvan},
     TITLE = {Complete blow up and global behaviour of solutions of
              {$u_t-\Delta u=g(u)$}},
   JOURNAL = {Ann. Inst. H. Poincar\'e{} C Anal. Non Lin\'eaire},
  FJOURNAL = {Annales de l'Institut Henri Poincar\'e{} C. Analyse Non
              Lin\'eaire},
    VOLUME = {15},
      YEAR = {1998},
    NUMBER = {6},
     PAGES = {687--723},
      ISSN = {0294-1449,1873-1430},
   MRCLASS = {35J65 (35B40)},
  MRNUMBER = {1650970},
MRREVIEWER = {Qing\ Fang},
       DOI = {10.1016/S0294-1449(99)80002-X},
       URL = {https://doi-org.utokyo.idm.oclc.org/10.1016/S0294-1449(99)80002-X},
}

@article {V99,
    AUTHOR = {Vazquez, Juan Luis},
     TITLE = {Domain of existence and blowup for the exponential
              reaction-diffusion equation},
   JOURNAL = {Indiana Univ. Math. J.},
  FJOURNAL = {Indiana University Mathematics Journal},
    VOLUME = {48},
      YEAR = {1999},
    NUMBER = {2},
     PAGES = {677--709},
      ISSN = {0022-2518,1943-5258},
   MRCLASS = {35K57 (34A12 35B05)},
  MRNUMBER = {1722813},
MRREVIEWER = {L.\ Hsiao},
       DOI = {10.1512/iumj.1999.48.1581},
       URL = {https://doi-org.utokyo.idm.oclc.org/10.1512/iumj.1999.48.1581},
}

@article {FKWY07-ADE,
    AUTHOR = {Fila, Marek and King, John R. and Winkler, Michael and
              Yanagida, Eiji},
     TITLE = {Grow-up rate of solutions of a semilinear parabolic equation
              with a critical exponent},
   JOURNAL = {Adv. Differential Equations},
  FJOURNAL = {Advances in Differential Equations},
    VOLUME = {12},
      YEAR = {2007},
    NUMBER = {1},
     PAGES = {1--26},
      ISSN = {1079-9389},
   MRCLASS = {35K55 (35B33 35B40 35K15)},
  MRNUMBER = {2272819},
MRREVIEWER = {Andrey\ B.\ Muravnik},
}

@article {Mizo06,
    AUTHOR = {Mizoguchi, Noriko},
     TITLE = {Growup of solutions for a semilinear heat equation with
              supercritical nonlinearity},
   JOURNAL = {J. Differential Equations},
  FJOURNAL = {Journal of Differential Equations},
    VOLUME = {227},
      YEAR = {2006},
    NUMBER = {2},
     PAGES = {652--669},
      ISSN = {0022-0396,1090-2732},
   MRCLASS = {35K55 (35K15)},
  MRNUMBER = {2237683},
MRREVIEWER = {Lidia\ Sk\'ora},
       DOI = {10.1016/j.jde.2005.11.002},
       URL = {https://doi-org.utokyo.idm.oclc.org/10.1016/j.jde.2005.11.002},
}

@article {FKWY07,
    AUTHOR = {Fila, Marek and King, John R. and Winkler, Michael and
              Yanagida, Eiji},
     TITLE = {Grow-up rate of solutions of a semilinear parabolic equation
              with a critical exponent},
   JOURNAL = {Adv. Differential Equations},
  FJOURNAL = {Advances in Differential Equations},
    VOLUME = {12},
      YEAR = {2007},
    NUMBER = {1},
     PAGES = {1--26},
      ISSN = {1079-9389},
   MRCLASS = {35K55 (35B33 35B40 35K15)},
  MRNUMBER = {2272819},
MRREVIEWER = {Andrey\ B.\ Muravnik},
}

@article {FKWY06,
    AUTHOR = {Fila, Marek and King, John R. and Winkler, Michael and
              Yanagida, Eiji},
     TITLE = {Optimal lower bound of the grow-up rate for a supercritical
              parabolic equation},
   JOURNAL = {J. Differential Equations},
  FJOURNAL = {Journal of Differential Equations},
    VOLUME = {228},
      YEAR = {2006},
    NUMBER = {1},
     PAGES = {339--356},
      ISSN = {0022-0396,1090-2732},
   MRCLASS = {35K55 (35B33 35B40)},
  MRNUMBER = {2254434},
MRREVIEWER = {Jan\ W.\ Cholewa},
       DOI = {10.1016/j.jde.2006.01.019},
       URL = {https://doi-org.utokyo.idm.oclc.org/10.1016/j.jde.2006.01.019},
}

@article {GK,
    AUTHOR = {Galaktionov, Victor A. and King, John R.},
     TITLE = {Stabilization to a singular steady state for the
              {F}rank-{K}amenetskii equation in a critical dimension},
   JOURNAL = {Proc. Roy. Soc. Edinburgh Sect. A},
  FJOURNAL = {Proceedings of the Royal Society of Edinburgh. Section A.
              Mathematics},
    VOLUME = {135},
      YEAR = {2005},
    NUMBER = {4},
     PAGES = {777--787},
      ISSN = {0308-2105,1473-7124},
   MRCLASS = {35K55 (35B35)},
  MRNUMBER = {2173339},
MRREVIEWER = {Jana\ Kopfova},
       DOI = {10.1017/S030821050000411X},
       URL = {https://doi-org.utokyo.idm.oclc.org/10.1017/S030821050000411X},
}

@book {Korbook,
    AUTHOR = {Korman, Philip},
     TITLE = {Global solution curves for semilinear elliptic equations},
 PUBLISHER = {World Scientific Publishing Co. Pte. Ltd., Hackensack, NJ},
      YEAR = {2012},
     PAGES = {xii+241},
      ISBN = {978-981-4374-34-7; 981-4374-34-2},
   MRCLASS = {35-02 (35B32 35J60 35J65 35J66)},
  MRNUMBER = {2954053},
MRREVIEWER = {Vicen\c tiu\ D.\ R\u adulescu},
       DOI = {10.1142/8308},
       URL = {https://doi-org.utokyo.idm.oclc.org/10.1142/8308},
}

@article{KM26,
      title={Monotonicity of the bifurcation curve for supercritical elliptic problems in the borderline dimension {$N=10$}}, 
      author={Kenta Kumagai and Yasuhito Miyamoto},
    journal={preprint, arXiv:2605.30946},
      year={},
      eprint={2605.30946},
      archivePrefix={arXiv},
      primaryClass={math.AP},
      url={https://arxiv.org/abs/2605.30946}, 
}

@article {I11,
    AUTHOR = {Ioku, Norisuke},
     TITLE = {The {C}auchy problem for heat equations with exponential
              nonlinearity},
   JOURNAL = {J. Differential Equations},
  FJOURNAL = {Journal of Differential Equations},
    VOLUME = {251},
      YEAR = {2011},
    NUMBER = {4-5},
     PAGES = {1172--1194},
      ISSN = {0022-0396,1090-2732},
   MRCLASS = {35K91 (35A01 35A23 35K15 46E30)},
  MRNUMBER = {2812586},
MRREVIEWER = {Chunshan\ Zhao},
       DOI = {10.1016/j.jde.2011.02.015},
       URL = {https://doi-org.utokyo.idm.oclc.org/10.1016/j.jde.2011.02.015},
}

@article {Tel,
    AUTHOR = {Tello, J. Ignacio},
     TITLE = {Stability of steady states of the {C}auchy problem for the
              exponential reaction-diffusion equation},
   JOURNAL = {J. Math. Anal. Appl.},
  FJOURNAL = {Journal of Mathematical Analysis and Applications},
    VOLUME = {324},
      YEAR = {2006},
    NUMBER = {1},
     PAGES = {381--396},
      ISSN = {0022-247X,1096-0813},
   MRCLASS = {35K57 (35B35 35K15)},
  MRNUMBER = {2262478},
MRREVIEWER = {Xinhua\ Ji},
       DOI = {10.1016/j.jmaa.2005.12.011},
       URL = {https://doi-org.utokyo.idm.oclc.org/10.1016/j.jmaa.2005.12.011},
}

@article {LT1987,
    AUTHOR = {Lacey, A. A. and Tzanetis, D.},
     TITLE = {Global existence and convergence to a singular steady state
              for a semilinear heat equation},
   JOURNAL = {Proc. Roy. Soc. Edinburgh Sect. A},
  FJOURNAL = {Proceedings of the Royal Society of Edinburgh. Section A.
              Mathematics},
    VOLUME = {105},
      YEAR = {1987},
     PAGES = {289--305},
      ISSN = {0308-2105,1473-7124},
   MRCLASS = {35K55 (35B40)},
  MRNUMBER = {890063},
MRREVIEWER = {R\"udiger\ Landes},
       DOI = {10.1017/S0308210500022113},
       URL = {https://doi-org.utokyo.idm.oclc.org/10.1017/S0308210500022113},
}

@article{KM25,
  author  = {Sho Katayama and Yasuhito Miyamoto},
  title   = {Infinite multiplicity of positive solutions of an inhomogeneous supercritical elliptic equation on {$\mathbb{R}^N$}},
  journal = {Ann. Mat. Pura Appl. (4)},
  volume  = {205},
  year    = {2026},
  number  = {3},
  pages   = {1079--1121},
  doi     = {10.1007/s10231-025-01633-5},
}

@article {HM25,
    AUTHOR = {Hisa, Kotaro and Miyamoto, Yasuhito},
     TITLE = {Threshold property of a singular stationary solution for
              semilinear heat equations with exponential growth},
   JOURNAL = {Manuscripta Math.},
  FJOURNAL = {Manuscripta Mathematica},
    VOLUME = {176},
      YEAR = {2025},
    NUMBER = {5},
     PAGES = {Paper No. 61, 30},
      ISSN = {0025-2611,1432-1785},
   MRCLASS = {35K58 (35A01 35A21 35B44 35K15)},
  MRNUMBER = {4947254},
       DOI = {10.1007/s00229-025-01661-8},
       URL = {https://doi-org.utokyo.idm.oclc.org/10.1007/s00229-025-01661-8},
}

@book {Lieberman,
    AUTHOR = {Lieberman, Gary M.},
     TITLE = {Second order parabolic differential equations},
 PUBLISHER = {World Scientific Publishing Co., Inc., River Edge, NJ},
      YEAR = {1996},
     PAGES = {xii+439},
      ISBN = {981-02-2883-X},
   MRCLASS = {35-02 (35Bxx 35Dxx 35Kxx)},
  MRNUMBER = {1465184},
MRREVIEWER = {Siegfried\ Carl},
       DOI = {10.1142/3302},
       URL = {https://doi-org.utokyo.idm.oclc.org/10.1142/3302},
}

@book {QSbook,
    AUTHOR = {Quittner, Pavol and Souplet, Philippe},
     TITLE = {Superlinear parabolic problems},
    SERIES = {Birkh\"auser Advanced Texts: Basler Lehrb\"ucher.
              [Birkh\"auser Advanced Texts: Basel Textbooks]},
   EDITION = {Second},
      NOTE = {Blow-up, global existence and steady states},
 PUBLISHER = {Birkh\"auser/Springer, Cham},
      YEAR = {2019},
     PAGES = {xvi+725},
      ISBN = {978-3-030-18220-5; 978-3-030-18222-9},
   MRCLASS = {35-02 (35B44 35J57 35J60 35K51 35K55)},
  MRNUMBER = {3967048},
       DOI = {10.1007/978-3-030-18222-9},
       URL = {https://doi-org.utokyo.idm.oclc.org/10.1007/978-3-030-18222-9},
}

@article {Fu1969,
    AUTHOR = {Fujita, Hiroshi},
     TITLE = {On the nonlinear equations {$\Delta u+e\sp{u}=0$} and
              {$\partial v/\partial t=\Delta v+e \sp{v}$}},
   JOURNAL = {Bull. Amer. Math. Soc.},
  FJOURNAL = {Bulletin of the American Mathematical Society},
    VOLUME = {75},
      YEAR = {1969},
     PAGES = {132--135},
      ISSN = {0002-9904},
   MRCLASS = {35.36},
  MRNUMBER = {239258},
MRREVIEWER = {P.\ Cooperman},
       DOI = {10.1090/S0002-9904-1969-12175-0},
       URL = {https://doi-org.utokyo.idm.oclc.org/10.1090/S0002-9904-1969-12175-0},
}

@article {CZ2022,
    AUTHOR = {Chang, Caihong and Zhang, Zhengce},
     TITLE = {Asymptotic behavior of blowup solutions for {H}\'enon type
              parabolic equations with exponential nonlinearity},
   JOURNAL = {Electron. J. Differential Equations},
  FJOURNAL = {Electronic Journal of Differential Equations},
      YEAR = {2022},
     PAGES = {Paper No. 42, 19},
      ISSN = {1072-6691},
   MRCLASS = {35K20 (35B40 35B44)},
  MRNUMBER = {4445281},
MRREVIEWER = {Weiwei\ Ding},
       DOI = {10.58997/ejde.2022.42},
       URL = {https://doi-org.utokyo.idm.oclc.org/10.58997/ejde.2022.42},
}

@article {VZ,
    AUTHOR = {Vazquez, Juan Luis and Zuazua, Enrike},
     TITLE = {The {H}ardy inequality and the asymptotic behaviour of the
              heat equation with an inverse-square potential},
   JOURNAL = {J. Funct. Anal.},
  FJOURNAL = {Journal of Functional Analysis},
    VOLUME = {173},
      YEAR = {2000},
    NUMBER = {1},
     PAGES = {103--153},
      ISSN = {0022-1236,1096-0783},
   MRCLASS = {35K05 (35B40 35D05)},
  MRNUMBER = {1760280},
MRREVIEWER = {Snoussi\ Seifeddine},
       DOI = {10.1006/jfan.1999.3556},
       URL = {https://doi-org.utokyo.idm.oclc.org/10.1006/jfan.1999.3556},
}

@article {DGLV1998,
    AUTHOR = {Dold, J. W. and Galaktionov, V. A. and Lacey, A. A. and
              V\'azquez, J. L.},
     TITLE = {Rate of approach to a singular steady state in quasilinear
              reaction-diffusion equations},
   JOURNAL = {Ann. Scuola Norm. Sup. Pisa Cl. Sci. (4)},
  FJOURNAL = {Annali della Scuola Normale Superiore di Pisa. Classe di
              Scienze. Serie IV},
    VOLUME = {26},
      YEAR = {1998},
    NUMBER = {4},
     PAGES = {663--687},
      ISSN = {0391-173X,2036-2145},
   MRCLASS = {35K57 (35K65)},
  MRNUMBER = {1648562},
MRREVIEWER = {Andreas\ Unterreiter},
       URL = {http://www.numdam.org/item?id=ASNSP_1998_4_26_4_663_0},
}

@article {FK25,
    AUTHOR = {Fujishima, Yohei and Kan, Toru},
     TITLE = {Uniform boundedness and blow-up rate of solutions in
              non-scale-invariant superlinear heat equations},
   JOURNAL = {J. Elliptic Parabol. Equ.},
  FJOURNAL = {Journal of Elliptic and Parabolic Equations},
    VOLUME = {11},
      YEAR = {2025},
    NUMBER = {3},
     PAGES = {2185--2217},
      ISSN = {2296-9020,2296-9039},
   MRCLASS = {35K58 (35B44 35B45)},
  MRNUMBER = {5000894},
MRREVIEWER = {Hongwei\ Chen},
       DOI = {10.1007/s41808-025-00335-6},
       URL = {https://doi-org.utokyo.idm.oclc.org/10.1007/s41808-025-00335-6},
}

@article {PV1995,
    AUTHOR = {Peral, I. and V\'azquez, J. L.},
     TITLE = {On the stability or instability of the singular solution of
              the semilinear heat equation with exponential reaction term},
   JOURNAL = {Arch. Rational Mech. Anal.},
  FJOURNAL = {Archive for Rational Mechanics and Analysis},
    VOLUME = {129},
      YEAR = {1995},
    NUMBER = {3},
     PAGES = {201--224},
      ISSN = {0003-9527},
   MRCLASS = {35B35 (35K55)},
  MRNUMBER = {1328476},
MRREVIEWER = {Qing\ Fang},
       DOI = {10.1007/BF00383673},
       URL = {https://doi-org.utokyo.idm.oclc.org/10.1007/BF00383673},
}
}
\end{document}